\documentclass[a4paper,10.5pt,reqno]{amsart}
\usepackage[dvips]{graphicx}
\usepackage{amsmath,amssymb,amsthm,amsfonts,mathrsfs,amsopn}
\usepackage{dsfont}
\usepackage{mathtools}
\mathtoolsset{showonlyrefs}
\usepackage{amssymb}
\usepackage{color}

\usepackage[colorlinks=true,urlcolor=blue, citecolor=red,linkcolor=blue,
linktocpage,pdfpagelabels,bookmarksnumbered,bookmarksopen]{hyperref}
\usepackage{latexsym}
\usepackage[english]{babel}
\usepackage{tikz}
\usetikzlibrary{matrix,arrows,decorations.pathmorphing}

\usepackage{tikz-cd}

\usepackage{geometry}
\newcommand{\R}{\mathbb{R}}

\newcommand{\N}{\mathbb{N}}
\newcommand{\C}{\mathbb{C}}

\newcommand{\B}{\mathbb{B}}

\newcommand{\1}{\bf{1}}

\newcommand{\vol}{\rm vol}

\newcommand{\Sph}{\mathbb{S}}
\newcommand{\dd}{\mathop{}\!\mathrm{d}}

\newcommand{\be}{\begin{equation}}%
\newcommand{\ee}{\end{equation}}

\DeclareMathOperator{\G}{G} 

\newtheorem{theorem}{Theorem}[section]

\newtheorem{corollary}[theorem]{Corollary}
\newtheorem{definition}[theorem]{Definition}

\newtheorem{lemma}[theorem]{Lemma}

\newtheorem{proposition}[theorem]{Proposition}
\newtheorem{remark}[theorem]{Remark}

\begin{document}
\title{Conformal Dirac-Einstein equations on manifolds with boundary.}
\author{William Borrelli$^{(1)}$, Ali Maalaoui$^{(2)}$ \& Vittorio Martino$^{(3)}$}
\addtocounter{footnote}{1}
\footnotetext{Dipartimento di Matematica e Fisica, Universit\`a Cattolica del Sacro Cuore, Via della Garzetta 48, I-25133, Brescia, Italy. E-mail address:
{\tt{william.borrelli@unicatt.it}}}
\addtocounter{footnote}{1}
\footnotetext{Department of Mathematics, Clark University, 950 Main Street, Worcester, MA 01610, USA. E-mail address:
{\tt{amaalaoui@clarku.edu}}}
\addtocounter{footnote}{1}
\footnotetext{Dipartimento di Matematica, Universit\`a di Bologna, Piazza di Porta S.Donato 5, 40126 Bologna,
Italy. E-mail address:
{\tt{vittorio.martino3@unibo.it}}}

\date{}
\maketitle

\vspace{5mm}

{\noindent\bf Abstract} { In this paper we study Dirac-Einstein equations on manifolds with boundary, restricted to a conformal class with constant boundary volume, under chiral bag boundary conditions for the Dirac operator. We characterize the bubbling phenomenon, also classifying ground state bubbles. Finally, we prove an Aubin-type inequality and a related existence result.}

\vspace{5mm}

\vspace{5mm}


\vspace{5mm}


\section{Introduction and main results}
The Einstein-Dirac equations describe the interaction of spin $\frac{1}{2}$ particles with a gravitational field. This interaction is modeled through a coupling of the Einstein-Hilbert functional (the total curvature) and a fermionic interaction using the Dirac operator. Given a three dimensional compact spin manifold $M$, the functional reads as
$$\mathcal{ED}(g,\psi)=\int_{M}R_{g}\ dv_{g}+\int_{M} \langle D_{g}\psi, \psi\rangle -|\psi|_{g}^{2}\ dv_{g},$$
where $g$ is a Riemannian metric on $M$ and $\psi$ is a spinor. We recall that the study of the first term of this functional was heavily investigated and this led to beautiful results in geometric analysis such as the solution of the Yamabe problem, when restricted to a conformal class of the metric, and the positive mass theorem \cite{SY, Wit}.\\
Going back to the Einstein-Hilbert functional, the natural extension was then to consider manifolds with non-empty boundaries. The functional is then modified by adding a term involving the total mean curvature of the boundary and this is known as the Gibbons-Hawking-York action (see, e.g. \cite{GH, Y}). Again, if one restricts the variations to a conformal class of the metric, one is led to the Yamabe problem on manifolds with boundary, first investigated by Escobar \cite{escobar}. This led to the study of the positive mass theorem on manifolds with boundary \cite{ABD, Hir} and Dirac type operators on such manifolds \cite{BC}. Notice that this brings extra challenges to decide which boundary condition one should impose on spinors that would guarantee ellipticity and allows achieving the same geometric properties. We refer to \cite{HMR} and the references therein for the study of several boundary conditions for the Dirac operator and their ellipticity.

In this paper, we propose to study the conformal version of the coupling of the Gibbons-Hawking-York functional with the fermionic interaction while imposing a natural local boundary condition. Indeed, let $(M,g)$ be a compact oriented three-dimensional Riemannian manifold with boundary. It is well-known that it admits a spin structure $\sigma$, see for instance \cite[page 87]{LawsonMichelsohn}. We assume that the manifold $M$ admits a \emph{chirality operator} $\G$ on the spinor bundle $\Sigma M$. Examples of such 3-dimensional manifolds are given by space-like hypersurfaces in a Lorentzian manifold, where $\G$ is given by the Clifford multiplication  by a unit time-like normal field to $M$, see e.g. \cite{HMR} and references therein. The existence of a chirality operator allows to define a class of (local) boundary conditions for the Dirac operator that behave nicely under a conformal change of metric (see Section \ref{sec:prop}).

Let us consider the energy functional $E^{b}$
\be\label{eq:energy}
E^{b}(u,\psi)=\frac{1}{2}\Big(\int_{M} \vert\nabla u\vert^2+\frac{R_g}{8}u^{2}\ dv_g+\int_{\partial M}\frac{1}{2}h_{g}u^{2}\ d\sigma_{g}+\int_{M}\langle D_g\psi,\psi\rangle - |\psi|^{2}|u|^{2}\ dv_g\Big)-\frac{b}{4}\int_{\partial M}u^{4}\ d\sigma_{g}\,,
\ee
where  $h_g$ is the mean curvature of $\partial M$, $b$ is a non-negative constant, $D_g$, $\Sigma M$ are the Dirac operator and the spinor bundle associated with $\sigma$, respectively, and $\langle\cdot,\cdot\rangle$ is the real part of the compatible Hermitian metric on $\Sigma M$. Notice that $E^{b}$ is defined on the space $\mathcal{H}=H^{1}(M)\times H^{\frac{1}{2}}_{+}(\Sigma M)$ where  $H^{\frac{1}{2}}_{+}(\Sigma M)$ is the space of $H^{\frac{1}{2}}$-sections of the spinor bundle satisfying chiral bag boundary conditions (see Section \ref{sec:prop}).
The energy functional $E^{b}$ arises from the following generalization of the Gibbons-Hawking-York functional, which in its turn generalizes the Hilbert-Einstein functional for manifolds with boundary
\be\label{eq:DEgeneral}
\mathcal{E}(g,\psi)=\int_{M}R_{g}dv_{g}+\frac{1}{2}\int_{\partial M}h_{g}\ d\sigma_{g}+\int_{M}\langle D_{g}\psi,\psi \rangle-\langle \psi,\psi\rangle\, dv_g\ ,
\ee
where $g$ is a Riemannian metric on $M$ and $\psi$ is a spinor field. As mentioned above, this model was investigated (for manifolds without boundary) in full generality by \cite{Belg, Fin, Kim} for the properties of the critical points of the functional and examples of Dirac-Einstein structures, while in \cite{MV0} the authors study the compactness property of the set Einstein-Dirac structures. The functional $E^{b}$ is obtained from $\mathcal E$ by restricting the latter to a conformal class of the metric with constant boundary volume ($Vol_{g}(\partial M)=1$). The conformal version of the functional $\mathcal{E}$ has been studied by the second and third named authors in \cite{MV3} in the case of closed oriented Riemannian 3-manifolds. Observe that the functional \eqref{eq:energy} can be considered as the three-dimensional analogue of the {\em Super-Liouville functional}, for which a blow-up analysis and the existence of critical points has been studied in the literature, see e.g. \cite{JWZ,JMW1,JMW2} (for the case $\partial M=\emptyset$) and \cite{FK, J1,J2,J3} (for the case $\partial M \not= 0$).

\medskip

The aim of the present paper is to study various properties of $E^{b}$ and of its critical points, considering a local boundary conditions for the Dirac operator $D_g$.

\medskip

The critical points of $E^{b}$ solve the following system:
\begin{equation}\label{eq1}
\left\{\begin{array}{ll}
L_{g}u=u|\psi|^{2}\\
&\qquad\mbox{on $M$}\\
D\psi=u^{2}\psi \\
&\\
B_{g}u:=\frac{\partial u}{\partial \nu}+\frac{1}{2}h_{g}u=bu^{3}\\
& \mbox{on $\partial M$}\\
\mathbb{B}^{+}\psi=0 \\
\end{array}
\right.
\end{equation}
Here $\mathbb{B}^{+}$ is the operator defining chiral boundary conditions for spinors (see Section \ref{sec:prop}).

We define the Yamabe invariant for manifolds with boundary $Y(M,\partial M,[g])$ by
\be\label{eq:yamabeinv}
Y(M,\partial M, [g])=\inf_{u\in H^{1}(M),u>0} \frac{\int_{M}|\nabla u|^{2}+\frac{1}{8}R_gu^{2}\ dv_{g}+\int_{\partial M}\frac{1}{2}h_{g}u^{2} \ d\sigma_{g}}{\Big(\int_{M}u^{6} \ dv_{g}\Big)^{\frac{1}{3}}}\,,
\ee
where $R_g$ is the scalar curvature of the metric $g$.

It is also important to notice that the operator $(L_{g},B_{g})$ is elliptic with compact resolvent. So under the boundary condition $B_{g}u=0$, $L_{g}$ is self-adjoint and has a discrete spectrum. The first eigenvalue of this operator is defined by
$$\lambda_{1}=\inf_{u\in H^{1}(M)\setminus\{0\}}\frac{\int_{M}|\nabla u|^{2}+\frac{1}{8}R_gu^{2}\ dv_{g}+\frac{1}{2}\int_{\partial M}h_gu^{2} \ d\sigma_{g}}{\int_{M}u^{2} \ dv_{g}}.$$
The sign of $\lambda_{1}$ is a conformal invariant and it has the same sign as $Y(M,\partial M,[g])$. From now on, we will focus on the case $\lambda_{1}>0$ and by using the first eigenfunction as a conformal change, we can assume without loss of generality that $h_g=0$ and $R_g>0$, as it was shown in Proposition 1.3 and Lemma 1.1 in \cite{escobar}. The quantity $\|u\|^{2}=\int_{M} |\nabla u|^{2}+\frac{1}{8}R_g u^{2}\ dv_{g}$, defines then a norm on $H^{1}(M)$, equivalent to the standard one.
Our new functional reads
\be\label{eq:newenergy}
E^{b}(u,\psi)=\frac{1}{2}\Big(\int_{M}|\nabla u|^{2}+\frac{1}{8}R_gu^{2}\ dv_{g}+\int_{M}\langle D\psi,\psi\rangle - |\psi|^{2}|u|^{2}\ dv_{g}\Big)-\frac{b}{4}\int_{\partial M}u^{4}\ d\sigma_{g} \,.
\ee

Looking for solutions to \eqref{eq1} via variational methods, bubbling phenomena may occur, due to the conformal invariance of the functional, and one is led to study the limit equations describing the blow-up profiles. We will distinguish two kinds of blow-up:
\begin{itemize}
\item Interior blow-up, when the energy concentration occurs in the interior of the manifold. This phenomena is similar to the case of manifolds with no boundary that was studied in \cite{MV3} and hence the limit equations are posed on the whole space $\R^{3}$  and read
\be\label{eq:blowupint}
\left\{\begin{array}{ll}
-\Delta u=u|\psi|^{2}\\
& \qquad \mbox{on $\R^3$.}\\
D\psi=u^{2}\psi \\
\end{array}
\right.
\ee

\item Boundary blow-up, when the energy concentration occurs at the boundary $\partial M$. the limit equations are posed on the half-space $\R^3_+$ and are given by
\be\label{eq:blowup}
\left\{\begin{array}{ll}
-\Delta u=u|\psi|^{2}\\
& \qquad \mbox{on $\R^3_+$}\\
D\psi=u^{2}\psi\\
\frac{\partial u}{\partial \nu}=bu^3\\
&\qquad \mbox{on $\partial\R^3_+$}\\
\mathbb{B}^{+}\psi=0
\end{array}
\right.
\ee

\end{itemize}

We recall that solutions of $(\ref{eq:blowupint})$ have been studied by the first two named authors in \cite{BM} where a classification result was exhibited for ground state solutions. We also distinguish the two limiting functionals. Namely, weak solutions for $(\ref{eq:blowupint})$ are critical points of the energy functional
\be\label{eq:bubbleenergy}
E_{\R^{3}}(u,\psi)=\frac{1}{2}\left( \int_{\R^3}\vert\nabla u\vert^2+\langle D\psi,\psi\rangle - u^2\vert \psi\vert^2\, dv_{g_{\R^{3}}}\right)\,,
\ee
on $\mathring{H}^1(\R^3)\times\mathring{H}^{1/2}(\Sigma_{g_{0}}\R^3)$.
On the other hand, weak solutions $(u,\psi)\in \mathring{H}^1(\R^3_+)\times \mathring{H}^{1/2}_{+}(\Sigma_{g_0}\R^3_+)$ to \eqref{eq:blowup} correspond to critical points of the following functional
\be\label{eq:bubbleaction}
E_{\R^3_+}^{b}(u,\psi)=\frac{1}{2}\left( \int_{\R^3_+}\vert\nabla u\vert^2+\langle D\psi,\psi\rangle - u^2\vert \psi\vert^2\, dv_{g_{\R^{3}}}\right)-\frac{b}{4}\int_{\partial \R^{3}_{+}}u^{4} d\sigma_{g_{\R^{3}}} \,.
\ee

In our investigation, we start by studying the asymptotic decomposition and the energy quantization of Palais-Smale sequences of the functional $E^{b}$, as described in the following result.

\begin{theorem}\label{first}
Let us assume that $M$ has a positive Yamabe constant $Y_g(M,\partial M)$ and let $(u_{n},\psi_{n})$ be a Palais-Smale sequence for $E$ at level $c$. Then there exist $u_{\infty}\in C^{\infty}(M)$, $\psi_{\infty}\in C^{\infty}(\Sigma M)$ such that $(u_{\infty},\psi_{\infty})$ is a solution of $(\ref{eq1})$, $m+\ell$ sequences of points $\tilde{x}_{n}^{1},\cdots, \tilde{x}_{n}^{m} \in M$ and $x_{n}^{1},\cdots, x_{n}^{\ell}$, such that $\lim_{n\to \infty}\tilde{x}_{n}^{k}= \tilde{x}^{k}\in \overset{\circ}{M}$, for $k=1,\dots,m$. $\lim_{n\to \infty}x_{n}^{k}=x^{k}\in \partial M$ for $k=1,\cdots, \ell$ and $m+\ell$ sequences of real numbers $\tilde{R}_{n}^{1},\cdots, \tilde{R}_{n}^{m}, R_{n}^{1},\cdots, R_{n}^{\ell}$ converging to zero, such that:
\begin{itemize}
\item[i)]   $\displaystyle u_{n}=u_{\infty}+\sum_{k=1}^{m} \tilde{v}_{n}^{k}+\sum_{k=1}^{\ell} v_{n}^{k}+o(1)$  in  $H^{1}(M)$,
\item[ii)]  $\displaystyle \psi_{n}=\psi_{\infty}+\sum_{k=1}^{m}\tilde{\phi}_{n}^{k}+\sum_{k=1}^{\ell} \phi_{n}^{k}+o(1)$ in $H^{\frac{1}{2}}_{+}(\Sigma M)$,
\item[iii)] $\displaystyle E^{b}(u_{n},\psi_{n})=E^{b}(u_{\infty},\psi_{\infty})+\sum_{k=1}^{m}E_{\mathbb{R}^{3}}(\tilde{U}_\infty^{k},\tilde{\Psi}_\infty^{k})+\sum_{k=1}^{\ell}E^{b}_{\mathbb{R}^{3}_{+}}(U_{\infty}^{k},\Psi_{\infty}^{k})+o(1)$,
\end{itemize}
where
$$\tilde{v}_{n}^{k}=(\tilde{R}_{n}^{k})^{-\frac{1}{2}}\tilde{\beta}_{k}\tilde{\sigma}_{n,k}^{*}(\tilde{U}_\infty^{k}) ,$$
$$v_{n}^{k}=(R_{n}^{k})^{-\frac{1}{2}}\beta_{k} \sigma_{n,k}^{*}(U_\infty^{k}) ,$$
$$\tilde{\phi}_{n}^{k}=(\tilde{R}_{n}^{k})^{-1}\tilde{\beta}_{k}\tilde{\sigma}_{n,k}^{*}(\tilde{\Psi}_\infty^{k}) ,$$
$$\phi_{n}^{k}=(R_{n}^{k})^{-1}\beta_{k}\sigma_{n,k}^{*}(\Psi_\infty^{k}) ,$$

with $\tilde{\sigma}_{n,k}=(\tilde{\rho}_{n,k})^{-1}$ and $\tilde{\rho}_{n,k}(\cdot)=exp_{\tilde{x}_{n}^k}(\tilde{R}_{n}^k \cdot)$ is the exponential map defined in a suitable neighborhood of $\R^{3}$, $\sigma_{n,k}=\rho_{n,k}^{-1}$ and $\rho_{n,k}=\mathcal{F}_{x^{k}}(R_{n}^{k} \cdot)$ is the Fermi coordinate patch in a suitable neighborhood of $\R^{3}_{+}$.
Also, here $\tilde{\beta}_{k}$ (resp. $\beta_{k}$) is a smooth compactly supported function, such that $\tilde{\beta}_{k}=1$ on $B_{1}(\tilde{x}^{k})$ (resp. $\beta_{k}=1$ on $B_{1}(x^{k})$) and $supp(\tilde{\beta}_{k})\subset B_{2}(\tilde{x}^{k})$ (resp. $supp(\beta_{k})\subset B_{2}(x^{k})$) and $(\tilde{U}_\infty^{k},\tilde{\Psi}_\infty^{k})$ are solutions to the system $(\ref{eq:blowupint})$ on $\mathbb{R}^{3}$ with its Euclidean metric $g_{\R^3}$ while $(U_{\infty}^{k},\Psi_{\infty}^{k})$ are solutions to the system $(\ref{eq:blowup})$ on $\mathbb{R}^{3}_{+}$.
\end{theorem}
Roughly speaking, in items $i),ii)$ above we distinguish concentration profiles corresponding to interior and boundary points, while item $iii)$ gives the corresponding energy quantization. Notice that such result holds for arbitrary $b\geq0$.

\medskip
Next, we will focus on the case $b=0$ and we will let $E:=E^{0}$. We now consider the equation describing boundary concentration profiles \eqref{eq:blowup}. The case of interior blowup points has been treated in \cite{BM}.
As explained in Section \ref{sec:prop}, both equation \eqref{eq:blowup} and the functional \eqref{eq:bubbleaction} are conformally covariant, and then can be equivalently considered on the round hemisphere $(\Sph^3_+,g_0)$:
\begin{equation}\label{eq:bubblessphere}
\left\{\begin{array}{ll}
L_{g_0}u=u|\psi|^{2}\\
D_{g_0}\psi=u^{2}\psi \qquad \mbox{on $\Sph^3_+$}\\
B_{g_{0}}u=0\\
\mathbb{B}^{+}\psi=0\qquad \mbox{on $\partial\Sph^3_+$}
\end{array}
\right.
\end{equation}

\be\label{eq:bubblessphereaction}
E_{\Sph^3_+}=\frac{1}{2}\left( \int_{\Sph^3_+}|\nabla_{g_0} u|^{2}+\frac{R_{g_{0}}}{8}u^{2}+\langle D\psi,\psi\rangle - u^2\vert \psi\vert^2\, dv_{g_0}\right)\,.
\ee
As proved in Lemma \ref{lem:lowerbubbles}, for a non trivial solution $(u,\psi)\in H^1(\Sph^3_+)\times H^{1/2}_{+}(\Sigma_{g_0}\Sph^3_+)$ to \eqref{eq:bubblessphere} there holds
\be\label{eq:energygap}
E_{\Sph^3_+}(u,\psi)\geq \frac{1}{2}Y(\Sph^3_+,\partial\Sph^3_+,[g_0])\lambda^{+}_{\text{CHI}}(\Sph^3_+,\partial\Sph^3_+,[g_0])\,,
\ee
where $Y(\Sph^3_+,\partial\Sph^3_+,[g_0])$ and $\lambda^{+}_{\text{CHI}}(\Sph^3_+,\partial\Sph^3_+,[g_0])$ are the Yamabe and the chiral invariant, respectively, defined as in \eqref{eq:yamabeinv}, \eqref{eq:chiralbag}. We mention that a similar lower bound for non-trivial solution to critical Dirac equations on the round sphere has been proved by Isobe \cite{isobe}, in general dimension. Considering \eqref{eq:blowupint} and \eqref{eq:bubbleenergy} on the round sphere $(\Sph^n,g_0)$ a similar energy gap for non-trivial solutions is proved in \cite{MV3}, namely
\be\label{eq:gapsphere}
E_{\Sph^3}(u,\psi)\geq \frac{1}{2}Y(\Sph^3, [g_0])\lambda^{+}(\Sph^3, [g_0])\,,
\ee
where
\[
E_{\Sph^3}(u,\psi)=\frac{1}{2}\left( \int_{\Sph^3}|\nabla_{g_0} u|^{2}+\frac{R_{g_{0}}}{8}u^{2}+\langle D\psi,\psi\rangle - u^2\vert \psi\vert^2\, dv_{g_0}\right)\,,
\]
$Y(\Sph^3, [g_0])$ is the Yamabe invariant and the conformal invariant
\[
\lambda^{+}(\Sph^3, [g_0]):=\inf_{\bar{g}\in[g_{0}]}\vert\lambda^+_1(\bar{g})\vert \vol (M,\overline{g})^{1/3}
\]
is the analogue of \eqref{eq:chiralbag} on the round sphere.
\begin{remark}\label{rem:threshold}
Notice that since (see \cite{escobar,raulot})
\[
Y(\Sph^3_+,\partial\Sph^3_+,[g_0])<Y(\Sph^3, [g_0])\,,
\]
and
\[
 \lambda^{+}_{\text{CHI}}(\Sph^3_+,\partial\Sph^3_+,[g_0])<\lambda^{+}(\Sph^3, [g_0])\,,
\]
the (sharp) lower bound for the energy of boundary bubbles is strictly smaller than the one for interior bubbles. As a consequence, the former represents the threshold level for compact Palais-Smale sequences, see Theorem \ref{Aubin type result}.
\end{remark}

\begin{definition}
A non trivial solution $(u,\psi)\in H^1(\Sph^3_+)\times H^{1/2}_{+}(\Sigma_{g_0}\Sph^3_+)$ to \eqref{eq:bubblessphere} is called \emph{ground state solution} if
\be\label{eq:groundlevel}
E_{\Sph^3_+}(u,\psi)=\frac{1}{2}Y(\Sph^3_+,\partial\Sph^3_+,[g_0])\lambda^{+}_{\text{CHI}}(\Sph^3_+,\partial\Sph^3_+,[g_0])\,,
\ee
that is, equality holds in \eqref{eq:energygap}.
\end{definition}
The following result classifies ground states of \eqref{eq:bubblessphere}. We mention that the analogous result for Dirac-Einstein equations on the 3-sphere (i.e. for minimizers of \eqref{eq:gapsphere}) and for critical Dirac equations in arbitrary dimensions has been proved in \cite{BM,BMW}, respectively.

\begin{theorem}\label{thm:classification}
Let $(u,\psi)\in H^1(\Sph^3_+)\times H^{1/2}_{+}(\Sigma_{g_0}\Sph^3_+)$ be a ground state solution to \eqref{eq:bubblessphere}, and assume $u\geq0$. Then, up to a conformal diffeomorphism,  $u\equiv1$ and $\psi$ is a $\left(-\frac{1}{2} \right)-$Killing spinor. Namely, there exists a conformal diffeomorphism $f$ of the round $3$-hemisphere and a $\left(-\frac{1}{2} \right)-$Killing spinor $\Psi$ such that
\[
u=(\det(\dd f))^{1/6}
\]
and
\[
\psi=(\det(\dd f))^{1/3}F_{f^*g_0,g_0}(f^*\Psi)\,,
\]
where $F_{f^*g_0,g_0}$ is the isometry of spinor bundles for conformally related metrics, and the pullback $f^*\Psi$ is defined as in \eqref{eq:pullbackspinor}.
\end{theorem}
The above results can be rewritten on the Euclidean space, in a more explicit form.
\begin{corollary}\label{cor:euclideanbubbles}
Let $(u,\psi)\in H^1(\R^3_+)\times H^{1/2}_+(\Sigma\R^3_+)$ be a ground state solution to \eqref{eq:blowup} with $u\geq0$. Then there exist $\lambda>0$, $y\in\R^2$ and  a parallel spinor $\Phi_0\in\Sigma\R^3_+$ such that, setting $\Phi=(\Phi_0-\nu\cdot\G\Phi_0)$ there holds
\be\label{eq:scalarpart}
u(x)=\left(\frac{2\lambda}{\lambda^2+\vert \tilde{x}-y\vert^2+x^2_3}\right)^{1/2}\,,\qquad x=(\tilde{x},x_3)\in\R^3_+
\ee
and
\be\label{eq:spinorpart}
\psi(x)=\left(\frac{2\lambda}{\lambda^2+\vert \tilde{x}-y\vert^2+x^2_3}\right)^{3/2}\left(1-\left(\frac{\tilde{x}-y}{\lambda},\frac{x_3}{\lambda} \right) \right)\cdot \Phi \,,\qquad x=(\tilde{x},x_3)\in\R^3_+
\ee
where $\tilde{x}\in\R^2$ and $x_3\geq0$.
\end{corollary}
\medskip

The last main results of the paper consist of the proof of an Aubin-type inequality and a related existence result for the problem \eqref{eq1}, with $b=0$ (see \cite[Section 5]{MV3} for similar results in the case of a manifold without boundary). We initially define the two functionals relevant to the statement of the result:
\begin{equation}\label{eq:functionals}
\tilde{E}(u,\psi)=\frac{\left(\displaystyle\int_{M}uL_{g}udv_{g}\right)\left(\displaystyle\int_{M}\langle D_g \psi,\psi\rangle dv_{g}\right)}{\displaystyle\int_{M}|u|^{2}|\psi|^{2}dv_{g}} ,\qquad I(\psi)=\frac{\displaystyle\int_{M}\langle D_g \psi,\psi\rangle dv_{g}}{\displaystyle\int_{M}|u|^{2}|\psi|^{2}dv_{g}} \;
\end{equation}
for any non-trivial  $(u,\psi) \in H^1(M) \times H^{\frac{1}{2}}_+(\Sigma M)$ with $\int_{M}|u|^{2}|\psi|^{2}\ dv_{g} \not=0$ and assuming (without loss of generality) that
$h_{g}=0$. These functional will be addressed in Section \ref{sec:aubin}. We also refer the reader to \cite{MV3} to get more context on the definition of these functionals. Now, let $H^{\frac{1}{2},-}_{+}$ be the negative space of $H^{\frac{1}{2}}_+(\Sigma M)$ according to the spectral decomposition of the Dirac operator and consider $P^{-}$ the projector on $H^{\frac{1}{2},-}$, as explained in Section \ref{sec:spinorspaces}.

We can define the following conformal constant (namely, depending only on the conformal class of the metric $g$)
\begin{equation}\label{Ytilde}
\tilde{Y} (M, \partial M, [g])=\inf\left\{\begin{array}{ll}
\tilde{E}(u,\psi); \text{ where  $(u,\psi)\in H^{1}(M)\setminus\{0\}\times H^{\frac{1}{2}}_+(\Sigma M)\setminus \{0\}$ s.t. }\\
\\
I(\psi)> 0, \quad  P^{-}\left(D_g \psi-I(\psi)u^{2}\psi\right) =0\end{array}  \right\} \;  .
\end{equation}
We point out here that there is an abuse of notation writing $P^{-}\left(D_g \psi-I(\psi)u^{2}\psi\right)$ instead of $P^{-}\left(\psi-I(\psi)D_{g}^{-1}(u^{2}\psi\right)$. In fact, the equation $P^{-}\left(D_g \psi-I(\psi)u^{2}\psi\right) =0$ should be understood in the sense of duality, that is:
$$\langle D_{g}\psi-I(\psi)u^{2}\psi,P^{-}(\varphi)\rangle_{H^{-\frac{1}{2}},H^{\frac{1}{2}}}=0 \text{ for all } \varphi\in H^{\frac{1}{2}}_{+}(\Sigma M).$$
We want to compare $\tilde{Y} (M, \partial M, [g])$ with the invariant on the sphere $(\Sph^3_+,g_0)$, namely,
$$\tilde{Y} (\Sph^3_+,\partial\Sph^3_+,[g_0]):=Y(\Sph^3_+,\partial\Sph^3_+,[g_0])\lambda^{+}_{\text{CHI}}(\Sph^3_+,\partial\Sph^3_+,[g_0]) \; .$$
This comparison will be entangled to an existence result for \eqref{eq1}, in analogy with the classical Yamabe problem \cite{escobar,leeparker}. Indeed, we prove the following:
\begin{theorem}\label{Aubin type result}
Let $(M,g)$ be a compact oriented three-dimensional Riemannian manifold with boundary. It holds:
\begin{equation}\label{Aubin inequality}
  Y(M, \partial M, [g])\lambda^{+}_{\text{CHI}}(M, \partial M, [g]) \leq\tilde{Y}(M, \partial M, [g])  \leq \tilde{Y} (\Sph^3_+,\partial\Sph^3_+,[g_0])\, .
\end{equation}

Moreover, if $$\tilde{Y}(M, \partial M, [g]) <  \tilde{Y} (\Sph^3_+,\partial\Sph^3_+,[g_0]) ,$$
then the problem $(\ref{eq1})$, with $b=0$, has a non-trivial ground state solution.
\end{theorem}

 \subsection{Outline of the paper}
In Section \ref{sec:prop} we collect some preliminary results about the operators involved and their conformal covariance, while in Section \ref{sec:reg} we prove a regularity result for solutions to \eqref{eq1}. We then analyze the bubbling phenomenon in Section \ref{sec:PS}, proving Theorem \ref{first}. Ground state bubbles are classified, as stated in Theorem \ref{thm:classification}, in Section \ref{sec:groundstates}. Finally, we prove the Aubin-type inequality \eqref{Aubin inequality} and the related existence result stated in Theorem \ref{Aubin type result}, in Section \ref{sec:aubin}.

\subsection*{Acknowledgements.}

The authors wish to thank the reviewer for careful reading the manuscript and for the accurate remarks.
\subsection*{Data availability statement}
 Data sharing not applicable to this article as no datasets were generated or analysed during the current study.

\section{Preliminary properties and Conformal Covariance}\label{sec:prop}

In this section we review some notions on spin structures and Dirac operators useful in the sequel, for the convenience of the reader. We refer to \cite{Jost,LawsonMichelsohn} for more details. We also define the Sobolev spaces that will play a role in the subsequent analysis and recall some of their properties.
\smallskip

The conformal Laplacian acting on functions on the Riemannian manifold $(M,g)$ is defined by
$$L_g u:=-\Delta_g u+\frac{1}{8}R_g u ,$$
where $\Delta_g$ is the standard Laplace-Beltrami operator and $R_g$ is the scalar curvature.
As already observed, $(L_{g},B_{g})$ is elliptic with compact resolvent, where the boundary condition $B_g$ is defined in \eqref{eq1}. We refer the reader to \cite[Chap~V]{ADN} and \cite[Chap~2,Sec~5]{LM}, for more details.

We will denote by $H^1(M)$ the usual Sobolev space on $M$. By the Sobolev embedding and trace theorems there is a continuous embedding $$H^1(M)\hookrightarrow L^p(M), \quad 1\leq p \leq 6 \text{ and } H^{1}(M)\hookrightarrow L^{q}(\partial M), \quad 1\leq q\leq 4,$$
which are compact if $1\leq p <6$ and $1\leq q<4$.

 \subsection{Spin structure and the Dirac operator}
 Let $(M,g)$ be an oriented Riemannian manifold, and let $P_{SO}(M,g)$ be its frame bundle.
 \begin{definition}
 A \textit{spin structure} on $(M,g)$ is a pair $(P_{Spin} (M,g),\sigma)$, where $P_{Spin} (M,g)$ is a $Spin(n)$-principal bundle and $\sigma:P_{Spin} (M,g)\rightarrow P_{SO}(M,g)$ is a 2-fold covering map, which is the non-trivial covering $\lambda:Spin(n)\rightarrow SO(n)$ on each fiber.
 \end{definition}

In other words, the quotient of each fiber by $\{-1,1\}\simeq\mathbb{Z}_{2}$ is isomorphic to the frame bundle of $M$, so that the following diagram commutes
\begin{center}
 \begin{tikzcd}
 P_{Spin}(M,g)\arrow[rr,"\sigma"]\arrow[dr]& &P_{SO}(M,g)\arrow[dl]\\
 &M &
 \end{tikzcd}
\end{center}
A Riemannian manifold $(M,g)$ endowed with a spin structure is called a \textit{spin manifold}.

In particular, the Euclidean half-space $(\R^{n}_+,g_{\R^n})$ and the round hemisphere $(\mathbb{S}^{n}_+,g_0)$ with $n\geq 2$, that are relevant for our purposes, admit a unique spin structure.

\begin{definition}
The \emph{complex spinor bundle} $\Sigma M\to M$ is the vector bundle associated to the $Spin(n)$-principal bundle $P_{Spin}(M,g)$ via the complex spinor representation of~$Spin(n)$.
\end{definition}
The complex spinor bundle~$\Sigma M$ has rank $N=2^{[\frac{n}{2}]}$ and it is endowed with a canonical spin connection~$\nabla$ (which is a lift of the Levi-Civita connection, denoted by the same symbol) and a Hermitian metric~$g$ shortly denoted by $(\cdot, \cdot)$. We will be using the euclidean structure rather than the Hermitian one, hence, we will consider $\langle \cdot,\cdot \rangle$, the real part of $(\cdot,\cdot)$.

In particular, the spinor bundle of the Euclidean half-space $\R^n_+$ is trivial, so that we can identify spinors with vector-valued functions $\psi:\R^n\to\C^N$.
\smallskip

The \emph{Clifford multiplication} is a map $\cdot: TM\to \operatorname{End}_\C(\Sigma M)$ verifying the \emph{Clifford relation}
\begin{equation}
 X\cdot Y+Y\cdot X=-2g(X,Y){\1}_{\Sigma M},
\end{equation}
for any tangent vector fields $X,Y\in\Gamma(TM)$, and is compatible with the bundle metric $g$.

Taking a local oriented orthonormal tangent frame $(e_j)^3_{j=1}$ the Dirac operator is defined as
\[
D^g\psi:=\sum^3_{j=1}e_j\cdot\nabla^g_{e_j}\psi\,,\qquad \psi\in\Gamma(\Sigma M)\,.
\]

Given~$\alpha\in\C$, a non-zero spinor field~$\psi\in\Gamma(\Sigma M)$ is called~$\alpha$-Killing if
 \begin{equation}\label{eq:killingequation}
  \nabla^g_X\psi=\alpha X\cdot\psi,\qquad \forall X\in\Gamma(TM).
 \end{equation}

For more information on Killing spinors we refer the reader to~\cite[Appendix A]{diracspectrum}. On $(\Sph^n_+,g_0)$ $\alpha$-Killing spinors only exist for $\alpha=\pm1/2$ and are restriction of $\alpha$-Killing spinors on the round sphere, and via the stereographic projection the pull-back on the Euclidean half-space $\R^3_+$ of the~$\pm\frac{1}{2}$-Killing spinors have the form
\begin{equation}\label{eq:Killing on Rn}
 \Psi(x)
 =\left( \frac{2}{1+|x|^2}\right)^{\frac{n}{2}}({\1}\pm x)\cdot\Phi_0
\end{equation}
where~$\1$ denotes the identity endomorphism of the spinor bundle $\Sigma_{g_{\R^3_+}}\R^3_+$ and $\Phi_0$ is a parallel spinor, see e.g. \cite{raulot2}.

Finally, observe that $\partial M$, being an oriented hypersurface of $M$, is itself a spin manifold and its spinor bundle is obtained by restricting $\Sigma M$ to the boundary of $M$.
 In particular in odd dimensions the spinor bundle restricted to $\partial M $ can be identified with the intrinsic bundle on $\partial M$, namely
 \be\label{eq:restriction}
 \Sigma M_{\vert\partial M} = \Sigma\partial M\,.
 \ee

 \subsection{Chiral bag boundary conditions and the chiral invariant}
When the Dirac operator is studied on manifold with boundaries, one can impose different kinds of boundary conditions depending on the physical model or the invariant that one needs to compute. For instance in \cite{APS} the authors introduced the celebrated non-local boundary condition in order to establish index theorems. The (APS) boundary condition was also extensively used to establish other results see \cite{HE2, Wit}. Another interesting boundary condition is the MITbag condition introduced in the study of fields confined in a finite region of the space. The advantage of this condition is its local aspect. In our case, we are interested in the chiral boundary condition (CHI) introduced in the study of the mass of asymptotically flat manifolds as in \cite{HE} and for the proof of the positive mass theorem for manifolds with boundary as in \cite{ABD}. This boundary condition is more adequate for our problem mainly because it is well behaved under a conformal change of the metric as it was expressed in \cite{raulot,raulot2}. We refer to reader to \cite{HMR} for more details on the above mentioned boundary conditions.

 A chirality operator on $M$ is a linear map $\G:\Gamma(\Sigma M)\to\Gamma(\Sigma M)$ such that
 \be\label{eq:chirality}
 \begin{split}
 \G^2&=\1\,,\qquad \langle\G \psi,\G\varphi\rangle=\langle\psi,\varphi\rangle\,, \\
 \nabla_X(\G\psi)&=\G\nabla_X\psi\,,\qquad X\cdot\G\psi=-\G X\cdot\psi\,,
 \end{split}
 \ee
 for each vector field $X$ and spinor fields $\psi,\varphi$. A typical example of a chirality operator in three dimensions is when the manifold is a space-like hypersurface of a four dimensional Lorentz manifold. $G$ in this case is defined by the Clifford multiplication by a unit time-like vector field. We refer the reader to \cite[Section~2]{HE} for more details about this example, and to the discussion in the proof of \cite[Thm.3.2]{Chrusciel-Maerten}.

 Given a chirality operator $\G$, the fibre preserving endomorphism $\nu\cdot\G:\Gamma(\Sigma\partial M)\to\Gamma(\Sigma\partial M)$ is self-adjoint with respect to the pointwise Hermitian product and squares to the identity. Then it has eigenvalues $\pm1$ and
 \be\label{eq:B+}
 \mathbb{B}^\pm_g:=\frac{1}{2}({\1}\pm\nu\cdot\G) : L^2(\Sigma\partial M)\to L^2(V^\pm)
 \ee
 is the the projector onto the eigensubbundle $V^\pm$ corresponding to the eigenvalue $\pm1$. Such operator defines an elliptic boundary condition for $D_g$ called \emph{chiral bag boundary condition}. More precisely, in \cite{fs}  the authors show that
 \[
 D_g:H^1_+(\Sigma M):=\{ \varphi\in H^1(\Sigma M) \,:\, \mathbb{B}^+_g(\varphi_{\vert \partial M})=0 \}\to L^2(\Sigma M)
 \]
 is a self-adjoint Fredholm operator, whose spectrum consists of isolated real eigenvalues with finite multiplicity. In fact, since $D_{g}$ is elliptic of order 1 and self-adjoint, it is Fredholm as an operator $D_{g}:H^{\frac{1}{2}}_{+}(\Sigma M):=\{ \varphi\in H^{\frac{1}{2}}(\Sigma M) \,:\, \mathbb{B}^+_g(\varphi_{\vert \partial M})=0 \} \to H^{-\frac{1}{2}}(\Sigma M)$. More details about these fractional Sobolev spaces will be provided in the next section.

 \begin{definition}
 The chiral bag invariant is defined as
 \be\label{eq:chiralbag}
 \lambda^\pm_{CHI}(M,\partial M, [g]):=\inf_{\bar{g}\in[g]}\vert\lambda^\pm_1(\bar{g})\vert \vol (M,\overline{g})^{1/n}\,,
 \ee
 where $\lambda^\pm_1(g)$ s the smallest positive/negative eigenvalue of the Dirac operator under chiral bag boundary conditions.
 \end{definition}

 If $n\geq 3$ the Hijazi inequality on manifolds with boundary \cite{raulot} relates $\lambda^\pm_1(g)$ to the first eigenvalue of the conformal Laplacian $\mu_1(L_g)$:
 \[
 \lambda^\pm_1(g)^2\geq\frac{n}{4(n-1)}\mu_1(L_g) \,,
 \]
 where equality holds if and only if $(M,g)$ is isometric to a round half-sphere. Moreover, there holds
 \be\label{eq:}
 \lambda^\pm_{CHI}(M,\partial M,[g])\geq\frac{n}{4(n-1)}Y(M,\partial M,[g])\,.
 \ee
  \subsection{Sobolev spaces of spinors}\label{sec:spinorspaces}

Using the spectral decomposition of the Dirac operator with chiral bag boundary conditions one can define fractional Sobolev spaces of spinors. Of particular interest for us will be the space $H^{\frac{1}{2}}_+(\Sigma M)$. As already observed, the Dirac operator $D_g$ has compact resolvent and there exists a complete $L^2$-orthonormal basis of eigenspinors $\{\psi_i\}_{i\in\mathbb{Z}}$ of the operator
$$D_g\psi_i=\lambda_i \psi_i ,$$
and the eigenvalues $\{\lambda_i\}_{i\in\mathbb{Z}}$ are unbounded, that is $|\lambda_i|\rightarrow\infty$, as $|i|\rightarrow\infty$.
Now if $\psi\in L^{2}(\Sigma M)$, there holds
$$\psi=\sum_{i\in \mathbb{Z}}a_{i}\psi_{i}.$$
so that we can define the unbounded operator $|D_g|^{s}: L^{2}(\Sigma M)\rightarrow L^{2}(\Sigma M)$ by
$$|D_g|^{s}(\psi)=\sum_{i\in \mathbb{Z}} a_{i}|\lambda_{i}|^{s}\psi_{i}.$$
We denote by $H^s_+(\Sigma M)$ the domain of $|D_g|^{s}$, namely there holds $\psi\in H^s_+(\Sigma M)$ if and only if
$$\sum_{i\in \mathbb{Z}} a_{i}^2|\lambda_{i}|^{2s}<+\infty .$$
 For $s >0$, the following inner product
$$\langle u,v\rangle_{s}=\langle|D_g|^{s}u,|D_g|^{s}v\rangle_{L^{2}},$$
induces a norm on $H^{s}_{+}(\Sigma M)$; we will take
$$\langle u,u\rangle:=\langle u,u\rangle_{\frac{1}{2}}=\|u\|^{2}$$
as our standard norm for the space $H^{\frac{1}{2}}_{+}(\Sigma M)$. The Sobolev embedding theorems ensures that there is a continuous embedding
$$H^{\frac{1}{2}}_{+}(\Sigma M) \hookrightarrow L^p(\Sigma M), \quad 1\leq p \leq 3 ,$$
which is compact if $1\leq p <3$.
Finally, we will decompose $H^{\frac{1}{2}}_{+}(\Sigma M)$ according to the spectrum of $D_g$. Let us consider the $L^2$-orthonormal basis of eigenspinors $\{\psi_i\}_{i\in\mathbb{Z}}$: we denote by $\psi_i^-$ the eigenspinors with negative eigenvalue, $\psi_i^+$ the eigenspinors with positive eigenvalue and $\psi_i^0$ the eigenspinors with zero eigenvalue; observe that the kernel of $D_g$ is finite dimensional. Now we set
$$H^{\frac{1}{2},-}_+:=\overline{\text{span}\{\psi_i^-\}_{i\in\mathbb{Z}}},\quad
H^{\frac{1}{2},0}_+:=\text{span}\{\psi_i^0\}_{i\in\mathbb{Z}}, \quad
H^{\frac{1}{2},+}_+:=\overline{\text{span}\{\psi_i^+\}_{i\in\mathbb{Z}}},$$
where the closure is taken with respect to the $H^{\frac{1}{2}}$-topology. Therefore we have the orthogonal decomposition  $H^{\frac{1}{2}}_+(\Sigma M)$ as:
$$H^{\frac{1}{2}}_+(\Sigma M)=H^{\frac{1}{2},-}_+\oplus H^{\frac{1}{2},0}_+\oplus H^{\frac{1}{2},+}_+.$$
Also, we let $P^{+}$ and $P^{-}$ be the projectors on $H^{\frac{1}{2},+}_+$ and $H^{\frac{1}{2},-}_+$ respectively. Depending on the situation, it will be practical to denote by $\mathcal{H}$ the product space $\mathcal{H}=H^1(M) \times H^{\frac{1}{2}}_+(\Sigma M)$.
\medskip

Observe that some of the result of the present paper will be for manifolds with {\emph positive} Yamabe invariant $Y(M,\partial M,[g])>0$.  This implies that there are no harmonic spinors, that is, the kernel of $D_g$ is trivial.


  \subsection{Conformal covariance}\label{sec:conformalcovariance}
  In this section we recall known formulae relating the Dirac and conformal Laplace operators for conformally equivalent metrics, see  e.g. \cite{diracspectrum,hitchin}. To this aim the various geometric objects are explicitly labeled with the corresponding metric $g$, e.g.~$\Sigma_g M, \nabla^{g}$, $D_g$, etc.
 \smallskip

Fix a smooth function $f\in C^\infty(M)$ and consider the conformal metric $g_f=e^{2f}g$, which induces an isometric isomorphism of spinor bundles
\begin{equation}\label{eq:covariantoperators}
 F\equiv F_{g,g_f}\colon (\Sigma_{g}M, g)\to (\Sigma_{g_f}M, g_f)\,.
\end{equation}
Then there holds
\begin{align}\label{eq:conformaldirac}
 D_{g_f}F(e^{-\frac{n-1}{2}f}\psi)
 = e^{-\frac{n+1}{2}f}F(D_g\psi),
\end{align}
\begin{align}\label{eq:conformalpenrose}
 L_{g_f}(e^{-\frac{n-2}{2}f}u)=e^{-\frac{n+2}{2}f}L_g u\,.
\end{align}

The following quantities are conformally invariant. Setting
\be\label{eq:rescaling}
\varphi:=F(e^{-\frac{n-1}{2}f}\psi)\,,\qquad v:=e^{-\frac{n-2}{2}f}u
\ee
there holds
\begin{equation}
\label{eq:conformalinvariance}
\begin{split}
\int_{M} uL_g u \,d{v_g} &=\int_{M}vL_{g_f} v\, d{v_{g_f}}\,,\qquad\int_{M} \langle D_g \psi,\psi\rangle\,d{v_g} =  \int_{M} \langle D_{g_f} \varphi,\varphi\rangle \,d{v_{g_f}} \\
\int_{M}\vert u\vert^6 d{v_g}&=\int_{M}\vert v\vert^6 d{v_{g_f}}\,,\qquad\int_{M} |\psi|^{3}\,d{v_g} =  \int_{M} |\varphi|^{3} \,d{v_{g_f}} \\
\int_{M}\vert u\vert^2\vert\psi\vert^2 d{v_g} &=\int_{M}\vert v\vert^2\vert\varphi\vert^2\,d{v_{g_f}}  \,, \qquad \int_{\partial M}u^{4}\ d\sigma_{g}=\int_{\partial M} v^{4}\ d\sigma_{g_{f}}\,.
\end{split}
\end{equation}
For boundary terms, notice that from the well-known formula for the mean curvature (see e.g. \cite{escobar}), $h_{g_f}=e^{-f}(h_g+\frac{\partial f}{\partial\nu})$, where $\nu$ is the outward normal to $\partial M$, one gets
\be\label{eq:Binvariance}
\int_{\partial M}v B_{g_f}v\,d\sigma_{g_f}=\int_{\partial M}u B_{g}u\,d\sigma_{g}\,.
\ee
As mentioned above, Chiral bag boundary conditions behave nicely with respect to conformal changes of metrics. Namely, if $\G$ is a chirality operator on $\Sigma_g M$, then
\[
\overline{\G}:=F\circ\G\circ F^{-1}:\Sigma_{\overline{g}}M\to\Sigma_{\overline{g}}M
\]
is a chirality operator for the metric $\overline{g}$ and the associated orthogonal projection $\mathbb{B}^\pm_{\overline{g}}$ as in \eqref{eq:B+} define elliptic boundary condition for $D_{\overline{g}}$ ( see \cite{HMR}).
Consequently the action \eqref{eq:energy} is conformally invariant, and hence also equation \eqref{eq1}. \medskip

In particular, using a conformal change of the metric, \eqref{eq:blowup} can be seen as an equation on the round hemisphere $(\Sph^3_+,g_0)$. This is done using the stereographic projection $\pi:\Sph^3_+\setminus{N}\to\R^3_+$, as a conformal transformation, where $N\in \partial\Sph^3_+$ is the north pole.  Indeed, one has $(\pi^{-1})^*g_0=\frac{4}{(1+\vert x\vert^2)^2}g_{\R^3}$ with $x\in\R^3$. Therefore, if $(u,\psi)\in \mathring{H}^{1}(\R^3_+)\times \mathring{H}^{1/2}_{+}(\R^3_+,\C^2)$ is a weak solution to \eqref{eq1}, then one can consider $(v,\varphi)$ defined in by \eqref{eq:rescaling} with $e^{2f}=\frac{4}{(1+\vert x\vert^2)^2}$. Such a pair is in $H^1(\Sph^3_+)\times H^{1/2}(\Sigma_{g_0}\Sph^3_+)$ by \eqref{eq:conformalinvariance}, and it is a weak solutions to \eqref{eq:bubblessphere}. Notice that, a priori, $(v,\varphi)$ is a weak solution only on $\Sph^3_+\setminus \{N\}$. However, the removability of the singularity at the north pole $N$ can be proved by standard cut-off arguments.

Finally, we conclude this section by observing that, using \eqref{eq:conformalinvariance}, the functional \eqref{eq:bubbleaction} is also conformally invariant, i.e.
\[
E^{b}_{\R^3_+}(u,\psi)=E^{b}_{\Sph^3_+}(v,\varphi)\,.
\]
\section{Regularity}\label{sec:reg}
In this section we prove a regularity result for solutions to \eqref{eq1}, using an argument similar to \cite{isobe, MV3}. The the equations considered are critical and compared to \cite{MV3}, we do have a critical behavior at the boundary. In order to initiate a bootstrap arguments one needs an extra step to improve the regularity. This is usually what is referred to in the literature by $\varepsilon$-regularity, where the smallness of some critical norm, induces an improved regularity.

\begin{theorem}\label{thmreg}
Let $(u,\psi)\in H^1(M) \times H^{\frac{1}{2}}_{+}(\Sigma M)$ be a weak solution to the system \eqref{eq1}, then $(u,\psi)\in C^{\infty}(M) \times C^{\infty}(\Sigma M)$.
\end{theorem}
\begin{proof}
First of all, we recall the Sobolev embeddings:
$$\left\{\begin{array}{lll}
H^1(M)\hookrightarrow L^p(M), \quad 1\leq p \leq 6 ,\\
H_{+}^{\frac{1}{2}}(\Sigma M) \hookrightarrow L^p(\Sigma M), \quad 1\leq p \leq 3,\\
H^1(M)\hookrightarrow L^q(\partial M), \quad 1\leq q \leq 4.
\end{array}
\right.$$
We also recall that the two operators $(L_{g},B_{g})$ and $(D_{g},\mathbb{B}^{+})$ are elliptic. Now let $\rho,\eta \in C^{\infty}(M)$, with $\eta=1$ on $supp(\rho)$ and let us denote $\Omega=supp(\eta)$. We compute
\begin{align}
L_g(\rho u)&=-\Delta_g(\rho u)+\frac{1}{8}R_g \rho u  \notag\\
&=-\rho\Delta_g u-u\Delta_g \rho-2g(\nabla\rho,\nabla u)+\frac{1}{8}R_g \rho u \notag\\
&=\rho L_g( u)-u\Delta_g \rho-2g(\nabla\rho,\nabla u) \notag\\
&=\eta|\psi|^2 \rho u -u\Delta_g \rho-2g(\nabla\rho,\nabla u) \notag ,
\end{align}
and
$$D_g(\rho \psi)=\rho D_g \psi+\nabla\rho\cdot\psi=\eta u^2 \rho\psi+\nabla\rho\cdot\psi.$$
If $supp(\rho)\cap \partial M\not =0$, we have
$$B_{g}(\rho u)=\rho B_{g}u+u\frac{\partial \rho}{\partial \nu}=b\eta \rho u^{3}+ u\frac{\partial \rho}{\partial \nu}.$$
and
$$\mathbb{B}^{+}(\rho \psi)=\rho \mathbb{B}^{+}\psi=0.$$

Notice that, since $u\in H^1(M)$,and $\psi \in L^3(\Sigma M),$ we have
$$u\Delta_g \rho+2g(\nabla\rho,\nabla u)\in L^2(M),\quad \nabla\rho\cdot\psi \in L^3(\Sigma M)\quad \text{ and } u\frac{\partial \rho}{\partial \nu}\in L^{4}(\partial M).$$

Now one can define the two perturbation maps:
$$\left\{\begin{array}{ll}
P_1:W^{2,q}(M)\longrightarrow L^{q}(M)\times W^{1-\frac{1}{q},q}(\partial M),\\
&\\
P_1(v)=(\eta|\psi|^{2}v, b\eta u^{2} v)
\end{array}\right. \text{ and }\quad \left\{\begin{array}{ll}
P_2:W^{1,p}_{+}(\Sigma M)\longrightarrow L^{p}(\Sigma M),\\
&\\
P_2(\phi)=\eta u^2 \phi.\end{array} \right.$$

In what follows we denote by $\Vert\cdot \Vert_{op}$ the operator norm, the norms of the spaces involved being clear from the context.

\begin{lemma}
The operators $P_{1}$ and $P_{2}$ are bounded for $1<q<\frac{3}{2}$ and $1<p<3$. Moreover,
$$\|P_1\|_{op}\leq C\Big(\|\psi\|^{2}_{L^{3}(\Sigma \tilde{\Omega})}+b\|u\|^{2}_{L^{6}(\tilde{\Omega})}+b\|u\|_{L^{6}(\tilde{\Omega})}+b^{2}\|u\|^{2}_{L^{4}(\tilde{\Omega}\cap \partial M)}\Big)$$
and
$$\|P_2\|_{op}\leq C_p \|u\|^2_{L^6(\Omega)},$$
where $\Omega \subset \tilde{\Omega}$.
\end{lemma}
\begin{proof}
Notice that the estimate for $\Vert P_2\Vert_{op}$ follows by a simple application of H\"older's inequalities. So we focus on $P_{1}$ and we start by the second component of $P_{1}$. Recall, from the Sobolev trace theorem, that
$$\|\eta u^{2}v\|_{W^{1-\frac{1}{q},q}(\partial M)}\leq C\|\eta u^{2}v\|_{W^{1,q}(M)}.$$
for all $q<\frac{3}{2}$.  Hence, one needs to show that $\eta u^{2} v \in W^{1,q}(M)$ for $1\leq q <\frac{3}{2}$ to insure that $P_{1}$ is well defined. Indeed, we have
$$\|\eta u^{2}v\|_{L^{q}(M)}\leq \|u\|_{L^{6}(\Omega)}^{2}\|v\|_{L^{\frac{3q}{3-q}}(M)}\leq C\|u\|_{L^{6}(\Omega)}^{2}\|v\|_{W^{2,q}(M)}.$$
On the other hand, we have
$$\nabla (\eta u^{2}v)=u^{2}v \nabla \eta +\eta u^{2} \nabla v+2\eta uv \nabla u.$$
Therefore, since
$$ \|\nabla\eta u^{2}v\|_{L^{q}(M)}\leq \|u\|_{L^{6}(\Omega)}^{2}\|v\|_{L^{\frac{3q}{3-q}}(M)}\leq C\|u\|_{L^{6}(\Omega)}^{2}\|v\|_{W^{2,q}(M)},$$
$$\|\eta u^{2}\nabla v\|_{L^{q}(M)}\leq \|u\|_{L^{6}(\Omega)}^{2}\|\nabla v\|_{L^{\frac{3q}{3-q}}(M)}\leq C\|u\|_{L^{6}(\Omega)}^{2}\|v\|_{W^{2,q}(M)}$$
and

\begin{align}
\|\eta u v \nabla u\|_{L^{q}(M)}&\leq \|u\|_{L^{6}(\Omega)}\|\nabla u\|_{L^{2}(\Omega)}\|v\|_{L^{\frac{3q}{3-2q}}(M)}\notag\\
&\leq C\|u\|_{L^{6}(\Omega)}\|\nabla u\|_{L^{2}(\Omega)}\|v\|_{W^{2,q}(M)}\notag\\
&\leq C(\|u\|^2_{L^{6}(\Omega)}+\|\nabla u\|^2_{L^{2}(\Omega)})\|v\|_{W^{2,q}(M)},
\end{align}
We see that $\eta u^{2}v \in W^{1,q}(M)$. Moreover, by H\"{o}lder's inequality, we have, for the first component of $P_{1}$,

$$\|\eta |\psi|^{2}v\|_{L^{q}(M)}\leq C\|\psi\|_{L^{3}(\Sigma \Omega)}^{2}\Vert v\Vert_{L^{\frac{3q}{3-2q}}(M)}\leq C\|\psi\|_{L^{3}(\Sigma \Omega)}^{2}\Vert v\Vert_{W^{2,q}(M)}.$$
Thus, $P_{1}$ is well defined and
$$\|P_{1}\|_{op}\leq C\Big(\|\psi\|_{L^{3}(\Sigma \Omega)}^{2}+b\|u\|_{L^{6}(\Omega)}^{2}+b\|\nabla u\|_{L^{2}(\Omega)}^{2}\Big).$$
To finish the proof, we just need to estimate $\|u\|_{H^{1}(\Omega)}$. So we use $\rho^{2}_{1} u$ as a test function in the first equation $(\ref{eq1})$, where $\rho_{1}\in C^{\infty}(M)$ is compactly supported and $\rho_{1}=1$ on $\Omega$. This leads to
$$\int_{M}\rho^{2}_1 u L_{g}u =\int_{M} \rho_1 u^{2}|\psi|^{2}.$$
Hence, integrating by parts leads to
$$\int_{M}u\nabla \rho^{2}_1\cdot \nabla u +\rho^{2}_1 |\nabla u|^{2} -b\int_{\partial M} \rho^{2}_1 u^{4}+\frac{1}{8}\int_{M}\rho^{2}_1 R u^{2}=\int_{M} \rho^{2}_1 u^{2}|\psi|^{2}.$$
Since $$\int_{M}u\nabla \rho^{2}_1\cdot \nabla u \leq 2\|u\nabla\rho_1\|_{L^{2}(M)}\|\rho_1\nabla u\|_{L^{2}(M)},$$
we have
$$\int_{M}\rho^{2}_1 |\nabla u|^{2}\leq C(\|u\|_{L^{6}(\Omega)}^{2}\|\psi\|_{L^{3}(\Omega)}^{2}+\|u\|_{L^{6}(\Omega)}^{2}+b\|u\|_{L^{4}(\partial \Omega\cap \partial M)}^{4})\,.$$
Thus, if $\Omega \subset \tilde{\Omega}=supp(\rho_{1})$ we have
$$\|u\|_{H^{1}(\Omega)}\leq C(\|u\|_{L^{6}(\tilde{\Omega})}\|\psi\|_{L^{3}(\Sigma\tilde{\Omega})}+\|u\|_{L^{6}(\tilde{\Omega})}+b\|u\|_{L^{4}(\partial \tilde{\Omega}\cap \partial M)}^{2}).$$
The set $\tilde{\Omega}$ should be thought of as a small dilation of $\Omega$ that shrinks when $\Omega$ itself is shrunk. To summarize, we have
$$\|P_{1}\|_{op}\leq C\Big(\|\psi\|^{2}_{L^{3}(\Sigma \tilde{\Omega})}+b\|u\|^{2}_{L^{6}(\tilde{\Omega})}+b\|u\|_{L^{6}(\tilde{\Omega})}+b^{2}\|u\|^{2}_{L^{4}(\tilde{\Omega}\cap \partial M)}\Big).$$

\end{proof}

In this way the operators
$$(L_g-\eta|\psi|^2, B_{g}-b\eta u^{2}):W^{2,q}(M)\longrightarrow L^{q}(M)\times L^{r}(\partial M), \qquad 1<q<\frac{3}{2} , 1<r<2,$$
$$D_g-\eta u^2:W^{1,p}(\Sigma M)\longrightarrow L^{p}(\Sigma M), \qquad 1<p<3 ,$$
are invertible if $\|\psi\|_{L^3(\Sigma \tilde{\Omega})}$, $\|u\|_{L^{4}(\tilde{\Omega} \cap \partial M)}$ and $\|u\|_{L^6(\tilde{\Omega})}$ are small, which is possible by taking $\tilde{\Omega}$ even smaller. Therefore there is a unique solution $(v,\phi)$ with $v\in W^{2,q}(M)$ and $\phi\in W^{1,p}(\Sigma M)$ to the equations
$$\left\{\begin{array}{ll}
L_g v-\eta|\psi|^2v & = -u\Delta_g \rho-2g(\nabla\rho,\nabla u) \notag ,\\
B_{g}v-\eta u^{2}v &= u\frac{\partial \rho}{\partial \nu},\notag\\
D_g \phi-\eta u^2 \phi & =\nabla\rho\cdot\psi \notag, \\
\mathbb{B}^{+}\phi &=0\notag
\end{array}
\right.
$$
for $1<q<\frac{3}{2}$ and $1<p<3$.\\
Now we consider another pair of maps, defined as follows.
$$\left\{ \begin{array}{ll}
\tilde P_1:H^{1}(M)\longrightarrow H^{-1}(M)\times H^{-\frac{1}{2}}(\partial M),\\
&\\
 \tilde P_1(\tilde v)=(\eta|\psi|^2 \tilde v, b\eta u^{2}\tilde v)
\end{array} \right. \text{ and }\quad \left\{\begin{array}{ll}
\tilde P_2:H^{\frac{1}{2}}_{+}(\Sigma M)\longrightarrow H^{-\frac{1}{2}}(\Sigma M),\\
&\\
\tilde P_2(\tilde \phi)=\eta u^2 \tilde \phi.
\end{array}
\right.$$
Again, by H\"{o}lder's inequality and Sobolev embedding $L^{\frac{3}{2}}(\Sigma M)\hookrightarrow H^{-\frac{1}{2}}(M)$, it is clear that $\tilde{P}_{2}$ is well defined and $$\|\tilde{P}_2\|_{op}\leq C \|u\|^2_{L^6(\Omega)}.$$
A more detailed argument is needed for $\tilde{P}_{1}$. Indeed, we recall that we have the continuous injections $L^{\frac{6}{5}}(M)\hookrightarrow H^{-1}(M)$ and $L^{\frac{4}{3}}(\partial M)\hookrightarrow H^{-\frac{1}{2}}(\partial M)$. Therefore there holds
\begin{align}
\|\eta |\psi|^{2}\tilde{v}\|_{H^{-1}(M)}&\leq C \|\eta |\psi|^{2}\tilde{v}\|_{L^{\frac{6}{5}}(M)}\notag\\
&\leq C\|\psi\|_{L^{3}(\Omega)}^{2}\|\tilde{v}\|_{L^{6}(M)}\notag\\
&\leq C\|\psi\|_{L^{3}(\Omega)}^{2}\|\tilde{v}\|_{H^{1}(M)}.
\end{align}
For the second term of $\tilde{P}_{1}$ we have
\begin{align}
\|\eta u^{2} \tilde{v}\|_{H^{-\frac{1}{2}}(\partial M)}&\leq C \|\eta u^{2} \tilde{v}\|_{L^{\frac{4}{3}}(\partial M)}\notag\\
&\leq C\|u\|_{L^{4}(\Omega \cap \partial M)}^{2}\|\tilde{v}\|_{L^{4}(\partial M)}\notag\\
&\leq  C\|u\|_{L^{4}(\Omega \cap \partial M)}^{2}\|\tilde{v}\|_{H^{1}(M)}.
\end{align}
Hence, $\tilde{P}_{1}$ is well defined and
$$\|\tilde P_1\|_{op}\leq C \Big( \|\psi\|^2_{L^3(\Sigma \Omega)}+b\|u\|_{L^{4}(\Omega \cap \partial M)}^{2}\Big).$$
As before, the operators
$$(L_g-\eta|\psi|^2, B_{g}-\eta u^{2}):H^{1}(M) \longrightarrow H^{-1}(M)\times H^{-\frac{1}{2}}(\partial M) ,$$
$$D_g-\eta u^2:H^{\frac{1}{2}}(\Sigma M)\longrightarrow H^{-\frac{1}{2}}(\Sigma M) ,$$
are invertible if $\|\psi\|_{L^3(\Sigma \Omega)}$, $\|u\|_{L^6(\Omega)}$ and $\|u\|_{L^{4}(\Omega \cap \partial M)}$ are small; therefore there are unique solutions $\tilde v\in H^{1}(M)$ and $\tilde \phi\in H^{\frac{1}{2}}(\Sigma M)$ to the equations
$$\left\{\begin{array}{ll}
L_g \tilde v-\eta|\psi|^2\tilde v =-u\Delta_g \rho-2g(\nabla\rho,\nabla u) \notag ,\\
B_{g}(\tilde v)-\eta b u^{2} \tilde v= u\frac{\partial \rho}{\partial \nu}\notag\\
D_g \tilde \phi-\eta u^2 \tilde \phi =\nabla\rho\cdot\psi  \notag\\
\mathbb{B}^{+}(\tilde \phi)=0 \notag
\end{array}
\right.$$

Moreover, since
$$ W^{2,q}(M) \hookrightarrow H^{1}(M), \quad \frac{6}{5}\leq q <\frac{3}{2} ,\text{ and } W^{1,p}_{+}(\Sigma M) \hookrightarrow H^{\frac{1}{2}}_{+}(\Sigma M), \quad \frac{3}{2}\leq p < 3 ,$$
one has, using the uniqueness, $\tilde v=v=\rho u$ and $\tilde \phi=\phi=\rho \psi$, under the above conditions on $q$ and $p$. Now, since $\rho$ and $\eta$ are smooth functions with arbitrary small supports, we have that $u\in  W^{2,q}(M)$ and $\psi \in W^{1,p}_{+}(\Sigma M)$, provided $\frac{6}{5}\leq q <\frac{3}{2}$ and $\frac{3}{2}\leq p < 3$. Therefore, by the Sobolev embedding, we get that $u\in L^q(M)$ and $\psi \in L^p(\Sigma M)$, for any $1< q,p <\infty$; and then by plugging them in the initial equations, we have that $u\in  W^{2,q}(M)$ and $\psi \in W^{1,p}_{+}(\Sigma M)$, for any $1< q,p <\infty$, by the elliptic regularity estimates ( see \cite[Theorem~15.2]{ADN} for the scalar part and \cite[Theorem~4.1]{CJSZ} for the spinorial part). Once more, by the Sobolev embedding for the H\"{o}lder spaces, we have that there exist $0<\alpha,\beta<1$ such that $u\in C^{0,\alpha}(M)$ and $\psi \in C^{0,\beta}(\Sigma M)$; finally by the elliptic regularity again and bootstrapping, we get $u\in C^{\infty}(M)$ and $\psi \in C^{\infty}(\Sigma M)$.
\end{proof}

\section{Bubbling of PS sequences}\label{sec:PS}
This section is devoted to the proof of Theorem \ref{first}, which requires various intermediate results. We start by showing that Palais-Smale sequences are bounded and then look at the behaviour near concentration points. This allows us to distinguish the different blowup behaviours and bubbling profiles, leading to the energy quantization.
\smallskip

We will be using notations and properties that we introduced in Section \ref{sec:spinorspaces}. In particular, since in the statement of Theorem \ref{first} we assume the positivity of the Yamabe invariant of the manifold, there will be no harmonic spinors.

\begin{proposition}\label{propPS}
Every (PS) sequence for $E$ is bounded.
\end{proposition}

\begin{proof}
Let $(u_{n},\psi_{n})_{n\in \N} \in H^{1}(M)\times H_{+}^{\frac{1}{2}}(\Sigma M)=:{\mathcal H}$ be a (PS) sequence for $E^{b}$, that is
$$E^{b}(u_{n},\psi_{n})\to c, \qquad dE^{b}(u_{n},\psi_{n})\to 0,  \text{ in }  H^{-1}(M)\times H^{-\frac{1}{2}}(\partial M)\times H^{-\frac{1}{2}}_{+}(\Sigma M) .$$
Therefore, there exists a sequence $(\varepsilon_n, r_{n}, \delta_n) \in  H^{-1}(M)\times H^{-\frac{1}{2}}(\partial M)\times H_{+}^{-\frac{1}{2}}(\Sigma M)$ such that

\begin{equation}\label{equ}
L_{g} u_{n}=|\psi_{n}|^{2}u_{n}+\varepsilon_{n} ,
\end{equation}
\begin{equation}\label{equb}
\frac{\partial}{\partial \nu} u_{n}=bu^{3}_{n}+r_{n} ,
\end{equation}

\begin{equation}\label{eqp}
D_g\psi_{n}=|u_{n}|^{2}\psi_{n}+\delta_{n} ,
\end{equation}

with
$$\varepsilon_{n}\to 0,  \text{ in } H^{-1}(M) \quad , \quad  \delta_{n}\to 0,  \text{ in } H^{-\frac{1}{2}}(M) \quad \text{ and } \quad r_{n}\to 0 \in H^{-\frac{1}{2}}(\partial M).$$
We let $z_{n}=(u_{n},\psi_{n}) \in \mathcal{H}$, then
$$\langle dE^{b}(z_{n}),z_{n}\rangle =\int_{M}|\nabla u_{n}|^{2}+\frac{R}{8}u^{2} +\int_{M}\langle D\psi_{n},\psi_{n}\rangle -2\int_{M}u_{n}^{2}|\psi_{n}|^{2} -b\int_{\partial M}u_{n}^{4}.$$
$$2E^{b}(z_{n})-\langle dE^{b}(z_{n}),z_{n} \rangle = \int_{M}|u_{n}|^{2}|\psi_{n}|^{2}\ dv_{g}+\frac{b}{2}\int_{\partial M}u_{n}^{4}\ d\sigma_{g} .$$
Hence
\begin{equation}\label{lev}
\frac{b}{2}\int_{\partial M}u_{n}^{4}\ dv_{g}+\int_{M}|u_{n}|^{2}|\psi_{n}|^{2}\ dv_{g}=2c+o(\|z_{n}\|).
\end{equation}
Multiplying \eqref{equ} by $u_{n}$ and integrating we have (combining with multiplying \eqref{equb} by $u_{n}$ and integrating)
$$\int_{M}|\nabla u_{n}|^{2} -\int_{\partial_{M}}bu_{n}^{4}+\int_{M}\frac{R}{8}u_{n}^{2}=\int_{M}u_{n}^{2}|\psi_{n}|^{2}+o(\|u_{n}\|).$$
Thus,

$$\|u_{n}\|^{2}=\int_{M}|u_{n}|^{2}|\psi_{n}|^{2}\ dv_{g}+\int_{\partial_{M}}bu_{n}^{4}\ d\sigma_{g}+o(\|u_{n}\|),$$
hence
\begin{equation}
\|u_{n}\|^{2}=2c+o(\|z_{n}\|) .
\end{equation}
Now multiplying \eqref{eqp} by $\psi_{n}^{+}=P^{+}(\psi_n)$, we find
\begin{align}
\|\psi_{n}^{+}\|^{2}&\leq C\int_{M}|u_{n}|^{2}|\psi_{n}||\psi_{n}^{+}|\ dv_{g} +o(\|\psi_{n}^{+}\|)\notag \\
&\leq C\left(\int_{M}|u_{n}|^{2}|\psi_{n}|^{2}\ dv_{g}\right)^{\frac{1}{2}} \left(\int_{M}|u_{n}|^{2}|\psi^{+}_{n}|^{2}\ dv_{g}\right)^{\frac{1}{2}}+o(\|\psi^{+}_{n}\|)\notag \\
&\leq C\Big(2c+o(\|z_{n}\|)\Big)^{\frac{1}{2}}\|u_{n}\|_{L^{6}}\|\psi_{n}^{+}\|_{L^{3}}+o(\|\psi_{n}\|)\notag \\
&\leq C\Big(2c+o(\|z_{n}\|)\Big)\|\psi_{n}\| +o(\|\psi_{n}\|)\notag.
\end{align}
Similarly, we have for $\psi_{n}^{-}=P^{-}(\psi_n)$ that
$$\|\psi_{n}^{-}\|^{2}\leq C\big(2c+o(\|z_{n}\|)\big)\|\psi_{n}\| +o(\|\psi_{n}\|).$$
Hence,
$$\|\psi_{n}\|\leq C\big(2c+o(\|z_{n}\|)\big)+o(1),$$
so that
$$\|z_{n}\|\leq C+o(\|z_{n}\|)$$
and $\|z_{n}\|$ is bounded.
\end{proof}

\noindent
As a corollary from the previous proposition, we have that up to a subsequence, $z_{n}\rightharpoonup z_{\infty}=(u_{\infty},\psi_{\infty})$ weakly in $\mathcal{H}$, also $u_{n}\to u_{\infty}$ in $L^{p}(M)$ for $p<6$ and $\psi_{n}\to\psi$ in $L^{q}(M)$ for $q<3$ and $u_{n}\to u_{\infty}\in L^{r}(\partial M)$ for $r<4$.
We claim that $z_{\infty}$ is a week solution to \eqref{eq1}. Indeed, let $z_0=(u_{0},\psi_{0})\in \mathcal{H}$. Since $(z_{n})$ is a (PS) sequence for $E^{b}$, we have
$$\int_{M}u_{0} L_{g}u_{n}\  dv_{g}=\int_{M}|\psi_{n}|^{2}u_{n}u_{0}\ dv_{g} +o(1),$$
but $|\psi_{n}|^{2}\in L^{\frac{3}{2}}(M)$ and $u_{n}\in L^{6}(M)$, thus $u_{n}|\psi_{n}|^{2}$ converges weakly to $u_{\infty}|\psi_{\infty}|^{2}$ in $L^{\frac{6}{5}}$, hence
$$\int_{M}|\psi_{n}|^{2}u_{n}u_{0}\ dv_{g}\to \int_{M}|\psi_{\infty}|^{2}u_{\infty}u_{0}\ dv_{g} .$$
Also, by the weak convergence we have that
$$\int_{M}  u_{0}L_{g}u_{n}\  dv_{g}\to \int_{M} u_{0}L_{g}u_{\infty}\  dv_{g} .$$
Therefore,
$$L_g u_{\infty}=|\psi_{\infty}|^{2}u_{\infty},$$
and similarly it holds
$$D_{g}\psi_{\infty}=|u_{\infty}|^{2}\psi_{\infty},$$
and
$$\frac{\partial u_{\infty}}{\partial \nu}=bu_{\infty}^{3}.$$

For now, We let $v_{n}=u_{n}-u_{\infty}$ and $\phi_{n}=\psi_{n}-\psi_{\infty}$, then we have the following
\begin{lemma}\label{lemsplit}
Let $h_{n}=(v_{n},\phi_{n})$, then
$$E^{b}(h_{n})=E^{b}(z_{n})-E^{b}(z_{\infty})+o(1)$$
and
$$dE^{b}(h_{n})\to 0, \text{ in } H^{-1}(M)\times H^{-\frac{1}{2}}(\partial M)\times H^{-\frac{1}{2}}_{+}(\Sigma M) .$$
\end{lemma}
\begin{proof}
\begin{align}
2E^{b}(z_{n})&=\int_{M}|\nabla(v_{n}+u_{\infty})|^{2}+\frac{R}{8}(v_{n}+u_{\infty})^{2}\ dv_{g} +\int_{M}\langle D_g(\phi_{n}+\psi_{\infty}),\phi_{n}+\psi_{\infty}\rangle\ dv_{g}\notag\\
&\quad -\int_{M}| v_{n}+u_{\infty}|^{2}|\phi_{n}+\psi_{\infty}|^{2}\ dv_{g}-\frac{b}{2}\int_{\partial M}(v_{n}+u_{\infty})^{4} \ d\sigma_{g}\notag\\
&=2E^{b}(h_{n})+2E^{b}(z_{\infty})+2\langle dE^{b}(z_{\infty}),h_{n}\rangle -\int_{M}|v_{n}|^{2}|\psi_{\infty}|^{2}+2|v_{n}|^{2}\langle \phi_{n},\psi_{\infty}\rangle\notag\\
&\quad +|u_{\infty}|^{2}|\phi_{n}|^{2}+2|\phi_{n}|^{2}v_{n}u_{\infty} +4v_{n}u_{\infty}\langle \psi_{\infty}, \phi_{n}\rangle dv_{g} \notag \\
&\quad -\frac{b}{2}\int_{\partial M} 4 v_{n}^{3}u_{\infty}+6v_{n}^{2}u_{\infty}^{2}\ d\sigma_{g}\notag
\end{align}
We first notice that, since $dE^{b}(z_{\infty})=0$, one can focus on the remaining terms. By the regularity result in Theorem \ref{thmreg}, we have that $u_{\infty}\in C^{2,\alpha}(M)$ and $\psi_{\infty}\in C^{1,\beta}(\Sigma M)$. Since $v_{n}\to 0$ strongly in $L^{2}(M)$ and $\phi_{n}\to 0$ in $L^{2}(\Sigma M)$, we have that
$$ -\int_{M}|v_{n}|^{2}|\psi_{\infty}|^{2}+2|v_{n}|^{2}\langle \phi_{n},\psi_{\infty}\rangle+|v_{\infty}|^{2}|\phi_{n}|^{2}+2|\psi_{\infty}|^{2}u_{\infty}v_{n}-4v_{n}u_{\infty}\langle \psi_{\infty}, \phi_{n}\rangle\ dv_{g} \to 0 .$$
The term $\int_{M}2|\phi_{n}|^{2}v_{n}u_{\infty}dv_{g}$ also converges to zero, as we have that $\phi_{n}\to 0$ in $L^{\frac{5}{2}}(\Sigma M)$ and $v_{n}\to 0$ in $L^{5}(M)$. In order to conclude with the energy splitting, we need to deal with boundary terms. Observe that $v_{n}\to 0$ in $L^{p}(\partial M)$ for all $p<4$. Hence,
$$\int_{\partial M} 4 v_{n}^{3}u_{\infty}+6v_{n}^{2}u_{\infty}^{2}\ d\sigma_{g}\to 0.$$
Therefore, we conclude that
$$E^{b}(z_{n})=E^{b}(h_{n})+E^{b}(z_{\infty})+o(1),$$
thus proving the energy splitting. \\
For the gradient part $dE^{b}$, we denote by $d_{u}E^{b}$ and $d_{\psi}E^{b}$ the scalar and the spinorial components respectively. We have then,
\begin{align}
d_{u}E^{b}(h_{n})& =d_{u}E^{b}(u_{n},\psi_{n})- d_{u}E^{b}(u_{\infty},\psi_{\infty})\notag\\
&+\Big(|\psi_{n}|^{2}u_{\infty}-|\psi_{\infty}|^{2}u_{n}+2\langle \psi_{n},\psi_{\infty}\rangle v_{n}, -3(u_{n}^{2}u_{\infty}-u_{n}u_{\infty}^{2})|_{\partial M}\Big) .\notag
\end{align}
But again, $d_{u}E^{b}(u_{\infty},\psi_{\infty})=0$ and since $\psi_{n}\to \psi_{\infty}$ in $L^{\frac{12}{5}}(\Sigma M)$ and $u_{n}\to u_{\infty}$ in $L^{\frac{6}{5}}(M)$ we have that
$$|\psi_{n}|^{2}u_{\infty}-|\psi_{\infty}|^{2}u_{n}\to 0$$
in $L^{\frac{6}{5}}(M)$ hence in $H^{-1}(M)$. We also have that $v_{n}\to 0$ in $L^{\frac{12}{5}}(M)$ and $\psi_{n}\to \psi_{\infty}$ in $L^{\frac{12}{5}}(\Sigma M)$, thus
$\langle \psi_{n},\psi_{\infty}\rangle v_{n}\to 0$ in $L^{\frac{6}{5}}(M)$, thus in $H^{-1}(M)$. Similarly, $u_{n}\to u_{\infty} $ in $L^{\frac{8}{3}}(\partial M)$, therefore $u_{n}^{2}u_{\infty}-u_{n}u_{\infty}^{2}\to 0$ in $L^{\frac{4}{3}}(\partial M)\subset H^{-\frac{1}{2}}(\partial M)$.

Therefore,
$$d_{u}E^{b}(h_{n})=o(1) \text{ in } H^{-1}(M)\times H^{-\frac{1}{2}}(\partial M).$$
The spinorial part can be handled the same way as in the case of manifolds without boundary, see \cite{MV3}.
\end{proof}

\noindent
From now on, we will assume that our (PS) sequence $z_{n}=(u_{n},\psi_{n})$, converges weakly to zero in $H^{1}(M)\times H^{\frac{1}{2}}_+(\Sigma M)$ and strongly in $L^{p}(M)\times L^{q}(\Sigma M)$, for $p<6$ and $q<3$.\\
We assume that $z_{n}$ {\em does not} converge strongly to zero in $H^{1}(M)\times H^{\frac{1}{2}}(\Sigma M)$, since otherwise the (PS) condition would be satisfied.  We will see below that this violation of $(PS)$ happens when a certain concentration phenomena occurs.
This concentration can either be on the boundary or in the interior of $M$. In order to describe this phenomena, denoting by $B_{r}(x)$ the geodesic ball with center in $x\in M$ and radius $r$, we define the following sets, for a given $\epsilon_0>0$:
$$\Sigma_{1}=\left\{x\in \overset{\circ}{M}; \liminf_{r\to 0} \liminf_{n\to\infty}\int_{B_{r}(x)}|u_{n}|^{6}dv_{g}\geq\epsilon_{0}\right\} ,$$
$$\Sigma_{2}=\left\{x\in \overset{\circ}{M}; \liminf_{r\to 0} \liminf_{n\to\infty}\int_{B_{r}(x)}|\psi_{n}|^{3}dv_{g}\geq\epsilon_{0}\right\} ,$$
$$\Sigma_{3}=\left\{x\in \overset{\circ}{M}; \liminf_{r\to 0} \liminf_{n\to\infty}\int_{B_{r}(x)}|u_{n}|^{2}|\psi_{n}|^{2}dv_{g}\geq \epsilon_{0}\right\} ,$$

$$\tilde{\Sigma}_{1}=\left\{x\in \partial M; \liminf_{r\to 0} \liminf_{n\to\infty}\int_{B_{r}(x)}|u_{n}|^{6}\ dv_{g}+b\int_{B_{r}(x)\cap \partial M}|u_{n}|^{4}\ d\sigma_{g}\geq  \epsilon_{0}\right\},$$
$$\tilde{\Sigma}_{2}=\left\{x\in \partial M; \liminf_{r\to 0} \liminf_{n\to\infty}\int_{B_{r}(x)}|\psi_{n}|^{3}dv_{g}\geq\epsilon_{0}\right\} ,$$
$$\tilde{\Sigma}_{3}=\left\{x\in \partial M; \liminf_{r\to 0} \liminf_{n\to\infty}\int_{B_{r}(x)}|u_{n}|^{2}|\psi_{n}|^{2}\ dv_{g}+\frac{b}{4}\int_{B_{r}(x_{0})\cap \partial M}|u_{n}|^{4}\ d\sigma_{g}\geq \epsilon_{0}\right\} ,$$

$$\tilde\Sigma_{4}=\left\{x\in \partial M; \liminf_{r\to 0} \liminf_{n\to\infty}\int_{B_{r}(x)\cap \partial M}|u_{n}|^{4} d\sigma_{g}\geq \epsilon_{0}\right\} .$$

The interior concentration points in $\Sigma_{1}$, $\Sigma_{2}$ and $\Sigma_{3}$ are relatively understood and were studied in \cite{MV3}. Indeed, the following result holds:
\begin{proposition}
There exists $\epsilon_{0}>0$ depending on $M$, such that if $x_{0}\in \overset{\circ}{M}$ and $x_{0}\not \in \Sigma_{1}\cap \Sigma_{2} \cap \Sigma_{3}$, then there exists $r>0$ such that $z_{n}\to 0$ in $H^{1}(B_{r}(x_{0}))\times H^{\frac{1}{2}}(\Sigma B_{r}(x_{0}))$.
In particular, one has that $\Sigma_{1}=\Sigma_{2}=\Sigma_{3}$.
\end{proposition}

For the boundary concentration points, the situation is a  bit different.
\begin{lemma}
There exists $\epsilon_{0}>0$ such that $\tilde{\Sigma}_{2}\subset \tilde{\Sigma}_{1}=\tilde{\Sigma}_{3}$. Moreover, if $x_{0}\in \partial M$ and $x_{0}\not \in \tilde{\Sigma}_{1}$ then there exists $r>0$ such that $z_{n}\to 0$ in $H^{1}(B_{r}(x_{0}))\times H^{\frac{1}{2}}(\Sigma B_{r}(x_{0}))$. In addition, $\tilde{\Sigma}_{2}\setminus \tilde\Sigma_{4}=\tilde{\Sigma}_{3}\setminus\tilde\Sigma_{4}=\tilde{\Sigma}_{1}\setminus \tilde\Sigma_{4}$.
\end{lemma}

Before we proceed to the proof of the previous lemma, we provide a clarification on the different concentration sets. Mainly, $\tilde{\Sigma}_{1}$ and $\tilde{\Sigma}_{3}$ represent either a semi-interior blow-up (controlled by the $L^{6}(M)$-norm), or a pure boundary blow-up (controlled by the $L^{4}(\partial M)$-norm), of the scalar part $u$. $\tilde{\Sigma}_{2}$ represents semi-interior blow up of the spinorial part and finally $\tilde{\Sigma}_{4}$ represents the pure boundary blow up of $u$. The main difference here, compared to\cite{MV3}, is that there can be a concentration in $u$ while $\psi$ is being controlled. This leads for instance to a limit equation that is purely scalar, namely, the equation of prescribing zero scalar curvature and constant mean curvature on the sphere. The other alternative is a semi-interior concentration of $u$ while the pure boundary behaviour is controlled. Again, this leads to \eqref{eq:bubblessphere} that we will study in depth in section \ref{sec:groundstates}.

\begin{proof}
We will prove this result by contradiction, by assuming that for every $\epsilon>0$, there exists $x_{0}\not \in \tilde{\Sigma}_{1}$, such that for every $r>0$, $z_{n}\not \to 0$ in $H^{1}(B_{r}(x_{0}))\times H^{\frac{1}{2}}(\Sigma B_{r}(x_{0}))$.\\
Given $\epsilon>0$, there exists $r>0$ such that $\int_{B_{4r}(x_{0})}|u_{n}|^{6}\ dv_{g}+b\int_{B_{4r}(x_{0})\cap \partial M}|u_{n}|^{4}\ d\sigma_{g}<\epsilon$. We first estimate the $\psi$ component. That is, we consider a smooth cut off function $\eta$ supported on $B_{4r}(x_{0})$ and equals to $1$ on $B_{2r}(x_{0})$, then by (\ref{eqp}) we have:
\begin{align}
D_g(\eta\psi_{n})&=\eta D_g\psi_{n}+\nabla \eta \cdot \psi_{n}\notag \\
&=\eta |u_{n}|^{2}\psi_{n}+\nabla \eta \cdot \psi_{n}+\eta\delta_{n}, \notag
\end{align}
where $\|\delta_{n}\|_{H^{-\frac{1}{2}}}\to 0$. Hence, recalling that $\ker(D_g)=\{0\}$, we have
\begin{align}
\|\eta\psi_{n}\|_{H^{\frac{1}{2}}}&\leq C_{1}\|\eta |u_{n}|^{2}\psi_{n}+\nabla \eta \cdot \psi_{n}+\eta\delta_{n}\|_{H^{-\frac{1}{2}}}\notag\\
&\leq C_{2}\left(\|\eta|u_{n}|^{2}\psi_{n}\|_{L^{\frac{3}{2}}}+ \|\psi_{n}\|_{L^{\frac{3}{2}}}+\|\delta_{n}\|_{H^{-\frac{1}{2}}}\right) .\notag
\end{align}
Since $\|\psi_{n}\|_{L^{\frac{3}{2}}}\to 0$, it remains to estimate
\begin{align}
\|\eta |u_{n}|^{2}\psi_{n}\|_{L^{\frac{3}{2}}}&\leq \|u_{n}\|_{L^{6}(B_{4r}(x_{0}))}^{2}\|\eta \psi_{n}\|_{L^{3}} \leq C_{3}\epsilon^{\frac{1}{3}} \|\eta \psi_{n}\|_{H^{\frac{1}{2}}} .\notag
\end{align}
Hence, taking $C_{3}\epsilon^{1/3}<\frac{1}{2}$, we deduce that $\eta \psi_{n}\to 0$ in $H^{\frac{1}{2}}$, yielding
$$\|\eta \psi_{n}\|_{L^{3}}\to 0.$$
In particular, $x_{0}\not \in \tilde{\Sigma}_{2}$, so that $\tilde\Sigma_2\subset\tilde\Sigma_1$. We estimate next the scalar component $u$. So, consider a smooth cut off function $\rho$ supported on $B_{2r}(x_{0})$ and equals to $1$ on $B_{r}(x_{0})$, then by (\ref{equ}) we have
\begin{align}
L_{g} (\rho u_{n})&=\rho L_{g} u_{n}- u_{n}\Delta_{g} \rho-2g(\nabla \rho, \nabla u_{n})\notag \\
&=\rho|\psi_{n}|^{2}u_{n}- u_{n}\Delta_{g} \rho-2\nabla \rho \cdot \nabla u_{n}+\rho \varepsilon_{n}\notag
\end{align}
and
$$\frac{\partial(\rho u_{n})}{\partial \nu}=\rho b u_{n}^{3}+\frac{\partial \rho}{\partial \nu}u_{n}+\rho r_{n}$$
where $\varepsilon_{n}\to 0$ in $H^{-1}(M)$ and $r_{n}\to 0$ in $H^{-\frac{1}{2}}(\partial M)$. From elliptic estimates now we have that
\begin{align}
\|\rho u_{n}\|_{H^{1}}&\leq C\Big(\| \rho|\psi_{n}|^{2 }u_{n}- u_{n}\Delta \rho+2g(\nabla \rho , \nabla u_{n})+\rho \varepsilon_{n}\|_{H^{-1}}+\|\rho b u_{n}^{3}+\frac{\partial \rho}{\partial \nu}u_{n}+\rho r_{n}\|_{H^{-\frac{1}{2}}(\partial M)}\Big) \notag\\
&\leq C\Big(\|\rho|\psi_{n}|^{2}u_{n}\|_{L^{\frac{6}{5}}}+\|u_{n}\Delta \rho\|_{L^{\frac{6}{5}}}+2\|g(\nabla \rho , \nabla u_{n})\|_{H^{-1}}+b\|\rho u_{n}^{3}\|_{L^{\frac{4}{3}}(\partial M)}\notag\\
&\qquad+\|\frac{\partial \rho}{\partial \nu}u_{n}\|_{L^{\frac{4}{3}}(\partial M)}+\|\rho \varepsilon_{n}\|_{H^{-1}}+\|\rho r_{n}\|_{H^{-\frac{1}{2}}(\partial M)}\Big).\notag
\end{align}
We estimate the terms $\|\rho|\psi_{n}|^{2}u_{n}\|_{L^{\frac{6}{5}}}$ and $\|\rho u_{n}^{3}\|_{L^{\frac{4}{3}}(\partial M)}$:
\begin{align}
\|\rho|\psi_{n}|^{2}u_{n}\|_{L^{\frac{6}{5}}}&\leq \| \eta \psi_{n}\|^{2}_{L^{3}}\|\rho u_{n}\|_{L^{6}} \leq C_{1}\| \eta \psi_{n}\|^{2}_{H^{\frac{1}{2}}}\|\rho u_{n}\|_{H^{1}} . \notag
\end{align}
and
$$\|\rho u_{n}^{3}\|_{L^{\frac{4}{3}}(\partial M)}\leq C\|\rho^{\frac{1}{3}} u_{n}\|_{L^{4}(\partial M)}^{3}\leq C_{2}\|\rho u_{n}\|_{H^{1}(M)}\|u_{n}\|_{L^{4}(B_{2r}\cap \partial M)}^{2}.$$
From the previous estimates and the estimate on the spinorial component $\psi_{n}$, we know that for $n$ big enough we have $CC_{1}\| \eta \psi_{n}\|^{2}_{H^{\frac{1}{2}}}<\frac{1}{2}$. Similarly, since $x_{0}\not \in \tilde{\Sigma}_{1}$, we can choose $r$ so that $CC_{2}\|u_{n}\|_{L^{4}(B_{2r}\cap \partial M)}^{2}\leq CC_{2}\epsilon <\frac{1}{2}$. Thus we have that
$$\|\rho u_{n}\|_{H^{1}}\leq C\Big( \|u_{n}\Delta \rho\|_{L^{\frac{6}{5}}}+2\|g(\nabla \rho , \nabla u_{n})\|_{H^{-1}}+\|\frac{\partial \rho}{\partial \nu}u_{n}\|_{L^{\frac{4}{3}}(\partial M)}+\|\rho \varepsilon_{n}\|_{H^{-1}}+\|\rho r_{n}\|_{H^{-\frac{1}{2}}(\partial M)}\Big).$$
Now clearly
$$\|u_{n}\Delta \rho\|_{L^{\frac{6}{5}}(M)}\leq C\|u_{n}\|_{L^{2}(M)}.$$
Similarly, $$\|\frac{\partial \rho}{\partial \nu}u_{n}\|_{L^{\frac{4}{3}}(\partial M)}\leq C\|u_{n}\|_{L^{2}(\partial M)}.$$
Next, the term
$$\|g(\nabla \rho , \nabla u_{n})\|_{H^{-1}}\leq \sup_{k\in H^{1};\|k\|_{H^{1}}\leq 1}\left|\int_{M}g(\nabla\rho,\nabla u_{n})kdv_{g}\right| .$$
But
\begin{align}
\left|\int_{M} g(\nabla \rho , \nabla u_{n})kdv_{g}\right|&\leq \left|\int_{M}u_{n}\left(k\Delta \rho+g(\nabla \rho, \nabla k ) \right)dv_{g}\right|+\left|\int_{\partial M} u_{n} \frac{\partial \rho}{\partial \nu} k \ d\sigma_{g}\right|\notag\\
&\leq C\|k\|_{H^{1}}\Big(\|u_{n}\|_{L^{2}(M)}+\|u_{n}\|_{L^{2}(\partial M)}\Big) \to 0 .\notag
\end{align}
Hence $\rho u_{n}$ converges to zero in $H^{1}(B_{r}(x_{0}))$ and this leads to a contradiction.\\
Now, we assume that $x_{0}\not \in \tilde{\Sigma}_{2}\cup \tilde \Sigma_{4}$. Then, there exists $r>0$ such that $\int_{B_{4r}(x_{0})}|\psi_{n}|^{3}dv_{g}<\epsilon_{0}$. Then again we compute
\begin{align}
L_{g} (\rho u_{n})&=\rho (L_{g} u_{n})- u_{n}\Delta_{g} \rho -2g(\nabla \rho , \nabla u_{n})\notag \\
&=\rho|\psi_{n}|^{2}u_{n}- u_{n}\Delta_{g} \rho-2g(\nabla \rho , \nabla u_{n})+\rho \varepsilon_{n} ,\notag
\end{align}
where $\varepsilon_{n}\to 0$ in $H^{-1}$. From elliptic estimates now we have that
\begin{align}
\|\rho u_{n}\|_{H^{1}}&\leq C\|\rho|\psi_{n}|^{2}u_{n}- u_{n}\Delta_g \rho-2g(\nabla \rho , \nabla u_{n})+\rho \varepsilon_{n}\|_{H^{-1}}\|\rho b u_{n}^{3}+\frac{\partial \rho}{\partial \nu}u_{n}+\rho r_{n}\|_{H^{-\frac{1}{2}}(\partial M)}\Big) \notag\\
&\leq C\Big(\|\rho|\psi_{n}|^{2}u_{n}\|_{L^{\frac{6}{5}}}+\|u_{n}\Delta \rho\|_{L^{\frac{6}{5}}}+2\|g(\nabla \rho , \nabla u_{n})\|_{H^{-1}}+b\|\rho u_{n}^{3}\|_{L^{\frac{4}{3}}(\partial M)}\notag\\
&\qquad+\|\frac{\partial \rho}{\partial \nu}u_{n}\|_{L^{\frac{4}{3}}(\partial M)}+\|\rho \varepsilon_{n}\|_{H^{-1}}+\|\rho r_{n}\|_{H^{-\frac{1}{2}}(\partial M)}\Big).\notag
\end{align}
Again, we estimate
\begin{align}
\|\rho|\psi_{n}|^{2}u_{n}\|_{L^{\frac{6}{5}}}&\leq \| \psi_{n}\|^{2}_{L^{3}}\|\rho u_{n}\|_{L^{6}}\leq C_{1}\epsilon_{0}^{\frac{2}{3}}\|\rho u_{n}\|_{H^{1}} .\notag
\end{align}
and
$$b\|\rho u_{n}^{3}\|_{L^{4}{3}(\partial M)}\leq C_{2}\epsilon_{0}^{\frac{1}{2}}\|\rho u_{n}\|_{H^{1}}.$$
Taking $\epsilon_{0}$ even smaller if necessary, so that $C_{1}C\epsilon_{0}^{\frac{2}{3}}<\frac{1}{2}$ and $C_{2}C\epsilon_{0}^{\frac{1}{2}}<\frac{1}{2}$, we have that
\begin{align}
\|\rho u_{n}\|_{H^{1}} & \leq C_{1}\Big(\|u_{n}\Delta_g \rho\|_{L^{\frac{6}{5}}}+2\|g(\nabla \rho , \nabla u_{n})\|_{H^{-1}} +\|\rho\varepsilon_{n}\|_{H^{-1}} \notag\\
&+\|\frac{\partial \rho}{\partial \nu}u_{n}\|_{L^{\frac{4}{3}}(\partial M)}+\|\rho r_{n}\|_{H^{-\frac{1}{2}}(\partial M)}\Big),\notag
\end{align}
and as in the previous case we have that
$$\|u_{n}\Delta_g \rho\|_{L^{\frac{6}{5}}}+2\|g(\nabla \rho , \nabla u_{n})\|_{H^{-1}} +\|\rho\varepsilon_{n}\|_{H^{-1}}+\|\frac{\partial \rho}{\partial \nu}u_{n}\|_{L^{\frac{4}{3}}(\partial M)}+\|\rho r_{n}\|_{H^{-\frac{1}{2}}(\partial M)}\to 0.$$
Hence $\|\rho u_{n}\|_{H^{1}}\to 0$. Next, we estimate the spinorial component:
\begin{align}
\|\eta\psi_{n}\|_{H^{\frac{1}{2}}}&\leq C_{1}\|\eta |u_{n}|^{2}\psi_{n}+\nabla \eta \cdot \psi_{n}+\eta \delta_{n}\|_{H^{-\frac{1}{2}}}\notag\\
&\leq C_{2}\left(\|\eta |u_{n}|^{2}\psi_{n}\|_{L^{\frac{3}{2}}}+ \|\psi_{n}\|_{L^{\frac{3}{2}}}+\|\delta_{n}\|_{H^{-\frac{1}{2}}}\right).\notag
\end{align}
But
$$\|\eta |u_{n}|^{2}\psi_{n}\|_{L^{\frac{3}{2}}}\leq \|\rho u_{n}\|_{L^{6}}^{2}\|\eta \psi_{n}\|_{H^{\frac{1}{2}}}.$$
Using the fact that $\|\rho u_{n}\|_{H^{1}}\to 0$, we have that $z_{n}\to 0$ in $H^{1}(B_{r}(x_{0}))\times H^{\frac{1}{2}}(\Sigma B_{r}(x_{0}))$. In particular, $x_{0}\not \in \tilde{\Sigma}_{1}$ and as a corollary we have $\tilde{\Sigma}_{1}\setminus \tilde{\Sigma}_{4}=\tilde{\Sigma}_{2}\setminus\tilde{\Sigma}_{4}$. \\
To finish the proof, it remains to study $\tilde{\Sigma}_{3}$. First, it is clear from H\"{o}lder's inequality that if $x_{0}\not\in \tilde{\Sigma}_{1}$, then $x_{0}\not \in \tilde{\Sigma}_{3}$. So $\tilde{\Sigma}_{3}\subset \tilde{\Sigma}_{1}$.\\
Assume that $x_{0}\not \in \tilde{\Sigma}_{3}$. Then one can pick $r>0$ so that $\int_{B_{2r}(x_{0})}|u_{n}|^{2}|\psi_{n}|^{2}dv_{g}<\epsilon_{0}$. We have then
$$\|\rho \psi_{n}\|_{H^{\frac{1}{2}}} \leq C_{2}\left(\|\rho |u_{n}|^{2}\psi_{n}\|_{L^{\frac{3}{2}}}+o(1)\right).$$
But
$$\|\rho |u_{n}|^{2}\psi_{n}\|_{L^{\frac{3}{2}}}\leq \left(\int_{B_{2r}(x_{0})}|u_{n}|^{2}|\psi_{n}|^{2}\ dv_{g}\right)^{\frac{1}{2}}\|\rho u_{n}\|_{L^{6}} ,$$
thus
$$\|\rho\psi_{n}\|_{H^{\frac{1}{2}}} \leq C\epsilon_{0}^{\frac{1}{2}}\|\rho u_{n}\|_{H^{1}}+o(1).$$
Similarly, we have for the  $u$ component,
$$\|\rho u_{n}\|_{H^{1}}\leq C \Big(\|\rho |\psi_{n}|^{2}u_{n}\|_{L^{\frac{6}{5}}}+b\|\rho u_{n}^{3}\|_{L^{\frac{4}{3}}(\partial M)}\Big)+o(1),$$
and
$$\|\rho |\psi_{n}|^{2}u_{n}\|_{L^{\frac{6}{5}}}\leq \left(\int_{B_{2r}(x_{0})}|u_{n}|^{2}|\psi_{n}|^{2}dv_{g}\right)^{\frac{1}{2}}\|\rho \psi_{n}\|_{L^{3}}. $$
Hence
$$\|\rho u_{n}\|_{H^{1}}\leq C\epsilon_{0}^{\frac{1}{2}}\|\rho \psi_{n}\|_{H^{\frac{1}{2}}}+o(1).$$
Combining both the previous inequalities we have $\rho z_{n}\to 0$ in $H^1(B_{r}(x_{0}))\times H^{\frac{1}{2}}(\Sigma B_{r}(x_{0}))$ and $x_{0}\not \in \tilde{\Sigma}_{1}$.
\end{proof}

\noindent
Based on the previous lemma, we define the sets $\mathcal{S}=\mathcal{S}_{int}\cup \mathcal{S}_{\partial}$, where  $\mathcal{S}_{int}:=\Sigma_{1}$ and $\mathcal{S}_{\partial}:=\tilde{\Sigma}_{1}$. So we can deduce
\begin{corollary}
Let $(z_{n})$ be a (PS) sequence at the level $c$. Then

\begin{itemize}
\item If $(z_{n})$ does not satisfy the (PS) condition, then
$$\mathcal{S}\not= \emptyset.$$

\item If $c<\frac{\epsilon_{0}}{2}$ then $(z_{n})$ converges strongly to zero.

\end{itemize}
\end{corollary}

\begin{proof}
The proof follows from the boundedness of the (PS) sequences. Indeed, from \eqref{lev} we have
$$\int_{M}|u_{n}|^{2}|\psi_{n}|^{2}dv_{g}+\frac{b}{2}\int_{\partial M}u_{n}^{4}\ d\sigma_{g}=2c+o(1).$$
Hence if $2c<\epsilon_0$, we have that for $n$ large enough,
$$\int_{M}|u_{n}|^{2}|\psi_{n}|^{2}\ dv_{g}+\frac{b}{2}\int_{\partial M}u_{n}^{4}\ d\sigma_{g}<\epsilon_0.$$
Thus $z_{n}\to 0$ strongly in $H^{1}(M)\times H^{\frac{1}{2}}_{+}(\Sigma M)$.
\end{proof}

We recall that if $\overline{g}=f^{4}g$ then
$$B_{\overline{g}}(f^{-1}u)=f^{-3}B_{g}(u),$$
If $\psi \in \Sigma_{g}M$ and $F(\psi) \in \Sigma_{\overline{g}}M$, where $F$ is the isomorphism defined in $(\ref{eq:covariantoperators})$, then we have
$$\mathbb{B}^{\pm}_{\overline{g}}(F(\psi))=F(\mathbb{B}_{g}^{\pm}(\psi)).$$

\medskip

For now on, we will focus on the set $\mathcal{S}_{\partial}$. The analysis of concentration phenomena occurring at points in $\mathcal{S}_{int}$ is exactly similar to the case of a manifold without boundary treated in \cite{MV3}.
\medskip

 For a given (PS) sequence $(z_{n})$, we define the boundary concentration function $Q_{n}$ for $r> 0$ by
$$Q_{n}(r)=\sup_{x\in \partial M} \int_{B_{r}(x)}|u_{n}|^{2}|\psi_{n}|^{2}\ dv_{g}+\frac{b}{2}\int_{B_{r}(x)\cap \partial M}u_{n}^{4}\ d\sigma_{g}.$$
We choose $\epsilon>0$ so that $3\epsilon<\epsilon_{0}$, then if $\Sigma_3\not=0$, we have the existence of $x_{n}\in M$ and $R_{n}\to 0$ such that
$$Q_{n}(R_{n})=\int_{B_{R_{n}}(x_{n})}|u_{n}|^{2}|\psi_{n}|^{2}dv_{g}+\frac{b}{2}\int_{B_{R_{n}}(x_{n})\cap \partial M}u_{n}^{4}\ d\sigma_{g}=\epsilon .$$
Without loss of generality, we can always assume that $x_{n}\in \partial M$, $x_{n}\to x_{0}$ and $i(M)\geq 3$, where $i(M)$ is the boundary injectivity radius of $M$. Also, we define the map $\rho_{n}(x)=\mathcal{F}_{x_{n}}(R_{n}x)$ for $x\in \R^{3}_{+}$ such that $R_{n}|x|<3$. Here $\mathcal F$ denotes the Fermi coordinates map, for brevity.
We denote also $\sigma_{n}=\rho_{n}^{-1}$.

We let $B_{R,+}^{0}$ denote the upper half of the Euclidean ball centered at zero and with radius $R$. That is,
$$B_{R.+}^{0}=\{x=(\tilde x,\tilde y,\tilde z)\in \R^{3}; |x|< R; \tilde z\geq 0\}.$$
We can then consider the metric $g_{n}$ on $B_{R,+}^{0}$ defined by a suitable rescaled of the pull-back of $g$:
$$g_{n}=R_{n}^{-2}\rho_{n}^{*}g .$$
Clearly, the two Riemannian patches $(B_{R,+}^{0},g_{n})$ and $(B_{RR_{n},+}(x_{n}),g)$ are conformally equivalent for $n$ large enough and $g_{n}\to g_{\R^{3}_{+}}$ in $C^{\infty}(B_{R,+}^{0})$. We consider now the identification map  (see \cite{BG})
$$(\rho_{n})_{*}:\Sigma_{p}(B_{R,+}^{0},g_{n})\to \Sigma_{\rho_{n}(p)}(B_{RR_{n},+}(x_{n}),g),$$
and we set
$$\rho_{n}^{*}(\varphi)=(\rho_{n})_{*}^{-1}\circ \varphi \circ \rho_{n}.$$
Using these maps, we can define the spinors $\Psi_{n}$ on $\Sigma B_{R,+}^{0}$ by
$$\Psi_{n}=R_{n}\rho_{n}^{*}\psi_{n} ,$$
and from the conformal change of the Dirac operator, we have that
$$D_{g_{n}}\Psi_{n}=R_{n}^{2}\rho_{n}^{*}D_{g}\psi_{n} ,$$
and
$$\mathbb{B}_{g}^{\pm}(\Psi_n)=0, \text{ on } B_{R,+}^{0}\cap \partial \R^{3}_{+}.$$
So we get:
$$\int_{B_{R,+}^{0}}\langle D_{g_{n}}\Psi_{n},\Psi_{n}\rangle dv_{g_{n}}=\int_{B_{RR_{n},+}(x_{n})}\langle D_{g}\psi_{n},\psi_{n}\rangle dv_{g},$$
$$\int_{B_{R,+}^{0}}|\Psi_{n}|^{3}dv_{g_{n}}=\int_{B_{RR_{n},+}(x_{n})}|\psi_{n}|^{3}dv_{g}.$$
Now we consider the $u$ component, that is we define
$$U_{n}=R_{n}^{\frac{1}{2}}\rho_{n}^{*}u_{n} ,$$
so that by conformal change of the conformal Laplacian, we have:
$$L_{g_{n}}U_{n}=R_{n}^{\frac{5}{2}}\rho^{*}_{n}L_{g}u_{n} ,$$
and
$$B_{g_{n}}U_{n}=R^{\frac{3}{2}}\rho_{n}^{*}B_{g}u_{n}.$$
Hence
$$\int_{B_{R,+}^{0}}U_{n}L_{g_{n}}U_{n}dv_{g_{n}}=\int_{B_{RR_{n},+}(x_{n})}u_{n}L_{g}u_{n}dv_{g},$$

\begin{equation}\label{e1a}
\int_{B_{R,+}^{0}}|U_{n}|^{6}dv_{g_{n}}=\int_{B_{RR_{n},+}(x_{n})}|u_{n}|^{6}dv_{g},
\end{equation}

\begin{equation}\label{e1b}
\int_{B_{R,+}^{0}\cap \partial\R^{3}_{+}}|U_{n}|^{4}d\sigma_{g_{n}}=\int_{B_{RR_{n},+}(x_{n})\cap \partial M}|u_{n}|^{4}d\sigma_{g},
\end{equation}
and
$$\int_{B_{R,+}^{0}}|U_{n}|^{2}|\Psi_{n}|^{2}dv_{g_{n}}=\int_{B_{RR_{n},+}(x_{n})}|u_{n}|^{2}|\psi_{n}|^{2}dv_{g}.$$
We have the following:
\begin{lemma}\label{lemloc}
Let us set
$$F_{n}=L_{g_{n}}U_{n}-|\Psi_{n}|^{2}U_{n}, \qquad H_{n}=D_{g_{n}}\Psi_{n}-|U_{n}|^{2}\Psi_{n} \text{ and } K_{n}=B_{g_{n}}U_{n}-bU_{n}^{3}.$$
Then
$$F_{n}\to 0  \text{ in } H^{-1}_{loc}(\R^{3}_{+}), \qquad  H_{n}\to 0   \text{ in } H^{-\frac{1}{2}}_{loc}(\Sigma \R^{3}_{+}), \text{ and }  K_{n}\to 0 \text{ in } H^{-\frac{1}{2}}_{loc}(\partial \R^{3}_{+}).$$
\end{lemma}

\noindent
Here the convergence in $H^{-1}_{loc}$ is understood in the sense that for all $R>0$,
$$\sup\left\{\langle F_{n}, F\rangle_{H^{-1},H^{1}}; F\in H^{1}(\R^{3}_{+}), \; \text{supp}(F)\subset B_{R,+}^{0},\; \|F\|_{H^{1}}\leq 1\right\}\to 0,$$
and similarly for $H_{n}$ and $K_{n}$.
\begin{proof}
Recall \eqref{equ},\eqref{equb},\eqref{eqp}. Notice that by construction, we have that
$$L_{g_{n}}U_{n}-|\Psi_{n}|^{2}U_{n}=R_{n}^{\frac{5}{2}}\rho_{n}^{*}(L_{g}u_{n}-|\psi_{n}|^{2}u_{n}) .$$
Hence we get
$$F_{n}=R_{n}^{\frac{5}{2}}\rho_{n}^{*}(\varepsilon_{n}) ,$$
and similarly
$$H_{n}=R_{n}^{2}\rho_{n}^{*}(\delta_{n}) \text{ and } K_{n}=R_{n}^{\frac{3}{2}}r_{n} .$$
Now we consider $F\in H^{1}(\R^{3}_{+})$ such that $supp(F)\subset B_{R,+}^{0}$ and $\|F\|_{H^{1}}\leq 1$. Since $R_{n}\to 0$, then for $n$ big enough we have that:
\begin{align}
\langle F_{n},F\rangle_{H^{-1},H^{1}} &=\int_{B^{0}_{R_{n}^{-1},+}}F_{n}F dv_{g_n}\notag \\
&=\int_{B^{0}_{R_{n}^{-1}},+}\rho^{*}_{n}(\varepsilon_{n})R_{n}^{\frac{5}{2}} F dv_{g_{n}}\notag\\
&=\int_{B^{0}_{R_{n}^{-1},+}}\rho^{*}_{n}(\varepsilon_{n})R_{n}^{-\frac{1}{2}} F dv_{\rho^{*}_{n}g}\notag\\
&=\int_{B_{1}(x_{n})}\varepsilon_{n}R_{n}^{-\frac{1}{2}}\sigma_{n}^{*}(F) dv_{g} \,,\notag
\end{align}
where $\sigma^*_n=(\rho^*_n)^{-1}$. On the other hand, we have that $\|R_{n}^{-\frac{1}{2}}\sigma_{n}^{*}(F)\|_{H^{1}}\leq C$, hence
$$\langle F_{n},F\rangle_{H^{-1},H^{1}}\to 0.$$
A similar estimate holds for $H_{n}$ and $K_{n}$. Indeed, in order to estimate the $H^{\frac{1}{2}}$-norm independently from $n$, we can assume without loss of generality that our test function $F$ is in $H^{1}$ and use the interpolation inequality $$\|F\|_{H^{\frac{1}{2}}}^{2}\leq \|F\|_{H^{1}}\|F\|_{L^{2}}.$$
\end{proof}

\noindent
Consider the space
$$D^{1}(\R^{3}_{+})=\left\{u\in L^{6}(\R^{3}_{+}); |\nabla u| \in L^{2}(\R^{3}_{+})\right\} $$
and
$$D^{\frac{1}{2}}(\Sigma\R^{3})=\left\{\psi \in L^{3}(\Sigma\R^{3});|\xi|^{\frac{1}{2}}|\widehat{\psi}|\in L^{2}(\R^{3})\right\},$$
where here $\widehat{\psi}$ is the Fourier transform of $\psi$. The space $D^{\frac{1}{2}}(\R^{3}_{+})$ is then the restriction of functions in $D^{\frac{1}{2}}(\R^{3})$ to $\R^{3}_{+}$. We have then the following:

\begin{lemma}\label{lemlim}
For $\epsilon>0$ small enough, there exist $U_{\infty}\in D^{1}(\R^{3}_{+})$ and $\Psi_{\infty}\in D^{\frac{1}{2}}(\Sigma \R^{3}_{+})$ such that $U_{n}\to U_{\infty}$ in $H^{1}_{loc}(\R^{3}_{+})$ and $\Psi_{n}\to \Psi_{\infty}$ in $H^{\frac{1}{2}}_{loc}(\Sigma\R^{3}_{+})$. Moreover they satisfy
\begin{equation}\label{eqR3}
\left\{\begin{array}{ll}
-\Delta_{g_{\R^3}} U_{\infty}=|\Psi_{\infty}|^{2}U_{\infty} \\
& \text{ on } \R^{3}_{+} .\\
D_{g_{\R^3}}\Psi_{\infty}=|U_{\infty}|^{2}\Psi_{\infty}\\
&\\
\mathbb{B}^{+}_{\R^{3}}\Psi_{\infty}=0\\
&\text{ on } \partial \R^{3}_{+}.\\
B_{g_{\R^{3}}}U_{\infty}=bU_{\infty}^{3}\\
\end{array}
\right.
\end{equation}
\end{lemma}

\begin{proof}
Since the sequence $Z_{n}=(U_n,\Psi_n)$ is bounded in $H^{1}_{loc}(\R^{3}_{+})\times H^{\frac{1}{2}}_{loc}(\Sigma\R^{3}_{+})$, for every $\beta \in C^{\infty}_{0}(\R^{3}_{+})$, we have that $\beta Z_{n}$ is bounded in $H^{1}(\R^{3}_{+})\times H^{\frac{1}{2}}(\Sigma\R^{3}_{+})$, hence there exist $U_{\infty}$ and $\Psi_{\infty}$ such that $U_{n}\rightharpoonup U_{\infty}$ in $H^{1}_{loc}(\R^{3}_{+})$ and $U_{n}\to U_{\infty}$ strongly in $L_{loc}^{p}(\R^{3}_{+})$ for $p<6$ and in $L^{q}_{loc}(\partial \R^{3}_{+})$ for $q<4$. Similarly $\Psi_{n}\rightharpoonup \Psi_{\infty}$ in $H_{loc}^{\frac{1}{2}}(\Sigma \R^{3}_{+})$ and strongly in $L_{loc}^{p}(\Sigma \R^{3}_{+})$ for $p<3$. Now we notice that from \eqref{e1a}, we have that
$$\int_{B_{R,+}^{0}}|U_{n}|^{6}dv_{g_{n}}=\int_{B_{RR_{n}}(x_{n})}|u_{n}|^{6}dv_{g}.$$
Hence
\begin{equation}\label{e3}
\limsup_{n\to \infty}\int_{B_{R,+}^{0}}|U_{n}|^{6}dv_{g_{n}}\leq \sup_{n\geq 1}\int_{M}|u_{n}|^{6}dv_{g}<+\infty ,
\end{equation}
hence $U_{\infty}\in L^{6}(\R^{3}_{+})$ and similarly $\Psi_{\infty}\in L^{3}(\Sigma\R^{3}_{+})$ and $U_{\infty}\in L^{4}(\partial \R^{3}_{+})$. Also, as in the proof of Proposition \ref{propPS}, we see that $(U_{\infty},\Psi_{\infty})$ satisfies equation \eqref{eqR3}. Therefore,
$$\int_{\R^{3}_{+}}|\nabla U_{\infty}|^{2}dv_{g_{\R^{3}}}=\int_{\R^{3}_{+}}|U_{\infty}|^{2}|\Psi_{\infty}|^{2} dv_{g_{\R^{3}}}+b\int_{\R^{2}} U_{\infty}^{4}\ d\sigma_{g_{\R^{3}}}<\infty.$$
and $\nabla \Psi_{\infty}\in L^{\frac{3}{2}}(\Sigma\R^{3}_{+})\subset H^{-\frac{1}{2}}(\Sigma\R^{3}_{+})$, which leads to $U_{\infty}\in D^{1}(\R^{3}_{+})$ and $\Psi_{\infty}\in D^{\frac{1}{2}}(\Sigma\R^{3}_{+})$. Now, using again Lemma \ref{lemsplit}, we can assume at this stage that $\Psi_{\infty}=0$ and $U_{\infty}=0$ by replacing $\Psi_{n}$ by $\Psi_{n}-\Psi_{\infty}$ and $U_{n}$ by $U_{n}-U_{\infty}$.\\
Let $\beta \in C^{\infty}_{0}(\R^{3})$, then by elliptic regularity, we have that
\begin{align}
\|\beta^{2} U_{n}\|_{H^{1}(\R^{3}_{+})}&\leq C\left(\|L_{g_{\R^{3}}}(\beta^{2}U_{n})\|_{H^{-1}}+\|B_{g_{\R^{3}}}(\beta^{2}U_{n})\|_{H^{-\frac{1}{2}}}\right)\\
&\leq C\Big(\|L_{g_{n}}(\beta^{2}U_{n})\|_{H^{-1}}+\|B_{g_{n}}U_{n}\|_{H^{-\frac{1}{2}}}+\|(L_{g_{\R^{3}}}-L_{g_{n}})(\beta^{2}U_{n})\|_{H^{-1}}\notag\\
&\qquad+\|(B_{g_{n}}-B_{g_{\R^{3}}})(\beta^{2}U_{n})\|_{H^{-\frac{1}{2}}}+\|\beta^{2}U_{n}\|_{L^{2}}\Big).\notag
\end{align}
Now, we have that $\|\beta^{2}U_{n}\|_{L^{2}}\to 0$, and we want to estimate the term
$$\|(L_{g_{\R^{3}}}-L_{g_{n}})(\beta^{2}U_{n})\|_{H^{-1}}+\|(B_{g_{n}}-B_{g_{\R^{3}}})(\beta^{2}U_{n})\|_{H^{-\frac{1}{2}}}.$$

We consider the map $\mathcal{P}_{g}:H^{1}\to H^{-1}\times H^{-\frac{1}{2}}$ defined by
$$\mathcal{P}_{g}(u)=(L_{g}u,B_{g}u).$$
The adjoint of this map is $\mathcal{P}_{g}^{*}:H^{-1}\times H^{-\frac{1}{2}}\to H^{1}$ and defined by
$$\langle u,\mathcal{P}_{g}^{*}(f,g)\rangle =\langle \mathcal{P}_{g} u, (f,g)\rangle =\langle L_{g}u,f\rangle +\langle B_{g}u,g\rangle.$$

Then from \cite[Theorems~5.1~and~6.5]{LM}, we have that $\mathcal{P}_{g}$ and $\mathcal{P}^{*}_{g}$ are continuous. Let $F\in H^{1}\times  H^{\frac{1}{2}}$, then we have
\begin{align}
&|\langle (\mathcal{P}_{g_{\R^{3}}}-\mathcal{P}_{g_{n}})(\beta^{2}U_{n}),F\rangle|=|\langle \beta U_{n},\beta (\mathcal{P}_{g_{\R^{3}}}^{*}-\mathcal{P}_{g_{n}}^{*})(F)\rangle|\leq \|\beta U_{n}\|_{H^{1}}\|\beta (\mathcal{P}_{g_{\R^{3}}}^{*}-\mathcal{P}_{g_{n}}^{*})(F)\|_{H^{-1}}\notag\\
&\leq \|\beta U_{n}\|_{H^{1}}\|F\|_{H^{1}\times H^{\frac{1}{2}}}\Big(\|\beta (L_{g_{n}}-L_{g_{\R^{3}}})\|_{H^{1}\to H^{-1}}+\|\beta(B_{g_{n}}-B_{g_{\R^{3}}})\|_{H^{\frac{1}{2}}\to H^{-\frac{1}{2}}}\Big)\notag
\end{align}

Notice that since $g_{n}\to g_{\R^{3}}$ in $C^{\infty}$, hence
$$\|(L_{g_{\R^{3}}}-L_{g_{n}})(\beta^{2}U_{n})\|_{H^{-1}}+\|(B_{g_{n}}-B_{g_{\R^{3}}})(\beta^{2}U_{n})\|_{H^{-\frac{1}{2}}}\to 0.$$
It remains to estimate the term $\|L_{g_{n}}(\beta^{2}U_{n})\|_{H^{-1}}$ and $\|B_{g_{n}}(\beta^{2}U_{n})\|_{H^{-\frac{1}{2}}}$. To this aim, observe that

$$\|L_{g_{n}}(\beta^{2}U_{n})\|_{H^{-1}}+\|B_{g_{n}}(\beta^{2} U_{n})\|_{H^{-\frac{1}{2}}}\leq \|\beta^{2}(|\Psi_{n}|^{2}U_{n}+F_{n})\|_{H^{-1}}+\|\beta^{2}(bU_{n}^{3}+K_{n})\|_{H^{-\frac{1}{2}}}+o(1),$$
and from Lemma \ref{lemloc}, we have that $\beta^{2}F_{n}\to 0$ in $H^{-1}$ and $\beta^{2}K_{n}\to 0$ in $H^{-\frac{1}{2}}$, therefore, we get
$$\|\beta^{2} U_{n}\|_{H^{1}}\leq C\| \beta^{2}|\Psi_{n}|^{2}U_{n}\|_{H^{-1}}+b\|\beta U_{n}^{3}\|_{H^{-\frac{1}{2}}}+o(1).$$
Now, if we take $supp(\beta)\in B_{1,+}^{0}$ and recall that by definition
$$\int_{B_{1,+}^{0}}|U_{n}|^{2}|\Psi_{n}|^{2}dv_{g_n}+\frac{b}{2}\int_{B_{1,+}^{0}\cap \partial \R^{3}_{+}}|U_{n}|^{4}\leq \epsilon.$$
 we have that
\begin{align}
\|\beta^{2} U_{n}\|_{H^{1}}&\leq C\Big(\| \beta^{2}|\Psi_{n}|^{2}U_{n}\|_{L^{\frac{6}{5}}(B_{1,+}^{0})}+b\|U_{n}^{3}\|_{L^{\frac{4}{3}}(B_{1,+}^{0}\cap \partial \R^{3}_{+})}\Big)+o(1)\notag \\
&\leq C\Big(\Big( \int_{B_{1}^{0}}|U_{n}|^{2}|\Psi_{n}|^{2}dv_{g_{n}}\Big)^{\frac{1}{2}}\|\beta^{2}\Psi_{n}\|_{L^{3}}+b\|\beta^{2}U_{n}\|_{H^{1}}\Big(\int_{B_{1,+}^{0}\cap \partial \R^{3}_{+}}U_{n}^{4}\ d\sigma_{g_{\R^{3}}}\Big)^{\frac{1}{2}}\Big)+o(1)\notag\\
&\leq C\epsilon^{\frac{1}{2}}\|\beta^{2}\Psi_{n}\|_{L^{3}}+C\epsilon^{\frac{1}{2}}\|\beta^{2}U_{n}\|_{H^{1}}+o(1) .\notag
\end{align}
A similar computation can be done to show that
$$\|\beta^{2} \Psi_{n}\|_{H^{\frac{1}{2}}}\leq  C\epsilon^{\frac{1}{2}}\|\beta^{2}U_{n}\|_{L^{6}}+o(1),$$
and combining these last two estimates we have that
$$\|\beta^{2}U_{n}\|_{H^{1}}+\|\beta^{2}\Psi_{n}\|_{H^{\frac{1}{2}}}\to 0.$$
\end{proof}

\noindent
We deduce from the previous Lemma that since
$$\int_{B_{1,+}^{0}}|U_{n}|^{2}|\Psi_{n}|^{2}\ dv_{g_n}+\frac{b}{2}\int_{B_{1,+}^{0}\cap \partial \R^{3}_{+}}|U_{n}|^{4} \ d\sigma_{g_{n}}=Q(R_{n})=\epsilon,$$
we also have
$$\int_{B_{1,+}^{0}}|U_{\infty}|^{2}|\Psi_{\infty}|^{2}\ dv_{g_{\R^{3}}}+\frac{b}{2}\int_{B_{1,+}^{0}\cap \partial \R^{3}_{+}}|U_{\infty}|^{4}\ d\sigma_{g_{\R^{3}}}=\epsilon.$$
In particular, $U_{\infty}\not=0$ and $(U_{\infty},\Psi_{\infty})$ satisfy equation (\ref{eqR3}); by the regularity results proved in the previous section, we have that $U_{\infty}\in C^{2,\alpha}(\R^{3}_{+})$ and $\Psi_{\infty}\in C^{1,\beta}(\Sigma\R^{3}_{+})$.\\
Notice here that there is a fundamental difference from the case where $M$ has no boundary, namely, if $b>0$. Indeed, here, $\Psi_{\infty}$ can be identically zero. Hence, we do have two kinds of solutions, corresponding to different blowup profiles: a purely Yamabe-Escobar solution $(U_{\infty},0)$ with
$$\left\{\begin{array}{ll}
-\Delta U_{\infty}=0 \text{ on } \R^{n}_{+}\\
&\\
\frac{\partial U_{\infty}}{\partial \nu}(x_{0},y_{0},0)=bU_{\infty}^{3} \text{ on } \partial \R^{3}_{+}.
\end{array}
\right.
$$
Or coupled solution with $\Psi_{\infty}\not=0$.


Now, we assume that $x_{n}\to x_{0}$ and  we consider a cut-off function $\beta=1$ on $B_{1}(x_{0})$ and $supp(\beta)\subset B_{2}(x_{0})$, we define then $v_{n}\in C^{2,\alpha}(M)$ and $\phi_{n}\in C^{1,\beta}(\Sigma M)$ by
\begin{equation}\label{e5}
v_{n}=R_{n}^{-\frac{1}{2}}\beta\sigma_{n}^{*}(U_\infty)
\end{equation}
and
\begin{equation}\label{e6}
\phi_{n}=R_{n}^{-1}\beta\sigma_{n}^{*}(\Psi_\infty) .
\end{equation}
We start by stating the following Lemma
\begin{lemma}\label{lemma1}
Let  $\overline{u}_{n}=u_{n}-v_{n}$ and $\overline{\psi}_{n}=\psi_{n}-\phi_{n}$. Then, up to a subsequence, $\overline{u}_{n}\rightharpoonup 0$ in $H^{1}(M)$ and $\overline{\psi}_{n}\rightharpoonup 0$ in $H^{\frac{1}{2}}_{+}(\Sigma M)$.
\end{lemma}
\begin{proof}
The proof of this lemma is similar to the one for manifolds without boundary. We will still write here its details since the technical manipulations of the integrals will be useful in later proofs.\\
We already have that $u_{n}\rightharpoonup 0$ and $\psi_{n}\rightharpoonup 0$, thus to prove the lemma we only need to show the weak convergence for $v_{n}$ and $\phi_{n}$: these sequences are bounded in $H^{1}(M)$ and $H^{1/2}(\Sigma M)$ respectively, then up to subsequences, they converge to some limit. So if we show that the distributional limit is zero, then the limit in the desired space is also zero. So let $f\in C^{\infty}(M)$ and $h\in C^{\infty}(\Sigma M)$. We want to show that
$$\int_{M}v_{n}fdv_{g} \to 0$$
and
$$\int_{M}\langle \phi_{n},h\rangle\ dv_{g}\to 0.$$
We fix $R>0$, so that
\begin{align}
\int_{B_{R_{n}R}(x_{n})}v_{n}f\ dv_{g}&=R_{n}^{-\frac{1}{2}}\int_{B_{R_{n}R}(x_{n})}\beta \sigma_{n}^{*}(U_\infty) f\ dv_{g}\notag\\
&=R_{n}^{\frac{5}{2}}\int_{B_{R,+}^{0}}\rho_{n}^{*}(\beta) \rho_{n}^{*}(f)U_\infty\ dv_{g_{n}} .\notag
\end{align}
Hence
$$\left|\int_{B_{R_{n}R}(x_{n})}v_{n}f\ dv_{g}\right|\leq CR_{n}^{\frac{5}{2}} \|f\|_{\infty}\int_{B_{R,+}^{0}}|U_\infty|\ dv_{g_{\R^{3}}}.$$
Also, for $n$ big enough we have that
\begin{align}
\int_{M\setminus B_{R_{n}R}(x_{n})}v_{n}f\ dv_{g}&=\int_{B_{3}(x_{n})\setminus B_{R_{n}R}(x_{n})}v_{n}f\ dv_{g}\notag\\
&=R_{n}^{\frac{5}{2}}\int_{B_{3R_{n}^{-1},+}^{0}\setminus B_{R,+}^{0}}\rho_{n}^{*}(\beta) \rho_{n}^{*}(f)U_\infty\ dv_{g_{n}} .\notag
\end{align}
Hence, we have
\begin{align}
\left|\int_{M\setminus B_{R_{n}R}(x_{n})}v_{n}f\ dv_{g}\right|&\leq CR_{n}^{\frac{5}{2}}\|f\|_{\infty}\int_{B_{3R_{n}^{-1}}^{0}\setminus B_{R}^{0}}|U_\infty|\ dv_{g_{\R^{3}}} \notag \\
&\leq C\|f\|_{\infty}\left(\int_{B_{3R_{n}^{-1},+}^{0}\setminus B_{R,+}^{0}}|U_\infty|^{6}\ dv_{g_{\R^{3}}}\right)^{\frac{1}{6}}.\notag
\end{align}
Based on these last two inequalities we have that
$$\left|\int_{M}v_{n}f\ dv_{g}\right|\leq C\|f\|_{\infty}\left(R_{n}^{\frac{5}{2}}\int_{B_{R,+}^{0}}|U_\infty|\ dv_{g_{\R^{3}}}+\left(\int_{\R^{3}\setminus B_{R,+}^{0}}|U_\infty|^{6}\ dv_{g_{\R^{3}_{+}}}\right)^{\frac{1}{6}}\right).$$
Letting $n\to \infty$ and then $R\to \infty$ we get the desired result. A similar inequality holds for the integral involving $\phi_{n}$ and $h$.
\end{proof}

\noindent
Now we estimate the differential, that is
\begin{lemma}\label{lemma2}
We have
$$dE^{b}(v_{n},\phi_{n})\to 0 \qquad \text{ and }\qquad dE^{b}(\overline{u}_{n},\overline{\psi}_{n})\to 0 ,$$
in $H^{-1}(M)\times H^{-\frac{1}{2}}(\partial M) \times H^{-\frac{1}{2}}(\Sigma M)$.
\end{lemma}
\begin{proof}
We set
$$f_{n}=L_{g}v_{n}-|\phi_{n}|^{2}v_{n}, \qquad k_{n}=B_{g}(u_{n})-u_{n}^{3}$$
and
$$h_{n}=D_{g}\phi_{n}-|v_{n}|^{2}\phi_{n} .$$
Let $f\in H^{1}(M)$, $h\in H^{\frac{1}{2}}_{+}(\Sigma M)$ and $k\in H^{\frac{1}{2}}(\partial M)$. We compute
\begin{align}
\int_{M}f_{n}f\ dv_{g}&=R_{n}^{-\frac{1}{2}}\left(\int_{M}(-\Delta_{g}\beta)\sigma_{n}^{*}(U_{\infty}) f\ dv_{g}+2\int_{M}g(-\nabla \beta,\nabla \sigma_{n}^{*}(U_{\infty})) f\ dv_{g} \right)\notag\\
&\quad +R_{n}^{-\frac{1}{2}}\int_{M}\beta L_{g}(\sigma_{n}^{*}(U_\infty))f\ dv_{g}-R_{n}^{-\frac{5}{2}}\int_{M}\beta^{3} |\sigma_{n}^{*}(\Psi_\infty)|^{2}\sigma_{n}^{*}(U_\infty) f\ dv_{g}\notag\\
&=R_{n}^{-\frac{1}{2}}\Big(\int_{M}(\Delta_{g}\beta) \sigma_{n}^{*}(U_{\infty}) f dv_{g}+2\int_{M}g(\nabla \beta,\nabla f) \sigma_{n}^{*}(U_{\infty}) dv_{g} -\int_{\partial M} \frac{\partial \beta}{\partial \nu} \sigma_{n}^{*}(U_{\infty}) d\sigma_{g} \Big)\notag\\
&\quad + R_{n}^{-\frac{5}{2}}\int_{M}\beta \sigma_{n}^{*}(L_{g_{n}}U_{\infty})f\ dv_{g}-R_{n}^{-\frac{5}{2}}\int_{M}\beta^{3} \sigma_{n}^{*}(|\Psi_{\infty}|^{2}U_{\infty})f\ dv_{g}\notag\\
&=R_{n}^{-\frac{1}{2}}\Big(\int_{M}(\Delta_{g}\beta) \sigma_{n}^{*}(U_{\infty}) f dv_{g}+2\int_{M}g(\nabla \beta,\nabla f) \sigma_{n}^{*}(U_{\infty})dv_{g}- \int_{\partial M} \frac{\partial \beta}{\partial \nu} \sigma_{n}^{*}(U_{\infty}) d\sigma_{g} \Big)\notag\\
&\quad + R_{n}^{-\frac{5}{2}}\int_{M}\beta \sigma_{n}^{*}\left((L_{g_{n}}-L_{g_{\R^{3}}})U_{\infty}\right)f\ dv_{g}+R_{n}^{-\frac{5}{2}}\int_{M}(\beta-\beta^{3}) \sigma_{n}^{*}(|\Psi_{\infty}|^{2}U_{\infty}) f\ dv_{g}\notag\\
&=:I_{1}+I_{2}+I_{3}+I_{4}+I_{5} .\notag
\end{align}
The estimate for $I_{1}$, $I_{2}$, $I_{4}$ and $I_{5}$ are similar to the case of manifolds with no boundary in \cite{MV3}, for which one can show that
$$I_{1}+I_{2}+I_{4}+I_{5}=\|h\|_{H^{1}(M)}o(1).$$
Therefore we focus on the extra term due to the presence of the boundary, namely $I_{3}$. We get
\begin{align}
|I_{3}|&=R_{n}^{-\frac{1}{2}}\Big|\int_{\partial M}\frac{\partial \beta}{\partial \nu}\sigma^{*}_{n}(U_{\infty}) f \ d\sigma_{g}\Big|\notag\\
&\leq R_{n}^{-\frac{1}{2}} \|f\|_{L^{4}(\partial M)}\|\frac{\partial \beta}{\partial \nu}\sigma^{*}_{n}(U_{\infty})\|_{L^{\frac{4}{3}}(\partial M)}\notag\\
&\leq C\|f\|_{H^{1}(M)} R_{n} \Big(\int_{R_{n}^{-1}\mathcal{F}^{-1}_{x_{n}}(B_{2}(x_{0})\cap \partial M )\setminus R_{n}^{-1}\mathcal{F}^{-1}_{x_{n}}(B_{1}(x_{0})\cap \partial M )}|U_{\infty}|^{\frac{4}{3}} \ d\sigma_{g_{n}}\Big)^{\frac{3}{4}}\notag\\
&\leq C\|f\|_{H^{1}} \Big(\int_{(B_{3R_{n}^{-1},+}^{0}\setminus B_{\frac{1}{2}R_{n}^{-1},+}^{0})\cap \partial \R^{3}_{+}}|U_{\infty}|^{4}\ d\sigma_{g_{\R^{3}}}\Big)^{\frac{1}{4}}\notag\\
&\leq \|f\|_{H^{1}}o(1).\notag
\end{align}

Therefore, we can conclude that $f_{n}\to 0$ in $H^{-1}(M)$ and a similar convergence holds for $h_{n}\to 0$ in $H^{-\frac{1}{2}}(\Sigma M)$.\\
We now deal with the boundary component $k_{n}$. So we take $k\in H^{\frac{1}{2}}(\partial M)$ and we have
\begin{align}
\int_{\partial M}k_{n}k d\sigma_{g}&=R_{n}^{-\frac{1}{2}}\int_{\partial M} \frac{\partial \beta}{\partial \nu}\sigma_{n}^{*}(U_{\infty})k\ d\sigma_{g}+R_{n}^{-\frac{3}{2}}\int_{\partial M}\beta k \sigma_{n}^{*}(B_{g_{n}}(U_{\infty}))\ d\sigma_{g}\notag\\
&-bR_{n}^{-\frac{3}{2}}\int_{\partial M}\beta^{3}k\sigma_{n}^{*}(U_{\infty})^{3}\ d\sigma_{g}\notag\\
&=R_{n}^{-\frac{1}{2}}\int_{\partial M} \frac{\partial \beta}{\partial \nu}\sigma_{n}^{*}(U_{\infty})k\ d\sigma_{g} +R_{n}^{-\frac{3}{2}}\int_{\partial M}\beta k \sigma_{n}^{*}(B_{g_{n}}-B_{g_{\R^{3}}})(U_{\infty}))\ d\sigma_{g}\notag\\
&+bR_{n}^{-\frac{3}{2}}\int_{\partial M}(\beta -\beta^{3})k\sigma_{n}^{*}(U_{\infty})^{3}\ d\sigma_{g}\notag\\
&=J_{1}+J_{2}+J_{3}.\notag
\end{align}

We start estimating $J_{1}$

\begin{align}
|J_{1}|&\leq C R_{n}^{-\frac{1}{2}}\|k\|_{H^{\frac{1}{2}}}\Big(\int_{R_{n}^{-1}\mathcal{F}_{x_{n}}^{-1}(B_{2}(x_{0})\cap \partial M)\setminus R_{n}^{-1}\mathcal{F}_{x_{n}}^{-1}(B_{1}(x_{0})\cap \partial M)}|U_{\infty}|^{\frac{4}{3}}\ d\sigma_{\rho^{*}_{n}g}\big)^{\frac{3}{4}}\notag\\
&\leq C R_{n} \|k\|_{H^{\frac{1}{2}}} \Big(\int_{\Big(B_{3R_{n}^{-1},+}^{0}\setminus B_{\frac{1}{2}R_{n}^{-1},+}^{0}\Big)\cap \partial \R^{3}_{+}}|U_{\infty}|^{\frac{4}{3}}\ d\sigma_{g_{n}}\Big)^{\frac{3}{4}}\notag\\
&\leq C \Big(\int_{\Big(B_{3R_{n}^{-1},+}^{0}\setminus B_{\frac{1}{2}R_{n}^{-1},+}^{0}\Big)\cap \partial \R^{3}_{+}}|U_{\infty}|^{4}\ d\sigma_{g_{\R^{3}}}\big)^{\frac{1}{4}}\|k\|_{H^{\frac{1}{2}}}\notag\\
&\leq \|k\|_{H^{\frac{1}{2}}} o(1).\notag
\end{align}

For $J_{2}$, we have
\begin{align}
|J_{2}|&\leq C\|k\|_{H^{\frac{1}{2}}}R_{n}^{-\frac{3}{2}}\Big(\int_{R_{n}^{-1}\mathcal{F}^{-1}_{x_{n}}(B_{2}(x_{0})\cap \partial M)} |(B_{g_{n}}-B_{g_{\R^{3}}})U_{\infty}|^{\frac{4}{3}}\ d\sigma_{\rho_{n}^{*}g}\Big)^{\frac{3}{4}}\notag\\
&\leq C\|k\|_{H^{\frac{1}{2}}}\|(B_{g_{n}}-B_{g_{\R^{3}}})U_{\infty}\|_{L^{\frac{4}{3}}(B_{3R_{n}^{-1},+}^{0}\cap \partial \R^{3}_{+})}\notag\\
&\leq C \|k\|_{H^{\frac{1}{2}}}\Big(\|(B_{g_{n}}-B_{g_{\R^{3}}})U_{\infty}\|_{L^{\frac{4}{3}}(B_{R,+}^{0}\cap \partial \R^{3}_{+})}+
\|(B_{g_{n}}-B_{g_{\R^{3}}})U_{\infty}\|_{L^{\frac{4}{3}}((\R^{3}_{+}\setminus B_{R,+}^{0})\cap \partial \R^{3}_{+})}\Big).\notag
\end{align}
The conclusion then follows from the fact that $g_{n}\to g_{\R^{3}}$ on $C^{\infty}(B_{R,+}^{0})$ hence
$$\|(B_{g_{n}}-B_{g_{\R^{3}}})U_{\infty}\|_{L^{\frac{4}{3}}(B_{R,+}^{0}\cap \partial \R^{3}_{+})}\to 0 \text{ as } n\to \infty$$ and $B_{g_{\R^{3}}}U_{\infty}=bU_{\infty}^{3} \in L^{\frac{4}{3}}(\partial \R^{3}_{+})$, therefore $\|(B_{g_{n}}-B_{g_{\R^{3}}})U_{\infty}\|_{L^{\frac{4}{3}}((\R^{3}_{+}\setminus B_{R,+}^{0})\cap \partial \R^{3}_{+})}\Big)\to 0$ as $R\to \infty$.

Lastly, we estimate $J_{3}$ as follows:

\begin{align}
|J_{3}|&\leq C\|k\|_{H^{\frac{1}{2}}(\partial M)}R_{n}^{-\frac{3}{2}}\Big(\int_{R_{n}^{-1}\mathcal{F}_{x_{n}}^{-1}(B_{2}(x_{0})\cap \partial M)\setminus R_{n}^{-1}\mathcal{F}_{x_{n}}^{-1}(B_{1}(x_{0})\cap \partial M)} |U_{\infty}|^{4} \ d\sigma_{\rho^{*}_{n}g}\Big)^{\frac{3}{4}}\notag\\
&\leq C\|k\|_{H^{\frac{1}{2}}(\partial M)} \|U_{\infty}\|_{L^{4}((B_{3R_{n}^{-1},+}^{0}\setminus B_{\frac{1}{2}R_{n}^{-1},+}^{0})\cap \partial \R^{3}_{+})}^{3}\notag\\
&\leq \|k\|_{H^{\frac{1}{2}}(\partial M)} o(1).\notag
\end{align}

Next, we consider $(\overline{u}_{n},\overline{\psi}_{n})$ and we fix again $f\in H^{1}(M)$. First, notice that
\begin{align}
d_{u}E^{b}(\overline{u}_{n},\overline{\psi}_{n})f&=\int_{M}fL_{g}\overline{u}_{n}dv_{g}-\int_{M}|\overline{\psi}_{n}|^{2}\overline{u}_{n} f\ dv_{g}+\int_{\partial M}\frac{\partial \overline{u}_{n}}{\partial \nu}f \ d\sigma_{g} -b\int_{\partial M}\overline{u}_{n}^{3}f \ d\sigma_{g}\notag\\
&=d_{u}E^{b}(u_{n},\psi_{n})f-d_{u}E^{b}(v_{n},\phi_{n})f+\int_{M}A_{n}f\ dv_{g}+\int_{\partial M} B_{n}f \ d\sigma_{g} ,\notag
\end{align}
where
$$A_{n}=|\psi_{n}|^{2}v_{n}-|\phi_{n}|^{2}u_{n}+2\langle \psi_{n},\phi_{n}\rangle (u_{n}-v_{n})$$
and
$$B_{n}=-b\Big(3u_{n}^{2}v_{n}-3u_{n}v_{n}^{2}\Big).$$
Now, since we already proved that $d_{u}E^{b}(u_{n},\psi_{n})\to 0$ and $d_{u}E^{b}(v_{n},\phi_{n})\to 0$, it is enough to show that $A_{n}\to 0$ in $H^{-1}$ and $B_{n}\to 0$ in $H^{-\frac{1}{2}}(\partial M)$. The convergence of the interior component $A_{n}$ is similar to the case of manifolds without boundary \cite{MV3}. Then we focus on the boundary component $B_{n}$, for which we get

\begin{align}
&\|B_{n}\|_{L^{\frac{4}{3}}(\partial M \setminus \partial B_{R_{n}R}(x_{n}))}\leq C\Big(\|u_{n}^{2}v_{n}\|_{L^{\frac{4}{3}}(\partial M \setminus \partial B_{R_{n}R}(x_{n}))}+\|u_{n}v_{n}^{2}\|_{L^{\frac{4}{3}}(\partial M \setminus \partial B_{R_{n}R}(x_{n}))}\Big)\notag\\
&\leq C\Big(\|u_{n}\|_{L^{4}(\partial M)}^{2}\|v_{n}\|_{L^{4}(\partial M \setminus \partial B_{R_{n}R}(x_{n}))}+ \|u_{n}\|_{L^{4}(\partial M)}\|v_{n}\|_{L^{4}(\partial M \setminus \partial B_{R_{n}R}(x_{n}))}^{2}\Big)\notag
\end{align}
But $\|u_{n}\|_{L^{4}(\partial M)}$ is uniformly bounded and
$$
\int_{ \partial M \setminus \partial B_{R_{n}R}(x_{n})}|v_{n}|^{4}\ d\sigma_{g} \leq C\int_{\left(B^{0}_{3R_{n}^{-1},+}\setminus B^{0}_{R,+}\right) \cap \partial \R^{3}_{+}}|U_{\infty}|^{4}\ d\sigma_{g_{n}}.$$
Since $U_{\infty}\in L^{4}(\partial \R^{3}_{+} )$ we have that $\|v_{n}\|_{L^{4}(\partial M \setminus \partial B_{R_{n}R}(x_{n}))}\to 0$ as $R\to \infty$. Thus
$$\|B_{n}\|_{L^{\frac{4}{3}}(\partial M \setminus \partial B_{R_{n}R}(x_{n}))} \to 0 \text{ as } R\to \infty.$$
Next, we consider $\|B_{n}\|_{L^{\frac{4}{3}}(B_{RR_{n}}(x_{n})\cap \partial M)}$. Observe that, by convexity of $t\mapsto t^{4/3}$, $t\geq0$, we have
\[
\vert u_n v_n\vert^{4/3}\leq C(\vert u_n\vert^{8/3}+\vert v_n\vert^{8/3})\,,
\]
since $\vert u_n v_n\vert\leq \vert u_n\vert^2+\vert v_n\vert^2$. Then, the H\"older inequality gives
\begin{align}
& \int_{B_{RR_{n}}(x_{n})\cap \partial M}|u_{n}v_{n}(v_{n}-u_{n})|^{\frac{4}{3}} \ d\sigma_{g} \notag\\
&\leq C(\|u_{n}\|_{L^{4}(\partial M)}^{\frac{8}{3}}+\|u_{n}\|_{L^{4}(\partial M)}^{\frac{8}{3}})\left(\int_{ B_{RR_{n}}(x_{n})\cap \partial M}|u_{n}-v_{n}|^{4} \ d\sigma_{g}\right)^{\frac{1}{3}}.\notag
\end{align}

But,
\begin{align}
\int_{ B_{RR_{n}}(x_{n})\cap \partial M}|u_{n}-v_{n}|^{4}\ d\sigma_{g}& \leq C \int_{B_{3R,+}^{0}\cap \partial \R^{3}_{+}}|R_{n}^{\frac{1}{2}}\rho_{n}^{*}u_{n}-U_{\infty}|^{4}\ d\sigma_{g_{n}}\notag\\
&\leq C \int_{B_{3R,+}^{0}\cap \partial \R^{3}_{+}}|U_{n}-U_{\infty}|^{4}\ d\sigma_{g_{\R^{3}}}.\notag
\end{align}
On the other hand, from Lemma \ref{lemlim} we have that $U_{n} \to U_{\infty}$ in  $H^{1}_{loc}(\R^{3}_{+})$, and in particular $$\int_{B_{3R,+}^{0}\cap \partial \R^{3}_{+}}|U_{n}-U_{\infty}|^{4}\ d\sigma_{g_{\R^{3}}}\to 0 \text{ as } n\to \infty.$$
Using the boundedness of $u_{n}$ and $v_{n}$ in $L^{4}(\partial M)$, we conclude that
$$\|B_{n}\|_{L^{\frac{4}{3}}(B_{RR_{n}}(x_{n})\cap \partial M)}\to 0 \text{ as } n\to \infty.$$
Thus, collecting the above estimates we find
$$\|B_{n}\|_{H^{-\frac{1}{2}}(\partial M)}\to 0,$$
concluding the proof for $d_{u}E^{b}$. The same computations also hold for $d_{\psi}E^{b}$.
\end{proof}

Now we estimate the energy, that is:
\begin{lemma}\label{lemma3}
We have
$$E^{b}(\overline{u}_{n},\overline{\psi}_{n})=E^{b}(u_{n},\psi_{n})-E_{\R^{3}_{+}}^{b}(U_{\infty},\Psi_{\infty})+o(1).$$
\end{lemma}
\begin{proof}
We have
\begin{align}
E^{b}(\overline{u}_{n},\overline{\psi}_{n})&=\frac{1}{2}\Big(\int_{M}| \nabla (u_{n}-v_{n})|^{2}+\frac{R}{8}(u_{n}-v_{n})^{2}+\langle D_{g}(\psi_{n}-\phi_{n}),\psi_{n}-\phi_{n}\rangle \notag\\
&\qquad-|\psi_{n}-\phi_{n}|^{2}|u_{n}-v_{n}|^{2}\ dv_{g}\Big)-\frac{b}{4}\int_{\partial M}(u_{n}-v_{n})^{4}\ d\sigma_{g}\notag\\
&=E^{b}(u_{n},\psi_{n})+E^{b}(v_{n},\phi_{n})-dE^{b}(v_{n},\phi_{n})(u_{n},\psi_{n})+\tilde{A}_{n}+\tilde{B}_{n},\notag
\end{align}
where
$$\tilde{A}_{n}=-\frac{1}{2}\left(\int_{M}|u_{n}|^{2}|\phi_{n}|^{2}+|v_{n}|^{2}|\psi_{n}|^{2}-2|u_{n}|^{2}\langle \psi_{n},\phi_{n}\rangle -2|\psi_{n}|^{2}u_{n}v_{n}+4u_{n}v_{n}\langle \psi_{n},\phi_{n}\rangle\ dv_{g}\right),$$
and
$$\tilde{B}_{n}=-\frac{b}{4}\int_{\partial M}6u_{n}^{2}v_{n}^{2}-4v_{n}u_{n}^{3}\ d\sigma_{g}.$$
We first notice that since $(u_{n},\psi_{n})$ is bounded in $H^{1}(M)\times H^{\frac{1}{2}}_{+}(\Sigma M)$, from Lemma \ref{lemsplit}, we have that $$dE^{b}(v_{n},\phi_{n})(u_{n},\psi_{n})\to 0 .$$

Let us start by estimating $\tilde{B}_{n}$.

\begin{align}
\tilde{B}_{n}&=-b\int_{\partial M}u_{n}v_{n}^{2}(u_{n}-v_{n})\ d\sigma_{g}-\frac{b}{2}\int_{\partial M}u_{n}^{2}v_{n}^{2}\ d \sigma_{g}\notag\\
&=L_{1}+L_{2}\notag
\end{align}

\begin{align}
|L_{1}|&=b\Big(\int_{B_{RR_{n}}(x_{n})\cap\partial M}u_{n}v_{n}^{2}(u_{n}-v_{n}) d\sigma_{g}+\int_{(M\setminus B_{RR_{n}}(x_{n}))\cap\partial M}u_{n}v_{n}^{2}(u_{n}-v_{n}) d\sigma_{g}\notag\\
&\leq C\Big( \|v_{n}\|_{L^{4}(\partial M)}^{2}\|u_{n}\|_{L^{4}(\partial M)}\|u_{n}-v_{n}\|_{L^{4}(B_{RR_{n}}(x_{n})\cap\partial M)}\notag\\
&\qquad+\|u_{n}\|_{L^{4}(M)}^{2}\|v_{n}\|^{2}_{L^{4}(\partial M \setminus B_{RR_{n}}(x_{n}))}+\|u_{n}\|_{L^{4}(M)}\|v_{n}\|_{L^{4}(\partial M \setminus B_{RR_{n}}(x_{n}))}^{3}\Big).\notag
\end{align}
But,
$$\int_{B_{RR_{n}}(x_{n})\cap\partial M}(u_{n}-v_{n})^{4} \ d\sigma_{g}\leq \int_{B_{3R,+}^{0}\cap \partial \R^{3}_{+}}(U_{n}-U_{\infty})^{4}\ d\sigma_{g_{\R^{3}}}.
$$
and since $U_{n}\to U_{\infty}$ in $H^{1}_{loc}(\R^{3}_{+})$, we have that
\begin{equation}\label{eq:B1}
\|u_{n}-v_{n}\|_{L^{4}(B_{RR_{n}}(x_{n})\cap \partial M)}\to 0, \text{ as } n\to \infty.
\end{equation}
Similarly, we have
$$\int_{\partial M}|v_{n}|^{4}\ d\sigma_{g}=\int_{\partial \R^{3}_{+}}\rho_{n}^{*}\beta^{4}U_{\infty}^{4}d\sigma_{g_{n}},$$
which clearly converges to $\|U_{\infty}\|_{L^{4}(\partial \R^{3}_{+})}^{4}$. Therefore,
\begin{equation}\label{eq:B2}
\|v_{n}\|_{L^{4}(\partial M)}\to \|U_{\infty}\|_{L^{4}(\partial \R^{3}_{+})}.
\end{equation}
Combining $\eqref{eq:B1}$, $\eqref{eq:B2}$ and the uniform boundedness of $\|u_{n}\|_{L^{4}(\partial M)}$, we conclude that
$$\|v_{n}\|_{L^{4}(\partial M)}^{2}\|u_{n}\|_{L^{4}(\partial M)}\|u_{n}-v_{n}\|_{L^{4}(B_{RR_{n}}(x_{n})\cap \partial M)}\to 0 \text{ as } n\to \infty.$$

Next, we estimate $\|v_{n}\|_{L^{4}(\partial M \setminus B_{RR_{n}}(x_{n}))}$. We have,
\begin{align}
\int_{\partial M\setminus B_{RR_{n}}(x_{n})}v_{n}^{4} \ d\sigma_{g} &\leq C \int_{\Big(B_{3R_{n}^{-1},+}^{0}\setminus B_{\frac{R}{2},+}^{0}\Big)\cap \partial \R^{3}_{+}}U_{\infty}^{4} d\sigma_{g_{n}}\notag\\
&\leq C\int_{\partial \R^{3}_{+}\setminus B_{\frac{R}{2},+}^{0}}U_{\infty}^{4} d\sigma_{g_{\R^{3}}}.\notag
\end{align}
Therefore, since $U_{\infty}\in L^{4}(\partial \R^{3}_{+})$, we have that
\begin{equation}
\|v_{n}\|_{L^{4}(\partial M \setminus B_{RR_{n}}(x_{n}))} \to 0 \text{ as } R\to \infty.
\end{equation}
And thus, $L_{1}=o(1)$. Using a similar procedure, we see that
$$L_{2}\to -\frac{b}{2}\int_{\partial \R^{3}_{+}}U_{\infty}^{4} d \sigma_{g_{\R^{3}_{+}}}.$$

For the quantity $\tilde{A}_{n}$ one can show that
$$\tilde{A}_{n}\to \int_{\R^{3}_{+}}U_{\infty}^{2}|\Psi_{\infty}|^{2}\ dv_{g_{\R^{3}}}.$$
The proof is similar to the case of manifolds with no boundary, hence we omit it. Therefore, we have
$$E^{b}(\overline{u}_{n},\overline{\psi}_{n})=E^{b}(u_{n},\psi_{n})+E^{b}(v_{n},\phi_{n})-2E^{b}_{\R^{3}_{+}}(U_{\infty},\Psi_{\infty})+o(1) .$$
Now we estimate $E^{b}(v_{n},\phi_{n})$. Indeed,
$$E^{b}(v_{n},\phi_{n})=\frac{1}{2}\Big(\int_{M}v_{n}L_{g}v_{n} +\langle D\phi_{n}, \phi_{n}\rangle -|v_{n}|^{2}|\phi_{n}|^{2}\ dv_{g}\Big)-\frac{1}{2}\int_{\partial M} v_{n}B_{g}v_{n}-\frac{b}{2}v_{n}^{4}\ d\sigma_{g}.$$
Again, following \cite{MV3}, we see that
$$\int_{M}v_{n}L_{g}v_{n} +\langle D\phi_{n}, \phi_{n}\rangle -|v_{n}|^{2}|\phi_{n}|^{2}\ dv_{g}\to \int_{\R^{3}_{+}}U_{\infty}L_{g_{\R^{3}}}U_{\infty}+\langle D_{g_{\R^{3}}} \Psi_{\infty},\Psi_{\infty}\rangle -|U_{\infty}|^{2}|\Psi_{\infty}|^{2}\ dv_{g_{\R^{3}}}.$$
Therefore, we focus on the boundary term. But as we saw above in $(\ref{eq:B2})$
$$\int_{\partial M}v_{n}^{4} d\sigma_{g}=\int_{\partial \R^{3}_{+}}U_{\infty}^{4}\ d\sigma_{g_{\R^{3}}}+o(1).$$
So we estimate
\begin{align}
\int_{\partial M} v_{n}B_{g}v_{n}\ d\sigma_{g}&=R_{n}^{-1}\int_{\partial M}\beta\frac{\partial \beta}{\partial \nu}\sigma_{n}^{*}(U_{\infty})^{2}\ d\sigma_{g}+R_{n}^{-2}\int_{\partial M}\beta^{2} \sigma_{n}^{*}(U_{\infty} B_{g_{n}}U_{\infty})\ d\sigma_{g}\notag\\
&=H_{1}+H_{2}.
\end{align}
Now,
\begin{align}
|H_{1}|&\leq CR_{n} \int_{\Big( B_{3R_{n}^{-1},+}^{0}\setminus B_{\frac{1}{2}R_{n},+}^{0}\Big)\cap \partial \R^{3}_{+}}U_{\infty}^{2}\ d\sigma_{g_{\R^{3}}}\notag\\
&\leq C \Big(\int_{\partial \R^{3}_{+}\setminus B_{\frac{1}{2}R_{n},+}^{0}}U_{\infty}^{4}\ d\sigma_{g_{\R^{3}}}\Big)^{\frac{1}{2}}\notag\\
&=o(1)\notag
\end{align}
Next,
\begin{align}
H_{2}&=R_{n}^{-2}\Big(\int_{\partial M}\beta^{2}\sigma_{n}^{*}(U_{\infty} B_{\R^{3}}U_{\infty})\ d\sigma_{g}+\int_{\partial M\cap B_{RR_{n}}(x_{n})}\beta^{2}\sigma_{n}^{*}(U_{\infty}(B_{g_{n}}-B_{g_{\R^{3}}})U_{\infty})\ d\sigma_{g}\notag\\
&\quad+\int_{\partial M\setminus B_{RR_{n}}(x_{n})}\beta^{2}\sigma_{n}^{*}(U_{\infty}(B_{g_{n}}-B_{g_{\R^{3}}})U_{\infty})\ d\sigma_{g}\Big).
\end{align}

But,
$$R_{n}^{-2}\int_{\partial M}\beta^{2}\sigma^{*}_{n}U_{\infty}B_{g_{\R^{3}}}U_{\infty}\ d\sigma_{g}=b\int_{\partial \R^{3}_{+}}\rho^{*}_{n}\beta^{2} U_{\infty}^{4}\ d\sigma_{g_{n}}\to b\int_{\partial \R^{3}_{+}}|U_{\infty}|^{4}\ d\sigma_{g_{\R^{3}}}.$$
and
\begin{align}
&\frac{1}{R_{n}^{2}}\left\vert\int_{\partial M\cap B_{RR_{n}}(x_{n})}\beta^{2}\sigma_{n}^{*}(U_{\infty}(B_{g_{n}}-B_{g_{\R^{3}}})U_{\infty}) d\sigma_{g}+\int_{\partial M\setminus B_{RR_{n}}(x_{n})}\beta^{2}\sigma_{n}^{*}(U_{\infty}(B_{g_{n}}-B_{g_{\R^{3}}})U_{\infty}) d\sigma_{g}\right\vert\notag\\
&\leq C\Big(\int_{B_{3R,+}^{0}\cap \partial \R^{3}_{0}}|U_{\infty}(B_{g_{n}}-B_{g_{\R^{3}}})U_{\infty}|\ d\sigma_{g_{\R^{3}}}+\int_{\partial \R^{3}_{+}\setminus B_{3R_{n}^{-1},+}^{0}}|U_{\infty}||(B_{g_{n}}-B_{g_{\R^{3}}})U_{\infty})|\ d\sigma_{g_{\R^{3}}}\notag\\
&\leq C\|U_{\infty}\|_{L^{4}(\partial \R^{3}_{+})}\Big(\|(B_{g_{n}}-B_{g_{\R^{3}}})U_{\infty}\|_{L^{\frac{4}{3}}(B_{3R,+}^{0}\cap \partial \R^{3}_{+})}+\|(B_{g_{n}}-B_{g_{\R^{3}}})U_{\infty}\|_{L^{\frac{4}{3}}(\partial \R^{3}_{+}\setminus B_{3R,+}^{0})}\Big).
\end{align}
And one can see as in estimate $J_{2}$, that since $\|(B_{g_{n}}-B_{g_{\R^{3}}})U_{\infty}\|_{L^{\frac{4}{3}}(\partial \R^{3}_{+}\setminus B_{3R,+}^{0})}\to 0$ as $R\to \infty$ and $\|(B_{g_{n}}-B_{g_{\R^{3}}})U_{\infty}\|_{L^{\frac{4}{3}}(\partial \R^{3}_{+}\setminus B_{3R,+}^{0})}\to 0$ as $n\to \infty$, we have
$$H_{2}=\int_{\partial \R^{3}_{0}}U_{\infty}^{4} d\sigma_{g_{\R^{3}}}+o(1).$$

Combining the previous estimates, we see that
$$E^{b}(v_{n},\phi_{n})=E^{b}_{\R^{3}_{+}}(U_{\infty}, \Psi_{\infty})+o(1)\,,$$
and thus
$$E^{b}(\overline{u}_{n},\overline{\psi}_{n})=E^{b}(u_{n},\psi_{n})-E^{b}_{\R^{3}_{+}}(U_{\infty},\Psi_{\infty})+o(1).$$

\end{proof}


\noindent
Now we will prove the following energy lower bound for solutions in $\R^3_{+}$.
\begin{proposition}
There exists $c_{3}>0$ such that if $(U,\Psi)\in D^{1}(\R^{3}_{+})\times D^{\frac{1}{2}}(\Sigma \R^{3}_{+})$ is a non trivial solution of
$$\left\{\begin{array}{ll}
-\Delta_{g_{\R^3}} U=|\Psi|^{2}U \\
& \text{ on } \R^{3}_+ .\\
D_{g_{\R^3}}\Psi=|U|^{2}\Psi\\
&\\
B_{g_{\R^{3}}}U=bU^{3}\\
&\text{ on } \partial \R^{3}_{+}.\\
\mathbb{B}_{g_{\R^{3}}}^{\pm}\Psi=0
\end{array}
\right.
$$
Then
\begin{equation}\label{eq:lowboundboundary}
E_{\R^{3_{+}}}(U,\Psi)\geq  c_{3}\,.
\end{equation}
\end{proposition}
\begin{proof}

We let $(u,\psi)$ the pull-back of $(U,\Psi)$ by stereographic projection to the hemisphere $\mathbb S^{3}_{+}$. Then we have
$$E_{\R^{3}_{+}}(U,\Psi)=E_{\mathbb S^{3}_{+}}(u,\psi)=\frac{1}{2}\int_{\mathbb S^{3}_{+}}|u|^{2}|\psi|^{2}\ dv_{g_{0}}+\frac{b}{4}\int_{\partial \mathbb S^{3}_{+}}u^{4}\ d\sigma_{g_{0}}.$$
Also, we have
$$\lambda^{+}_{CHI}( \mathbb S^{3}_{+},\partial  \mathbb S^{3}_{+},[g_{0}])\leq \Big(\int_{ \mathbb S^{3}_{+}}u^{6}\ dv_{g_{0}}\Big)^{\frac{1}{3}}.$$
Hence, since
$$Y(S^{3}_{+},\partial  \mathbb S^{3}_{+},[g_{0}])\Big(\int_{ \mathbb S^{3}_{+}}u^{6}\ dv_{g_{0}}\Big)^{\frac{1}{3}}\leq \int_{ \mathbb S^{3}_{+}}|\nabla u|^{2} +\frac{R}{8}u^{2}\ dv_{g_{0}},$$
we have
$$\lambda^{+}_{CHI}( \mathbb S^{3}_{+},\partial   \mathbb S^{3}_{+},[g_{0}])Y( \mathbb S^{3}_{+},\partial  \mathbb S^{3}_{+},[g_{0}])\leq \int_{ \mathbb S^{3}_{+}}|\nabla u|^{2} +\frac{R}{8}u^{2}\ dv_{g_{0}}.$$
On the other hand,
$$\int_{ \mathbb S^{3}_{+}}|\nabla u|^{2} +\frac{R}{8}u^{2}\ dv_{g_{0}}=2E^{b}_{ \mathbb S^{3}_{+}}(u,\psi)+\frac{b}{2}\int_{\partial  \mathbb S^{3}_{+}}u^{4}d\sigma_{g_{0}}\leq 4E^{b}_{ \mathbb S^{3}_{+}}(u,\psi)$$
Thus,
$$E^{b}_{\R^{3}_{+}}(U,\Psi)\geq \frac{1}{4}\lambda^{+}_{CHI}( \mathbb S^{3}_{+},\partial  \mathbb S^{3}_{+},[g_{0}])Y( \mathbb S^{3}_{+},\partial  \mathbb S^{3}_{+},[g_{0}])=c_{3}.$$

\end{proof}

\noindent
We are now in a position to prove the desired result for Palais-Smale sequences.
\begin{proof}  \emph{of Theorem \eqref{first} }\\
The process described in the previous propositions can be iterated for the boundary bubbles, $\ell$ times, for $(\overline{u}_{n},\overline{\psi}_{n})$. Similarly, as described in \cite{MV3}, one can extact $m$ interior bubbles. Thus we have
$$E^{b}(u_{n},\psi_{n})=E^{b}(u_{\infty},\psi_{\infty})+\sum_{k=1}^{m}E_{\R^{3}}(\tilde{U}_\infty^{k},\tilde{\Psi}_{\infty}^{k})+\sum_{k=1}^{\ell}E_{\R^{3}_{+}}^{b}(U_{\infty}^{k},\Psi_{\infty}^{k})+ o(1),$$
where $(U_\infty^{k},\Psi_\infty^{k})$ are solutions to equations $(\ref{eq:blowup})$ in $\mathbb{R}^{3}_{+}$ and $(\tilde{U}_\infty^{k},\tilde{\Psi}_{\infty}^{k})$ are solutions of $(\ref{eq:blowupint})$ in $\R^{3}$. Notice that an uniform lower bound for the energy of interior bubbles holds, analogous to \eqref{eq:lowboundboundary}.

Then one has (see Remark \ref{rem:threshold})
$$\sum_{k=1}^{m}E_{\R^{3}}(\tilde{U}_\infty^{k},\tilde{\Psi}_{\infty}^{k})+\sum_{k=1}^{\ell}E^{b}_{\R^{3}_{+}}(U_{\infty}^{k},\Psi_{\infty}^{k})\geq m\tilde c_3+\ell c_{3} \geq (m+l)c_3\,,$$
so that we stop the process at least when $c-(m+\ell) c_{3}<\frac{\epsilon_{0}}{2}$. Then combining the previous results and those from \cite{MV3}, we have the existence of $m$ sequences $\tilde{x}_{n}^{1},\cdots \tilde{x}_{n}^{m}$ such that $\tilde{x}_{n}^{k}\to \tilde{x}^{k}\in \overset{\circ}{M}$, $m$ sequences of real numbers $\tilde{R}_{n}^{1},\cdots, \tilde{R}_{n}^{m}$ converging to zero, and $\ell$ sequences $x_{n}^{1},\cdots x_{n}^{\ell}$ such that $x_{n}^{k}\to x^{k}\in \partial M$, $\ell$ sequences of real numbers $R_{n}^{1},\cdots, R_{n}^{k}$ such that
$$u_{n}=u_{\infty}+\tilde{v}_{n}^{1}+\cdots +\tilde{v}_{n}^{m}+v_{n}^{1}+\cdots +v_{n}^{\ell}+o(1) \text{ in } H^{1}(M),$$
$$\psi_{n}=\psi_{\infty}+\tilde{\phi}_{n}^{1}+\cdots +\tilde{\phi}_{n}^{m}+\phi_{n}^{1}+\cdots +\phi_{n}^{\ell}+o(1) \text{ in } H^{\frac{1}{2}}(\Sigma M),$$
where
$$
\tilde{v}_{n}^{k}=(\tilde{R}_{n}^{k})^{-\frac{1}{2}}\tilde{\beta}_{k}\tilde{\sigma}_{n,k}^{*}(\tilde{U}_\infty^{k}) ,$$
$$\tilde{\phi}_{n}^{k}=(\tilde{R}_{n}^{k})^{-1}\tilde{\beta}_{k}\tilde{\sigma}_{n,k}^{*}(\tilde{\Psi}_\infty^{k}) ,$$
$$v_{n}^{k}=(R_{n}^{k})^{-\frac{1}{2}}\beta_{k}\sigma_{n,k}^{*}(U_\infty^{k}) ,$$
$$\phi_{n}^{k}=(R_{n}^{k})^{-1}\beta_{k}\sigma_{n,k}^{*}(\Psi_\infty^{k}) ,$$

with $\tilde{\sigma}_{n,k}=(\tilde{\rho}_{n,k})^{-1}$ and $\tilde{\rho}_{n,k}(\cdot)=exp_{\tilde{x}_{n}^k}(\tilde{R}_{n}^k \cdot)$ and $\sigma_{n,k}=(\rho_{n,k})^{-1}$ and $\rho_{n,k}(\cdot)=\mathcal{F}_{x_{n}^k}(R_{n}^k \cdot).$
Also $\beta_{k}$ are smooth compactly supported functions, such that $\tilde{\beta}_{k}=1$ on $B_{1}(\tilde{x}^{k})$ and $supp(\tilde{\beta}_{k})\subset B_{2}(\tilde{x}^{k})$.
\end{proof}

\section{Classification of ground-state bubbles}\label{sec:groundstates}
In this Section we prove Theorem \ref{thm:classification}, as well as \eqref{eq:energygap}.

\begin{lemma}\label{lem:lowerbubbles}
Let $(u,\psi)\in H^1(\Sph^3_+)\times H^{1/2}_{+}(\Sph^3_+)$ be a non trivial solution to \eqref{eq:blowup}, then
\begin{equation}\label{eq:lowerbubbles}
E_{\Sph^3_+}(u,\psi)\geq \frac{1}{2}Y(\Sph^3_+,\partial\Sph^3_+,[g_0])\lambda^{+}_{\text{CHI}}(\Sph^3_+,\partial\Sph^3_+,[g_0])\,,\end{equation}
where the conformal invariants on the right-hand side are as in \eqref{eq:yamabeinv}, \eqref{eq:chiralbag}.
\end{lemma}
\begin{proof}
Recalling the variational characterization of the above invariants, we get
\be\label{eq:upperbound}
\displaystyle
Y(\Sph^n_+,\partial\Sph^n_+,[g_0])\leq \frac{\int_{\Sph^3_+}u L_{g_0}u\, dv_{g_0}}{\left(\int_{\Sph^3_+}\vert  u\vert^6\, dv_{g_0}\right)^{1/3}}=\frac{\int_{\Sph^3_+}u^2 \vert\psi\vert^2\, dv_{g_0}}{\left(\int_{\Sph^3_+}\vert  u\vert^6\, dv_{g_0}\right)^{1/3}}\,,
\ee
and
\[
\displaystyle
\lambda^{+}_{\text{CHI}}(\Sph^3_+,\partial\Sph^3_+,[g_0])\leq \frac{\left(\int_{\Sph^3_+}\vert D\psi\vert^{3/2}\, dv_{g_0} \right)^{4/3}}{\left\vert \int_{\Sph^3_+}\langle D\psi,\psi\rangle \, dv_{g_0} \right\vert}=\frac{\left(\int_{\Sph^3_+} \vert u\vert^3 \vert\psi\vert^{3/2}\, dv_{g_0} \right)^{4/3}}{\int_{\Sph^3_+}u^2 \vert\psi\vert^2 \, dv_{g_0} } \,.
\]
Then H\"older's inequality gives
\[
\lambda^{+}_{\text{CHI}}(\Sph^3_+,\partial\Sph^3_+,[g_0])\leq \left(\int_{\Sph^3_+}\vert u\vert^6 \, dv_{g_0}\right)^{1/3}\,.
\]
The claim then follows by observing that for a solution to \eqref{eq:blowup} there holds
\[
E_{\Sph^3_+}(u,\psi)=\frac{1}{2}\int_{\Sph^3_+}u^2 \vert\psi\vert^2\, dv_{g_0} \,.
\]
\end{proof}

\begin{proof}[Proof of Theorem \ref{thm:classification}]
Let $(u,\psi)\in H^1(\Sph^3_+)\times H^{1/2}_+(\Sph^3_+)$ be a ground state solution to \eqref{eq:bubblessphere}, with $u\geq0$. Then
\be\label{eq:level}
2E_{\Sph^3_+}(u,\psi)=\int_{\Sph^3_+}u^2\vert\psi\vert^2\, dv_{g_0}=Y(\Sph^3_+,\partial\Sph^3_+,[g_0])\lambda^{+}_{\text{CHI}}(\Sph^3_+,\partial\Sph^3_+,[g_0])
\ee

Moreover, observe that by the maximum principle actually $u>0$, so that the formula
\[
g:= \frac{4}{9} u^4g_0\,,\qquad \mbox{on $\Sph^3_+$}
\]
defines a proper conformal change of metric. Consider the associated isometry of spinor bundles
\[
F:\Sigma_{g_0}\Sph^3_+\to \Sigma_{g}\Sph^3_+\,.
\]
Setting
\[
v:=\sqrt{\frac{3}{2}}u\,,\qquad \varphi:= \frac{3}{2}u^{-2}F(\psi)\,,
\]
by \eqref{eq:conformaldirac} we deduce that $v$ and $\varphi$ solve \eqref{eq:bubblessphere} in the new metric $g$, namely
\begin{equation}\label{eq:bubblesnewmetric}
\left\{\begin{array}{ll}
L_{g}v=v|\varphi|^{2}\\
D_{g}\varphi=\varphi \qquad \mbox{on $\Sph^3_+$}\\
\frac{\partial u}{\partial \nu}=\frac{1}{2}h_gu^3\\
\mathbb{B}^{+}\psi=0\qquad \mbox{on $\partial\Sph^3_+$}
\end{array}
\right.
\end{equation}
where $h_g$ is the mean curvature of $\partial\Sph^3_+$ in the metric $g$. By \eqref{eq:upperbound} we find
\[
Y(\Sph^n_+,\partial\Sph^n_+,[g_0])\left(\int_{\Sph^3_+}\vert  u\vert^6\, dv_{g_0}\right)^{1/3}\leq \int_{\Sph^3_+}u^2 \vert\psi\vert^2\, dv_{g_0}\,,
\]
and combining this fact with \eqref{eq:level} we get
\[
\lambda^{+}_{\text{CHI}}(\Sph^3_+,\partial\Sph^3_+,[g_0])\geq \left(\int_{\Sph^3_+}\vert  u\vert^6\, dv_{g_0}\right)^{1/3}
\]
The Hijazi inequality for manifold with boundary \cite{raulot} then gives
\be\label{eq:equality}
\begin{split}
\frac{3}{8}Y(\Sph^n_+,\partial\Sph^n_+,[g_0])& \leq \vert\lambda_{\text{CHI}}\vert\vol(\Sph^3_+,g)  \\
& = \left(\int_{\Sph^3_+}\vert  u\vert^6\, dv_{g_0}\right)^{2/3}\leq \lambda^{+}_{\text{CHI}}(\Sph^3_+,\partial\Sph^3_+,[g_0])^2
\end{split}
\ee
where $\lambda_{\text{CHI}}$ is the smallest eigenvalue of $(D,\B^+)$. Since the first and the last term in \eqref{eq:equality} coincide, we have equality in the Hijazi inequality and $\varphi$ is a Killing spinor:
\[
\nabla^g_X\varphi =\alpha X\cdot \varphi\,,\qquad\forall X\in \Gamma(T\Sph^3_+)\,.
\]
Testing the above condition with vectors $e_i, i=1,2,3$ of a local tangent frame and summing up one gets
\[
D^g\varphi = \sum^3_{i=1}e_i\cdot \nabla^g_{e_i}\varphi=-3\alpha\varphi\,,
\]
by the Clifford anti-commutation rules. Thus $\alpha=-\frac{1}{2}$, using the second equation in \eqref{eq:bubblesnewmetric}. Moreover, observe that
\[
\nabla^g_X \varphi=-\frac{1}{2}X\cdot\varphi=-\frac{1}{3}X\cdot D^g\varphi\,,\qquad\forall X\in\Gamma(T\Sph^3_+)\,,
\]
and $\varphi$ is a twistor spinor, so that by \cite[Prop. A.2.1]{diracspectrum}
\[
(D^g)^2\varphi=\frac{9}{4}\varphi=\frac{3}{8}R_g\varphi\,,
\]
so that $R_g=6$. Then \cite[Thm. 4.1, 4.2]{escobar} implies that there exists an isometry
\[
f:(\Sph^3_+,g)\to(\Sph^3_+,g_0)
\]
such that $f^*g_0=g=u^4g_0$. We thus conclude that
\[
\dd\vol_{f^*g_0}=\det(\dd f)=\left(\frac{8}{27}u^6 \right)\dd\vol_{g_0}
\]
and
\[
u=\sqrt{\frac{3}{2}}\det(\dd f)^{1/6}\,.
\]
Observe that $f$ induces an isometry of spinor bundles $\Lambda$ and the following diagram commutes
\begin{center}
 \begin{tikzcd}
  (\Sigma_{f^*g_0}\Sph^3_+, f^*g_0) \arrow[r, "\Lambda"] \arrow[d]& (\Sigma_{g_0}\Sph^3_+,g_0) \arrow[d] \\
  (\Sph^3_+,f^*g_0) \arrow[r, "f"] & (\Sph^3_+,g_0)
 \end{tikzcd}\,,
\end{center}
the vertical arrows denoting the projections defining the spinor bundles. Notice that, with an abuse of notation, we denote the metrics on the spinor bundles with the same symbol as the manifold metrics. Then the spinor $\Psi:=\Lambda\circ\varphi\circ f^{-1}$ is $-\frac{1}{2}$-Killing with respect to the round metric $g_0$, as this property is preserved by isometries. Finally, we can rewrite
\be\label{eq:pullbackspinor}
\varphi=\Lambda^{-1}\circ\Psi\circ f=: f^*\Psi\,,
\ee
and then
\[
\psi=\frac{2}{3}u^2 F^{-1}(\varphi)=\det(\dd f)^{1/3} F^{-1}(f^*\Psi)\,.
\]

\end{proof}

\begin{proof}[Proof of Corollary \ref{cor:euclideanbubbles}]
Formula \eqref{eq:scalarpart} is obtained applying a conformal transformation of $\Sph^3_+ $ to a standard bubble for the Yamabe problem on the half-space. Then we only have to deal with the spinorial part \eqref{eq:spinorpart}.

Notice that Killing spinors on $\Sph^3_+$ are mapped to $\R^3_+$ via the stereographic projection to get
\[
\Psi(x)=\left(\frac{2}{1+\vert \tilde{x}\vert^2+x^2_3} \right)^{3/2}(1-(\tilde{x},x_3))\cdot\Phi\,,\qquad x=(\tilde{x},x_3)\in\R^3_+\,,
\]
and the parallel spinor $\Phi$ can be chosen to meet chiral bag boundary conditions. Our aim is to prove that applying conformal transformations of $(\R^3_+,g_{\R^3_+})$ we get another spinor of the form \eqref{eq:spinorpart}, for a suitable choice of parameters.

Given $y\in\R^2$, $\lambda>0$, consider the map $f_{y,\lambda}:\R^3_+\to\R^3_+$ defined as
\[
f(\tilde{x},x_3)=\lambda^{-1}(\tilde{x}-y,x_3)\,,
\]
so that the pull-back metric reads $f^*_{y,\lambda}g_{\R^3_+}=\lambda^{-2}g_{\R^3_+}$. Observe that since the frame bundle is trivial $P_{SO}(\R^3_+,g_{\R^3_+})=\R^3_+\times SO(3)$ we get an induced map
\[
\tilde{f}_{y,\lambda}(x,v_1,v_2,v_3)=(f_{y,\lambda}(x),v_1,v_2,v_3)\,,
\]
acting as the identity on the $SO(3)$ factor. Then such maps lifts to $P_{Spin}(\R^3_+,g_{\R^3_+})=\R^3_+\times Spin(3)$, still acting as the identity on the second component. The spinor bundle is also trivial, so that we obtain a map
\[
\Lambda_{y,\lambda}:\R^3_+\times\C^2\to\R^3_+\times\C^2
\]
and the transformation on $\Psi$ reads
\[
\begin{split}
\psi(x)&=\beta_{\lambda^{-2}g_{\R^3_+},g_{\R^3_+}}\Lambda^{-1}_{y,\lambda}\Psi(f_{y,\lambda}(x)) \\
& = \left(\frac{2\lambda}{\lambda^2+\vert \tilde{x}-y\vert^2+x^2_3} \right)^{3/2}F_{\lambda^{-2}g_{\R^3_+},g_{\R^3_+}}\Lambda^{-1}_{y,\lambda}\left(1-\left(\frac{\tilde{x}-y}{\lambda},\frac{x_3}{\lambda} \right) \right)\cdot\Phi_0.
\end{split}
\]
Here $F_{\lambda^{-2}g_{\R^3_+},g_{\R^3_+}}$ denotes the identification of spinors for conformally related metrics. The above discussion shows that $F_{\lambda^{-2}g_{\R^3_+},g_{\R^3_+}}\Lambda^{-1}_{y,\lambda}$ can be taken to be the identity map, thus proving \eqref{eq:spinorpart}. Let us now consider the effect of a rotation. Take $R\in SO(2)$, identifying it with a rotation preserving $\partial\R^3_+=\R^2$. By the previous discussion, such map lifts to $\Lambda_R: \Sigma\R^3_+\to \Sigma\R^3_+$. Now, considering a spinor of the form \eqref{eq:spinorpart} we get
\[
\begin{split}
\Lambda^{-1}_R(\psi(Rx))&=\left(\frac{2\lambda}{\lambda^2+\vert R\tilde{x}-y\vert^2+x^2_3} \right)^{3/2}F_{\lambda^{-2}g_{\R^3_+},g_{\R^3_+}}\Lambda^{-1}_{R}\left(1-\left(\frac{R\tilde{x}-y}{\lambda},\frac{x_3}{\lambda} \right) \right)\cdot\Phi_0 \\
&=\left(\frac{2\lambda}{\lambda^2+\vert \tilde{x}-R^{-1}y\vert^2+x^2_3} \right)^{3/2}\left(1-\left(\frac{\tilde{x}-R^{-1}y}{\lambda},\frac{x_3}{\lambda} \right) \right)\cdot \Lambda^{-1}_R(\Phi_0)\,,
\end{split}
\]
since for given $v\in\R^3$, $\varphi\in\Sigma\R^3_+$ we have $\Lambda_R(v\cdot\varphi)=(Rv)\cdot \Lambda_R(\varphi)$. The proof is concluded as we obtained a spinor of the form \eqref{eq:spinorpart}, with parameters $R^{-1}y\in\R^3_+$, $\lambda>0$ and $\Lambda^{-1}_R(\Phi_0)\in\Sigma\R^3_+$.
\end{proof}


\section{Aubin Type inequality}\label{sec:aubin}
This section is devoted to the proof of Theorem \ref{Aubin type result}.
\smallskip

To this aim, recall the definition of the functionals
$$\tilde{E}(u,\psi)=\frac{\left(\displaystyle\int_{M}uL_{g}udv_{g}\right)\left(\displaystyle\int_{M}\langle D_g \psi,\psi\rangle dv_{g}\right)}{\displaystyle\int_{M}|u|^{2}|\psi|^{2}dv_{g}} ,\qquad I(\psi)=\frac{\displaystyle\int_{M}\langle D_g \psi,\psi\rangle dv_{g}}{\displaystyle\int_{M}|u|^{2}|\psi|^{2}dv_{g}} \; $$

and of the conformal constant

\begin{equation}\label{eq:defYtilde}
\widetilde{Y} (M, \partial M, [g])=\inf\left\{\begin{array}{ll}
\tilde{E}(u,\psi); \text{ where  $(u,\psi)\in H^{1}(M)\setminus\{0\}\times H^{\frac{1}{2}}_+(\Sigma M)\setminus \{0\}$ s.t. }\\
\\
I(\psi)> 0, \quad  P^{-}\left(D_g \psi-I(\psi)u^{2}\psi\right) =0\end{array}  \right\} \;  .
\end{equation}

First of all, we need a characterization of the first eigenvalue of the Dirac operator in terms of the following minimization problem: for a given $u>0$ and smooth, we consider
\begin{equation}\label{eq:lambda}
\tilde{\lambda}_u=\inf\left\{I(\psi); \text{ for  $\psi\in H^{\frac{1}{2}}_+(\Sigma M)$ s.t. }  I(\psi)> 0, \; P^{-}\left(D_g \psi-I(\psi)u^{2}\psi\right) =0 \right\} \,,
\end{equation}
then we have
\begin{proposition}\label{tildelambda}
Let $(M,g)$ be a compact oriented three-dimensional Riemannian manifold with boundary.
For any given $u>0$ and smooth, we have that $\tilde{\lambda}_u>0$. Moreover, the minimization problem is achieved and $\tilde{\lambda}_u$ coincide with the first eigenvalue $\lambda_{1}^+(g_u)$ for the Dirac operator $D_{g_{u}}$, under chiral bag boundary conditions, where $g_u=\frac{4}{9}u^4g$.
\end{proposition}
\begin{proof}
Let $\psi_{n}$ be a minimizing sequence, $I(\psi_{n})\to \tilde{\lambda}_u$.  We can assume
\begin{equation}\label{eq:normalization}
\int_{M}u^{2}|\psi_n|^{2}dv_{g}=1\,,
\end{equation} Then we have
$$I(\psi_{n})=\|\psi_{n}^{+}\|^{2}-\|\psi_{n}^{-}\|^{2} .$$
Since
$$-\|\psi_{n}^{-}\|^{2}=\int_M\langle D_g\psi_n,\psi^-_n\rangle= I(\psi_{n})\int_{M}u^{2}\langle \psi_n,\psi_n^{-}\rangle dv_{g},$$
by the last condition in \eqref{eq:lambda}. Then, using H\"{o}lder's inequality and \eqref{eq:normalization}, we have
$$\|\psi_{n}^{-}\|^{2}\leq I(\psi_{n})\left(\int_{M}u^{2}|\psi_n^{-}|^{2}dv_{g}\right)^{\frac{1}{2}}.$$
Also, since the projector $P^{-}:H^{\frac{1}{2}}_+(\Sigma M)\to H^{\frac{1}{2},-}_+(\Sigma M)$ is a zero-order operator, we have
\begin{align}
\|\psi_{n}^{-}\|^{2}&\leq I(\psi_{n}) \sup(u) \|\psi^{-}\|_{L^{2}}\notag\\
&\leq I(\psi_{n}) \sup(u) \|\psi\|_{L^{2}} \notag\\
&\leq I(\psi_{n}) \frac{\sup(u)}{\inf(u)} \left(\int_{M}u^{2}|\psi_n|^{2}dv_{g}\right)^{\frac{1}{2}}\notag\\
&\leq C I(\psi_{n}),\notag
\end{align}
where $C=C(u)=\frac{\sup(u)}{\inf(u)}$ and then
$$\|\psi_{n}^{+}\|^{2}=I(\psi_{n})+\|\psi_{n}^{-}\|^{2}\leq (C+1) I(\psi_{n}).$$
Now, if $I(\psi_{n})\to 0$, we would have that $\psi_{n}\to 0$ in $H^{\frac{1}{2}}_+(\Sigma M)$, but $\int_{M}u^{2}|\psi_n|^{2}dv_{g}=1$. Therefore $\tilde{\lambda}_u>0$. Moreover if $\psi_{n}$ is any minimizing sequence, then $\|\psi_{n}\|\geq \delta$, for some $\delta>0$. In order to prove the existence of a minimizer, we will follow the argument in \cite{Pan} and \cite{Sul}. So, we define
$$S=\{\psi\in H^{\frac{1}{2}}_+(\Sigma M);P^{-}(D\psi-I(\psi)u^{2}\psi)=0\} , $$
and we claim that $S$ is a manifold. Let us consider the operator
$$F:H^{\frac{1}{2}}_+(\Sigma M)\to H^{\frac{1}{2},-}_+, \qquad F(\psi)=P^{-}(D_g\psi-I(\psi)u^{2}\psi),$$
so that $S=F^{-1}(0)$; therefore, if $DF(\psi)$ is onto for all $\psi\in S$, then $S$ will be indeed a manifold. We compute its differential $DF$:
$$DF(\psi)h=P^{-}(D_gh-I(\psi)u^{2}h)-\Big(DI(\psi)h\Big)P^{-}(u^{2}\psi).$$
Notice that $$DI(\psi)h=2\frac{\int_{M}\langle D_{g}\psi-I(\psi)u^{2}\psi,h\rangle \ dv_{g}}{\int_{M}u^{2}|\psi|^{2}\ dv_{g}}.$$
Hence, if we take $\psi \in S$ and restrict $h$ to $H^{\frac{1}{2},-}_+$, we get
$$\langle DF(\psi)h,h\rangle=-\|h\|^{2}-I(\psi)\int_{M}u^{2}|h|^{2}dv_{g} .$$
Therefore $DF(\psi)$ is definite negative on $H^{\frac{1}{2},-}_+$ and hence invertible on that space, so $S$ is a manifold. Now, we can find a minimizing sequence $\psi_{n}$ for the functional $I$ that is a (PS) sequence on $S$ (by Ekeland's variational principle): we will show that it is still a (PS) sequence in $H^{\frac{1}{2}}_+(\Sigma M)$. So, if we set $DI(\psi_{n})=\varepsilon_{n}$, we have that its tangential component $\varepsilon_{n}^{T}$, on the tangent space $T_{\psi_{n}}S$ of the manifold $S$, converges to zero since $\psi_{n}$ is a (PS) sequence for $I$ on $S$. Regarding its normal component, we notice that since the operator $DF(\psi_{n})_{|H^{\frac{1}{2},-}_{+}}:H^{\frac{1}{2},-}_+\to H^{\frac{1}{2},-}_+$ is invertible, it has an inverse, namely $K_{n}:H^{\frac{1}{2},-}_{+} \to H^{\frac{1}{2},-}_{+}$ satisfying $K_{n}\circ DF(\psi_{n})_{|H^{\frac{1}{2},-}_{+}}=Id_{H^{\frac{1}{2},-}_{+}}$ , such that $\|K_{n}\|_{op}\leq C$, since we have that $\int_{M}u^{2}|\psi_{n}|^{2}dv_{g}=1$. In this way, the operator $P_{n}=K_{n}\circ DF(\psi_{n})$ is a projector on $H^{\frac{1}{2},-}$ parallel to $T_{\psi_{n}}S$: indeed, we have that if $h\in H^{\frac{1}{2},-}_{+}$ then by construction $P_{n}h=h$ and $T_{\psi_{n}}S=\ker DF(\psi_{n})$; we consider then $P_{n}^{*}$, the adjoint of $P_{n}$, which is a projector on the normal space $N_{\psi}S$ of the manifold $S$, parallel to $H^{\frac{1}{2},+}_{+}$. Now, since $\langle \varepsilon_{n}, \varphi \rangle =0$ for all $\varphi\in H^{\frac{1}{2},-}_{+}$, therefore $\varepsilon_{n}\in H^{\frac{1}{2},+}_{+}$ and so we have $\varepsilon_{n}=(Id-P_{n}^{*})\varepsilon_{n}^{T}$ (see Figure \ref{Proj}). Indeed, $P_{n}^{*}\varepsilon_{n}=0$ and $P_{n}^{*}\varepsilon_{n}^{\perp}=\varepsilon_{n}^{\perp}$. Hence,
$$(Id-P_{n}^{*})\varepsilon_{n}^{T}=(Id-P_{n}^{*})(\varepsilon_{n}-\varepsilon_{n}^{\perp})=\varepsilon_{n}-\varepsilon_{n}^{\perp}+\varepsilon_{n}^{\perp}.$$

 This shows that $\varepsilon_{n}\to 0$ and $\psi_{n}$ is a (PS) sequence for $I$ in $H^{\frac{1}{2}}_+(\Sigma M)$.

\begin{figure}[h]
\centering
\includegraphics[scale=.3]{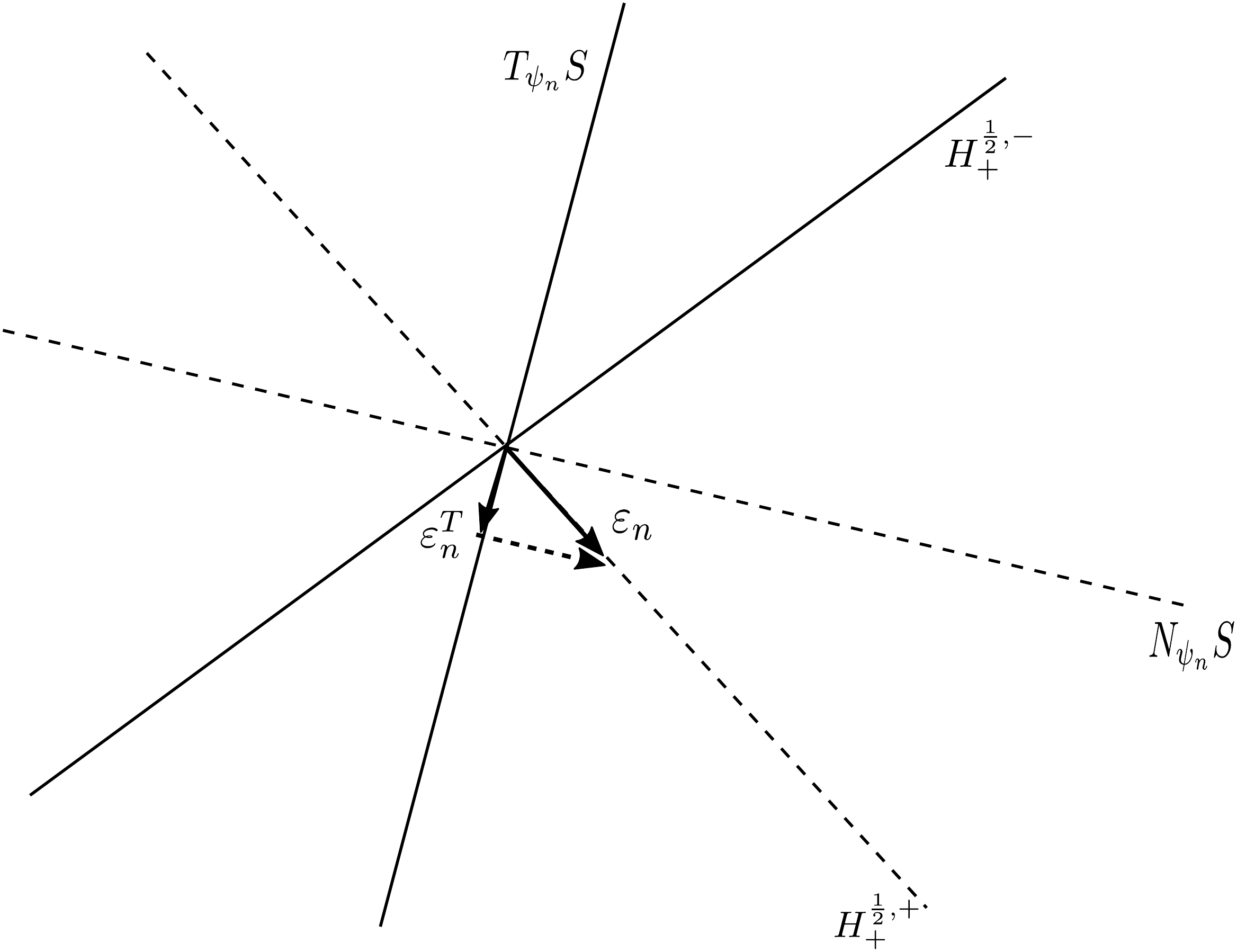}
\caption{Projection of $\varepsilon_{n}^{T}$.}
\label{Proj}
\end{figure}

 In particular this (PS) sequence $\psi_{n}$ satisfies
\begin{align}
\|\psi_{n}^{+}\|^{2}&\leq \int_{M}u^{2}|\psi^{+}_{n}||\psi_{n}|dv_{g}+o(\|\psi_{n}^{+}\|)\leq C +o(\|\psi_{n}^{+}\|) .\notag
\end{align}
Since a similar inequality holds also for $\psi_{n}^{-}$, then the (PS) sequence $\psi_{n}$ is bounded in $H^{\frac{1}{2}}_+(\Sigma M)$; therefore up to subsequences, there exists $\psi \in H^{\frac{1}{2}}_+(\Sigma M)$ such that $\psi_{n}\rightharpoonup \psi$ weakly in $H^{\frac{1}{2}}_+(\Sigma M)$ with $\psi$ satisfying
$$D_g\psi=\tilde{\lambda}_uu^{2}\psi.$$
Finally, since $\tilde{\lambda}_u$ is a minimum, by the conformal change $g_u= \frac{4}{9}u^{4}g$, we have that $$\tilde{\lambda}_u=\lambda_{1}^+(g_u).$$
\end{proof}

\noindent
We can now prove the Aubin-type inequality stated in Theorem \ref{Aubin type result}.
\begin{proof}[Proof of Theorem \ref{Aubin type result}]
We start by proving the first inequality in the formula (\ref{Aubin inequality}). By the Sobolev inequality and Proposition (\ref{tildelambda}), we have for any non-trivial $(u,\psi)\in H^{1}(M) \times H^{\frac{1}{2}}_+(\Sigma M)$:
\begin{align}
\tilde{E}(u,\psi)&\geq Y(M,\partial M,[g])\left(\int_{M}u^{6}dv_{g}\right)^{\frac{1}{3}}\frac{\displaystyle\int_{M}\langle D_g \psi,\psi\rangle dv_{g}}{\displaystyle\int_{M}|u|^{2}|\psi|^{2}dv_{g}}\notag\\
&\geq Y(M,\partial M,[g])\lambda^{+}_{1}(g_{u})Vol(g_{u})^{\frac{1}{3}}\notag \\
&\geq Y(M,\partial M,[g])\lambda^{+}_{\text{CHI}}(M,\partial M,[g]).
\end{align}
where $g_u=\frac{4}{9}u^4g$, and $ Y(M,\partial M,[g])$ is the Yamabe invariant \eqref{eq:yamabeinv}. By taking the infimum on the first side, we get the first inequality in (\ref{Aubin inequality}). In order to prove the other inequality, let us consider a spinor $\overline{\Psi}_{0}\in\Sigma\R^3_+$ such that $\overline{\Psi}_{0}$ meets the chiral bag boundary conditions. We define the spinor
$$\Psi_{0}(x)=\left(\frac{2}{1+\vert \tilde{x}-y\vert^2+x^2_3}\right)^{3/2}\left(1-\left(\tilde{x}-y,x_3 \right) \right)\cdot \overline{\Psi}_{0} $$
and the bubble
$$U_{0}=\left(\frac{2}{1+\vert \tilde{x}-y\vert^2+x^2_3}\right)^{1/2}$$
where $x=(\tilde{x},x_3)\in\R^3_+$, with $\tilde{x}\in\R^2$ and $x_3\geq0$.
Then $(U_{0},\Psi_{0})\in H^1(\R^3_+)\times H^{1/2}_+(\Sigma\R^3_+)$ is a ground state solution to \eqref{eq:blowup}, that is $(U_{0},\Psi_{0})$ is a critical point for $E_{\R^{3}_+}$ and
$$E_{\R^{3}_+}(U_{0},\Psi_{0})=\frac{1}{2}\tilde{Y}(\Sph^3_+,\partial\Sph^3_+,[g_0]).$$
Now, let $x_{0}\in M$ and $\lambda>0$, small enough. We set $\rho_{\lambda}(x)=exp_{x_{0}}(\lambda x)$, similar to \eqref{e5} and \eqref{e6}, and we consider $\beta$ a cut-off function supported in $B_{2}(x_{0})$ and such that $\beta\equiv 1$ on $B_{1}(x_{0})$. We can therefore define the functions
$$\left\{\begin{array}{ll}
u_{\lambda}=\lambda^{-\frac{1}{2}}\beta\sigma_{\lambda}^{*}(U_{0}),\\
\\
\psi_{\lambda}=\lambda^{-1}\beta\sigma_{\lambda}^{*}(\Psi_{0}),
\end{array}
\right.
$$
where $\sigma_{\lambda}=\rho_{\lambda}^{-1}$. Arguing as in Lemma \ref{lemma1}, Lemma \ref{lemma2}, Lemma \ref{lemma3} one can show that, as $\lambda \to 0$:
$$dE(u_{\lambda},\psi_{\lambda})\to 0,$$
and
$$\int_{M}u_{\lambda}L_{g}u_{\lambda}dv_{g}=\int_{\R^{3}_+}|\nabla U_{0}|^{2}dv_{g_{\R^{3}}}+o(1),$$
$$\int_{M}\langle D_g\psi_{\lambda},\psi_{\lambda}\rangle dv_{g}=\int_{\R^{3}_+}\langle D_{g_{\R^3}}\Psi_{0},\Psi_{0}\rangle dv_{g_{\R^{3}}}+o(1),$$
$$\int_{M}|u_{\lambda}|^{2}|\psi_{\lambda}|^{2}dv_{g}=\int_{\R^{3}_+}|U_{0}|^{2}|\Psi_{0}|^{2}dv_{g_{\R^{3}}}+o(1).$$
Moreover, since $b=0$, we do not have boundary terms. Now, in order to take the infimum of $\tilde{E}$, by definition \eqref{Ytilde} we need to have
$$P^{-}(D_g\psi_{\lambda}-I(\psi_{\lambda})|u_{\lambda}|^{2}\psi_{\lambda}) =0.$$
For our test functions this might not be true, therefore we will perturb $\psi_{\lambda}$ so that the previous condition is satisfied. In particular we will show that there exists $h\in H^{\frac{1}{2},-}_{+}$ so that
$$P^{-}(D_g(\psi_{\lambda}+h)-I(\psi_{\lambda}+h)|u_{\lambda}|^{2}(\psi_{\lambda}+h))=0.$$
This is equivalent to solving
$$D_g h-P^{-}[I(\psi_{\lambda})|u_{\lambda}|^{2}h]+B(h)=A_{\lambda},$$
where
$$A_{\lambda}=P^{-}(D_g\psi_{\lambda}-I(\psi_{\lambda})|u_{\lambda}|^{2}\psi_{\lambda}),\quad B(h)=P^{-}\left([I(\psi_{\lambda})-I(\psi_{\lambda}+h)]|u_{\lambda}|^{2}(\psi_{\lambda}+h)\right).$$
Next we define operator
$$T:H^{\frac{1}{2},-}_{+}\to H^{\frac{1}{2},-}_{+}, \quad T(h)=D_gh-I(\psi_{\lambda})P^{-}(|u_{\lambda}|^{2}h)+B(h),$$
and we compute its differential at zero
$$dT(0)\varphi=D_g\varphi-I(\psi_{\lambda})P^{-}(|u_{\lambda}|^{2}\varphi)-\langle dI(\psi_{\lambda}),\varphi\rangle P^{-}(|u_{\lambda}|^{2}\psi_{\lambda}).$$
Now, the operator
$$\varphi\mapsto D\varphi-I(\psi_{\lambda})P^{-}(|u_{\lambda}|^{2}\varphi)$$
is negative definite on $H^{\frac{1}{2},-}_{+}$, hence it is invertible for all $\lambda>0$; moreover we have that
$$\langle dI(\psi_{\lambda}),\varphi\rangle P^{-}(|u_{\lambda}|^{2}\psi_{\lambda})\to 0, \quad \text{as}  \quad \lambda\to 0.$$
Then $dT(0)$ is invertible for $\lambda$ small enough. Since we have that also $A_{\lambda}\to 0$ as $\lambda \to 0$, by the implicit function theorem we get the existence of $h_{\lambda}\in H^{\frac{1}{2},-}_{+}$ such that $T(h_{\lambda})=A_{\lambda}$, moreover $h_{\lambda}\to 0$ as $\lambda\to 0$. Now, as before, we have that as $\lambda \to 0$,
$$\int_{M}\langle D_g(\psi_{\lambda}+h_{\lambda}),(\psi_{\lambda}+h_{\lambda}\rangle\ dv_{g}=\int_{\R^3_+}\langle D_{g_{\R^3}}\Psi_{0},\Psi_{0}\rangle\ dv_{g_{\R^{3}}}+o(1),$$
and
$$\int_{M}|u_{\lambda}|^{2}|\psi_{\lambda}+h_{\lambda}|^{2}\ dv_{g}=\int_{\R^3_+}|U_{0}|^{2}|\Psi_{0}|^{2}\ dv_{g_{\R^{3}}}+o(1).$$
Hence, we have
$$\tilde{E}(u_{\lambda},\psi_{\lambda}+h_{\lambda})=Y(\Sph^3_+,\partial\Sph^3_+,[g_0])\lambda^{+}_{\text{CHI}}(\Sph^3_+,\partial\Sph^3_+,[g_0])+o(1),$$
therefore
$$\tilde{Y} (M, \partial M, [g])\leq \tilde{Y}(\Sph^3_+,\partial\Sph^3_+,[g_0]) ,$$
and this concludes the proof.
\end{proof}

\noindent
We are now in a position to prove our existence result.
\begin{proof}[Proof of Theorem \ref{Aubin type result}]
Let us consider the functional $E$: we want to show that $E$ has the geometry of mountain pass type, in order to apply a min-max argument. Technically, one cannot apply the classical mountain-pass theorem due to the nature of the functional and the restrictions that one has, though we will provide a heuristic reason here on why one should expect the value $\tilde{Y}(M,\partial M, [g])$ to be a critical value for $E$. So, for $t>0$ and $s>0$, we consider the functional $W(t,s,\varphi):=E(tu, s\psi+\varphi)$, where $\varphi \in H^{\frac{1}{2},-}_{+}$. This parametrization is consistent with the generalized Nehari manifold that will be introduced later. Notice that equivalently one could still define $W$ as $W(t,s,\varphi)= E(tu, s\psi^{+}+\varphi)$. Hence, one has
\begin{align}
W(t,s,\varphi)=&\frac{1}{2}\Big(t^{2}\|u\|^{2}+s^{2}\|\psi^{+}\|^{2}-t^{2}s^{2}\int_{M}u^{2}|\psi^{+}|^{2}\ dv_{g}-\|\varphi\|^{2}-t^{2}\int_{M}u^{2}|\varphi|^{2}\ dv_{g}\notag\\
&-2t^{2}s\int_{M}u^{2}\langle \psi^{+},\varphi\rangle \ dv_{g}\Big)\notag\\
&=\frac{1}{2}\Big(t^{2}\|u\|^{2}+s^{2}\|\psi^{+}\|^{2}-t^{2}s^{2}\int_{M}u^{2}|\psi^{+}|^{2}\ dv_{g}+Q_{t,s}(\varphi)\Big),
\end{align}
where $$Q_{t,s}(\varphi)=-\|\varphi\|^{2}-t^{2}\int_{M}u^{2}|\varphi|^{2}\ dv_{g}-2t^{2}s\int_{M}u^{2}\langle \psi^{+},\varphi\rangle \ dv_{g}.$$
Notice that $Q_{t,s}(\cdot)$ is negative definite and strongly concave. Hence, it has a unique maximizer $\tilde{\varphi}:=\tilde{\varphi}(t,s,u,\psi)$ and this maximizer satisfies:
$$D_{g}\tilde{\varphi}-P^{-}(u^{2}\tilde{\varphi}-st^{2}u^{2}\psi^{+})=0.$$
Notice that this is equivalent to $$P^{-}(D_{g}(s\psi^{+}+\varphi)-(tu)^{2}(s\psi^{+}+\varphi))=0,$$
which is in some sense, the constraint that we have in the definition of $\tilde{Y}(M,\partial M, [g])$. Now, if we substitute in $W$, we have
$$\tilde{W}(t,s)=W(t,s,\tilde{\varphi})=\frac{1}{2}\Big(t^{2}\|u\|^{2}+s^{2}\|\psi^{+}\|^{2}-t^{2}s^{2}\int_{M}u^{2}|\psi^{+}|^{2}\ dv_{g}\Big).$$
Thus, one can easily see that there exists $C_{1}$ and $C_{2}$ positive and depending on $u$ and $\psi^{+}$, such that
$$ \tilde{W}(t,s)\geq C_{1}(t^{2}+s^{2})-C_{2}(t^{4}+s^{4}).$$
Therefore, for $s$ and $t$ small enough we have $\tilde{W}(t,s)>c>0$. On the other hand $$\tilde{W}(t,t)\leq t^{2}(\|u\|^{2}+\|\psi^{+}\|^{2})-t^{4}\int_{M}u^{2}|\psi^{+}|^{2}\ dv_{g}.$$
Thus, as long as $\int_{M}u^{2}|\psi^{+}|^{2}\not=0$, we have $\tilde{W}(t,t)\to -\infty$ as $t\to +\infty$. With this we see that we have a sort of a mountain pass geometry.

Next we consider the following min-max problem
$$m=\inf\left\{\begin{array}{ll}
\displaystyle\max_{t\geq 0, s\geq 0, \varphi \in H^{\frac{1}{2},-}_{+}}E(tu,s\psi+\varphi); \text{ where  $(u,\psi)\in H^{1}(M)\setminus\{0\}\times H^{\frac{1}{2}}_+(\Sigma M)\setminus \{0\}$ s.t. }\\
\\
I(\psi)>0 \end{array} \right\},$$
which is equivalent to
$$m=\inf\left\{\begin{array}{ll}
\displaystyle\max_{t\geq 0, s\geq 0}E(tu,s\psi); \text{ where  $(u,\psi)\in H^{1}(M)\setminus\{0\}\times H^{\frac{1}{2}}_+(\Sigma M)\setminus \{0\}$ s.t. }\\
\\
I(\psi)>0; \quad P^{-}\left(D_g \psi-(tu)^{2}\psi\right)=0 \end{array} \right\},$$
Without the orthogonality condition, in the classical case, if $E$ satisfies (PS) at the level $m$ then $m$ would be a critical value; in particular, by a direct computation we have that
$$\max_{t> 0, s> 0}E(tu,s\psi)=\frac{1}{2}\tilde{E}(u,\psi),$$
therefore $2m=\tilde{Y} (M, \partial M, [g])$. This provides the heuristic proof on why one would expect $m$ to be a critical value.\\
Next we explicitly show that indeed we have a critical point at that level by introducing the generalized Nehari manifold:
$$N=\left\{\begin{array}{ll}(u,\psi)\in H^{1}(M)\times H^{\frac{1}{2}}_+(\Sigma M); s.t. \\
\\
\displaystyle\int_{M}uL_{g}u\ dv_{g}=\int_{M}\langle D_g\psi,\psi\rangle\ dv_{g}=\int_{M}|u|^{2}|\psi|^{2}\ dv_{g} \not=0;\\
\\
P^{-}\left(D_g \psi-I(\psi)u^{2}\psi\right)=0
\end{array}\right\} $$
We first show that $N$ is indeed a manifold, so we consider the operator
$$G:H^{1}(M)\times H^{\frac{1}{2}}_+(\Sigma M)\to \R\times \R \times H^{\frac{1}{2},-}_+(\Sigma M),$$
defined by
$$G(u,\psi)=\left[\int_{M}uL_{g}u-|u|^{2}|\psi|^{2}\ dv_{g},\int_{M}\langle D_g\psi-|u|^{2}\psi,\psi\rangle\ dv_{g},P^{-}(D_g\psi-I(\psi)|u|^{2}\psi)\right].$$
In this way, since $N=G^{-1}(0)$, if $DG(u,\psi)$ is onto for all $(u,\psi) \in N$ then $N$ is a manifold. Let $(u_{0},\psi_{0})\in N$, we will show that $DG(u_{0},\psi_{0})$ is invertible if restricted to some special subspace. For $h\in H^{\frac{1}{2},-}_{+}$, we will use the following representation
$$(tu_{0},s\psi_{0}+h)=[t,s,h] \in \R\times \R\times H^{\frac{1}{2},-}_{+},$$
and we will express $DG(u_{0},\psi_{0})$ in this basis. Since $\int_{M}|u_{0}|^{2}|\psi_{0}|^{2}\ dv_{g}\not=0$, we can assume for the sake of simplicity that $\int_{M}|u_{0}|^{2}|\psi_{0}|^{2}dv_{g}=1$. We have then
\begin{align}
DG(u_{0},\psi_{0})[1,0,0]=&[0,-2,2P^{-}(D_g\psi_{0})],\notag\\
DG(u_{0},\psi_{0})[0,1,0]=&[-2,0,0],\notag \\
DG(u_{0},\psi_{0})[0,0,h]=&[2\langle D_g\psi_{0}, h \rangle,0,D_gh-P^{-}(|u_{0}|^{2}h)]\notag.
\end{align}
Now we define the operator $K(h)= D_gh-|u_{0}|^{2}h$, and we see that it is negative definite on $H^{\frac{1}{2},-}$: in fact
$$\langle Kh,h\rangle=-\|h\|^{2}-\int_{M}|u_{0}|^{2}h^{2}dv_{g};$$
hence it is invertible. Now, for $[a,b,c]\in \R\times \R \times H^{\frac{1}{2},-}_{+}$, we want to find $[x_{1},x_{2},w]$ so that
$DG(u_{0},\psi_{0})[x_{1},x_{2},w]=[a,b,c]$, namely we have to solve the following system:
$$\displaystyle\left\{\begin{array}{ll}
a=-2x_{2}+2\langle D_g\psi_{0},w\rangle \\
b=-2x_{1}\\
c=2x_{1}P^{-}(D_g\psi_{0})+K(w)
\end{array}\right. $$
Since $K$ is invertible, we find
$$\left\{\begin{array}{ll}
\displaystyle x_{1}=-\frac{b}{2} \\
\\
\displaystyle x_{2}=-\frac{a}{2}+\langle D_g\psi_{0},K^{-1}(c+bP^{-}(D_g\psi_{0}))\rangle\\
\\
\displaystyle w=K^{-1}(c+bP^{-}(D_g\psi_{0}))
\end{array}\right. $$
Therefore $DG(u_{0},\psi_{0})$ is onto and hence $N$ is a manifold. Now, let us denote by $A(u_{0},\psi_{0})$ the inverse of $DG(u_{0},\psi_{0})_{|\R u_{0} \oplus \R \psi_{0}\oplus H^{\frac{1}{2},-}_{+}}$,
$$A(u_{0},\psi_{0}):\R\times \R \times H^{\frac{1}{2},-}_{+}\to \R u_{0} \oplus \R \psi_{0}\oplus H^{\frac{1}{2},-}_{+},$$
with
$$\|A(u_{0},\psi_{0})\|_{op}\leq C(\|u_{0}\|\|\psi_{0}\|).$$
As in the proof of Proposition \ref{tildelambda}, by using Ekeland's principle, we have the existence of a minimizing (PS) sequence $(u_{n},\psi_{n})\in N$, for $E$ restricted to $N$: we will show that this is indeed a (PS) sequence for $E$ also in $H^{1}(M)\times H^{\frac{1}{2}}_+(\Sigma M)$. We set $DE(u_{n},\psi_{n})=\varepsilon_{n}$ and we have that $\varepsilon_{n}^{T}\to 0$, since it is the tangential part of the (PS) sequence which is a (PS) sequence in $N$. Now we define
$$P_{n}=A(u_{n},\psi_{n})\circ DG(u_{n},\psi_{n})$$
and we notice that it is a projector on $\R u_{0} \oplus \R \psi_{0}\oplus H^{\frac{1}{2},-}_{+}$ parallel to $T_{(u_{0},\psi_{0})}N$. Moreover, since
$$E(u_{n}\psi_{n})=\frac{1}{2}\int_{M}|u_{n}|^{2}|\psi_{n}|^{2}\ dv_{g}\to m,$$
we have that $\|u_{n}\|^{2}\leq C$. Also
$$-\|\psi^{-}_{n}\|^{2}=\int_{M}|u_{n}|^{2}\langle \psi_{n},\psi_{n}^{-}\rangle\ dv_{g}.$$
Therefore
\begin{align}
\|\psi^{-}_{n}\|^{2}&\leq \int_{M}|u_{n}|^{2}|\psi_{n}||\psi_{n}^{-}| dv_{g}\notag\\
&\leq \left(\int_{M}|u_{n}|^{2}|\psi_{n}|^{2}\ dv_{g}\right)^{\frac{1}{2}}\left(\int_{M}|u_{n}|^{2}|\psi_{n}^{-}|^{2}\ dv_{g}\right)^{\frac{1}{2}}\notag\\
&\leq C_{1}\|u_{n}\|_{L^6}\|\psi_{n}^{-}\|_{L^3},\notag
\end{align}
so that
$$\|\psi^{-}_{n}\|\leq C.$$
But we have that
$$\|\psi_{n}^{+}\|^{2}-\|\psi_{n}^{-}\|^{2}=\int_{M}|u_{n}|^{2}|\psi_{n}|^{2}dv_{g},$$
hence
$$\|\psi_{n}^{+}\|^{2}\leq C.$$
Therefore, we have that $P_{n}$ is uniformly bounded. Let now $P_{n}^{*}$ be the adjoint of $P_{n}$; so that $P_{n}^{*}$ is also a projector on $(\R u_{0} \oplus \R \psi_{0}\oplus H^{\frac{1}{2},-}_{+})^{\perp}$ parallel to $\mathcal{N}_{(u_{n},\psi_{n})}N$, the normal space of $N$ at the point $(u_{n},\psi_{n})$. Since
$$\varepsilon_{n}\in (\R u_{0} \oplus \R \psi_{0}\oplus H^{\frac{1}{2},-})^{\perp},$$
hence $\varepsilon_{n}=(Id-P_{n}^{*})\varepsilon_{n}^{T}$ and so $(u_{n},\psi_{n})$ is indeed a (PS) sequence for $E$ in $H^{1}(M)\times H^{\frac{1}{2}}_+(\Sigma M)$; in particular this (PS) sequence is at the energy level $\frac{1}{2}\tilde{Y} (M, \partial M, [g])$. Finally, from the classification of the (PS) sequence Theorem (\ref{first}), if $\tilde{Y} (M, \partial M, [g])<\tilde{Y}(\Sph^3_+,\partial\Sph^3_+,[g_0])$, the (PS) sequence converges to a  solution to our problem and the proof is concluded.
\end{proof}


\end{document}